\documentclass{amsart}
\input xy
\usepackage[english]{babel}
\usepackage[latin1]{inputenc}
\usepackage{graphicx}
\usepackage{amsmath}
\usepackage{amssymb}
\usepackage{amscd}
\usepackage{color}
\usepackage{oldgerm}
\usepackage{amsfonts}
\usepackage{newlfont}
\usepackage{longtable}
\usepackage{multirow}

\usepackage[all]{xy}

\usepackage{upref}

\def\dim{\operatorname{dim}}
\def\Rad{\operatorname{Rad}}
\def\Ker{\operatorname{Ker}}
\def\End{\operatorname{End}}
\def\Der{\operatorname{Der}}
\def\ad{\operatorname{ad}}

\def\Im{\operatorname{Im}}
\def\lim{\opertaorname{lim}}

\def\Nil{\operatorname{Nil}}

\def\span{\operatorname{Span}}
\def\Mat{\operatorname{Mat}}
\newtheorem{teorema}{Theorem}[section]
\newtheorem{lema}[teorema]{Lema}
\newtheorem{proposicion}[teorema]{Proposition}
\newtheorem{definicion}[teorema]{Definition}
\newtheorem{corolario}[teorema]{Corollary}
\newtheorem{ejemplo}[teorema]{Example}
\newtheorem{rmk}[teorema]{Remark}
\def\F{\Bbb F}

\def\N{\Bbb N}

\def\G{\mathcal G}

\def\g{\frak g}

\def\h{\frak h}

\def\G{\frak{G}}

\begin{document}
\title{Invariant metrics on central extensions of quadratic Lie algebras}
\author
{R. Garc\'{\i}a-Delgado$^{(a)(b)}$, G. Salgado$^{(c)}$, O.A. S\'anchez-Valenzuela$^{(d)}$}
\address{(a)-(c) Facultad de Ciencias, Universidad Aut\'onoma de San Luis Potos\'{i}; SLP, M\'exico}
\address{(b)-(d) Centro de Investigaci\'on en Matem\'aticas, A.C., Unidad M\'erida; Yuc, M\'exico}
% \currentaddress{(b) Centro de Investigaci\'on en Matem\'aticas, A.C., Unidad M\'erida; Yucat\'an, M\'exico}
\address{(a)(b)}
\email{rosendo.garcia@cimat.mx, rosendo.garciadelgado@alumnos.uaslp.edu.mx}
\address{(c)}
\email{gsalgado@fciencias.uaslp.mx, gil.salgado@gmail.com}
\address{(d)}
\email{adolfo@cimat.mx}
\keywords {Quadratic Lie algebra; Central extension; Invariant metric; Skew-symmetric derivation}
\subjclass{
Primary:
17Bxx, 17B05, 17B60,
Secondary:
15A63, 17B30, 
}
%\thanks{\noindent \textup{2010} \textit{Mathematics Subject Classification}}

%\date{}
%\keywords
\maketitle

%%%%%%%%%%%%%%%%%%%%%%%%%%%%%%%%%%%%%%%%%%%%%%%%%%%%%%%%%%%%%%%%%%%%
\begin{abstract}
\noindent 
A quadratic Lie algebra is a Lie algebra endowed with a symmetric, invariant and non degenerate bilinear form; such
a bilinear form is called an invariant metric. The aim of this work is to describe the general structure of those
central extensions of quadratic Lie algebras which in turn have invariant metrics. The structure is such that
the central extensions can be described algebraically in terms of the original quadratic Lie algebra, and geometrically in terms of 
the direct sum decompositions that the invariant metrics involved give rise to.
\end{abstract}
%%%%%%%%%%%%%%%%%%%%%%%%%%%%%%%%%%%%%%%%%%%%%%%%%%%%%%%%%%%%%%%%%%%%%%%%%%%

\section*{Introduction}

\noindent A Lie algebra has invariant orthogonal geometry if it comes
equipped with a symmetric, bilinear, non-degenerate form 
satisfying, $B( x,[y,z])=B( [x,y],z)$. 
Such a Lie algebra is called \emph{quadratic\/.}
It is \emph{indecomposable} if cannot
be written as an orthogonal direct sum of two
non-trivial quadratic ideals.

\smallskip
\noindent
In the early 80's, V. Kac posed in  \cite{Kac} two problems.
The first problem asks for extending a given quadratic Lie algebra,
to another which is also quadratic and has a one-dimensional minimal central ideal.
The second problem
asks for proving that any solvable indecomposable Lie algebra is such an extension;
actually, a codimension-2 extension
of a quadratic solvable Lie algebra.
Such extensions have been referred to in the literature as {\it double extensions\/} (see \cite{Bou} and \cite{Fig}). In fact, a double extension is a semidirect product of two Lie algebras, 
one of which is a {\it central extension\/} of a quadratic Lie algebra.

\smallskip
\noindent 
An important goal that motivated this work was to address the following inverse problem in general: given a quadratic Lie algebra $\G$ having a subspace $\g$ that admits itself a structure of a quadratic Lie algebra whose invariant metric is not the restriction of the invariant metric on $\G$, {\it is it true that $\G$ is an extension of $\g$\/?} We have obtained a specific answer to this question in {\bf Theorem} \ref{teorema_c1}.

\smallskip
\noindent 
There is another relationship between central extensions and quadratic Lie algebras. 
Consider a 2-step nilpotent quadratic Lie algebra $\mathfrak{G}$; that is, $[\mathfrak{G},[\mathfrak{G},\mathfrak{G}]]=\{0\}$. Let $\mathfrak{g}$ be a subspace complementary to $[\mathfrak{G},\mathfrak{G}]$ in $\mathfrak{G}$: 
Thus, $\mathfrak{G}=\mathfrak{g} \oplus [\mathfrak{G},\mathfrak{G}]$. We may consider $\mathfrak{g}$ as an Abelian Lie algebra, such that the projection $\pi_{\mathfrak{g}}:\mathfrak{G} \rightarrow \mathfrak{g}$ onto $\mathfrak{g}$, is a Lie algebra epimorphism with kernel $[\mathfrak{G},\mathfrak{G}]_{\mathfrak{G}}$. Since $\mathfrak{g}$ is Abelian, any symmetric, non-degenerate bilinear form defined on it, is an invariant metric. Therefore, the quadratic Lie algebra $\mathfrak{G}$ is a central extension of the Abelian quadratic Lie algebra $\mathfrak{g}$ (see \cite{Ovan} for other important structure theorems about 2-step nilpotent quadratic Lie algebras). One of the problems addressed in this work is to understand the structure of central extensions $\frak{G}$ of a {\it non-Abelian quadratic Lie algebra\/} $\frak{g}$, which are themselves quadratic. 
\smallskip

\noindent 
On the other hand, when we think of quadratic Lie algebras, semisimple Lie algebras appear as immediate examples.
Central extensions of semisimple Lie algebras can then be approached within Levi's decomposition theorem. Then, the corresponding short exact sequence $0\to\Rad{\frak{G}}\to\frak{G}\to\frak{g}\to 0$ always splits and in particular, 
the structure of a central extension $\frak{G}$ of $\frak{g}$ carrying an invariant metric can be easily understood in terms of orthogonal decomposition, as the metric coincides 
up to a scalar factor with the corresponding Cartan-Killing form on each of the simple summands of $\g$.
The question of what happens when $\frak{g}$ is quadratic but not necessarily semisimple is also addressed in this work. 
Thus, we aim to advance our understanding of the structure
of central extensions of quadratic Lie algebras in general.

\smallskip
\noindent 
It is worth noting that central extensions $\frak{h}$ of non-semisimple quadratic Lie algebras $\frak{g}$
play an important first-step role in our understanding of the general problem addressed in this work.
It has been proved in \cite{MCGS}, that a derivation $D$ defined
on a one-dimensional central extension $\frak{h}=\frak{g}\oplus\frak{c}$,
of a Lie algebra $\frak{g}$ admitting an invariant metric, yields an extension $\frak{G}=\frak{h}\oplus\Bbb F\,D$ which can be made into a quadratic Lie algebra whenever $D$ acts trivially on $\frak{c}$
(see {\bf Theorem 3.2} in \cite{MCGS}).

\smallskip
\noindent
This work is organized as follows. In {\bf \S1} we recall the basic definitions, set up the notation
and quote some classical results that will be used throughout this work. In {\bf \S2}, we explore
in detail the structure of the central extensions
of quadratic Lie algebras under the assumption that these extensions admit an invariant metric. We state a proposition that gives necessary and sufficient conditions to determine
whether or not the kernel of a central extension is non-degenerate
({\bf Proposition \ref{teorema de estructura 2}}). Then {\bf \S2} is closed with  
a proposition that reduces the general problem to the understanding of two extreme cases
corresponding to the following situations:
either the kernel of the extension is 
a non-degenerate ideal, or it is an isotropic ideal. The non-trivial case is when the kernel of the central extension is isotropic. For that, we then introduce in {\bf \S3}
a construction called {\it double central extension\/}  inspired 
in the {\it generalized double extensions\/} of Bajo and Benayadi (see \cite{Ben}).
A proposition is proved that gives necessary and sufficient conditions
for a central extension of a quadratic Lie algebra, admitting an invariant metric with isotropic kernel, to be such
a double central extension ({\bf Propostion 3.1}).
Then, {\bf \S3} is closed with a theorem 
that describes the relationship between the algebraic and the geometric
structures of a central extension ({\bf Theorem \ref{teorema_c1}}). 
Finally, we prove that any 
quadratic Lie algebra obtained as a
central extension 
by a $1$-dimensional or by a
$2$-dimensional kernel, always splits ({\bf Proposition \ref{teorema c2}}), and we provide 
examples of extensions by $r$-dimensional ideals ($r\ge 3$),
that do not split (see {\bf Example \ref{ejemplo-correccion}}).

%%%%%%%%%%%%%%%%%%%%%%%%%%%%%%%%%%\section{Main results}

\section{Basic definitions}

\noindent All the Lie algebras considered on this paper are finite dimensional over an algebraically closed 
field $\mathbb{F}$ of zero characteristic.

\begin{definicion}{\rm
A Lie algebra over $\mathbb{F}$, is a pair consisting of an $\Bbb F$-vector space $\mathfrak{g}$,
and a skew-symmetric bilinear map 
$[\,\cdot\,,\,\cdot\,]:\mathfrak{g} \times \mathfrak{g} \rightarrow \mathfrak{g}$, 
satisfying the {\bf Jacobi identity}:
$$
[x,[y,z]]+[y,[z,x]]+[z,[x,y]]=0,\ \ \forall \, x,y,z \in \mathfrak{g}.
$$
The element $[x,y]\in\mathfrak{g}$ is referred to as 
{\it the Lie bracket\/} of $x$ and $y$ in $\mathfrak{g}$.
Most of the time we shall write $[\,\cdot\,,\,\cdot\,]_{\frak{g}}$,
since different Lie algebras will appear.
}
\end{definicion}

\noindent
We shall adhere ourselves to the standard convention of writing,
$$
C(\frak{g})=\{z\in\frak{g}\mid [z,x]_{\frak{g}}=0,\ \text{for all}\ x\in\frak{g}\,\},
$$
for {\bf the center} of the Lie algebra $(\mathfrak{g},[\cdot,\cdot]_{\frak{g}}$).
\smallskip

\noindent The {\bf derived central series} of a Lie algebra $(\g,[\cdot,\cdot]_{\g})$ is defined by,
$$
C_1(\mathfrak{g})=C(\mathfrak{g}); \qquad C_k(\mathfrak{g})=\pi_{k-1}^{-1}C(\g/C_{k-1}(\g)), \quad k\geq 1,
$$
where $\pi_{k-1}:\g \to \g/C_{k-1}(\g)$ is the corresponding canonical projection.
\smallskip

\noindent The {\bf descending central series} of a Lie algebra $(\g,[\cdot,\cdot]_{\g})$ is defined by,
$$
\mathfrak{g}^0=\mathfrak{g}; \qquad \mathfrak{g}^\ell=[\mathfrak{g},\mathfrak{g}^{\ell-1}], \quad \ell\geq 1.
$$
The Lie algebra $(\g,[\cdot,\cdot]_{\g})$ is {\bf nilpotent} if $\g^n=\{0\}$ for some $n\in\mathbb N$. If
$\g^n=\{0\}$, with $\g^{n-1}\neq \{0\}$, one says that $(\g,[\cdot,\cdot]_{\g})$ is an
{\bf $n$-step} nilpotent Lie algebra.

\begin{definicion}{\rm
Let $(\mathfrak{g},[\cdot,\cdot]_{\mathfrak{g}})$ be a Lie algebra over $\mathbb{F}$. An \textbf{invariant metric} in $(\mathfrak{g},[\cdot,\cdot]_{\mathfrak{g}})$ is a non-degenerate, symmetric, bilinear form 
 $B_{\mathfrak{g}}:\frak{g}\times\frak{g}\to\Bbb F$,
 satisfying $B_{\mathfrak{g}}(x,[y,z]_{\mathfrak{g}})=B_{\mathfrak{g}}([x,y]_{\mathfrak{g}},z)$ for all $x,y,z \in \mathfrak{g}$. One usually refers to this property as {\bf the $\frak{g}$-invariance of} $B_{\mathfrak{g}}$,
 and say that $B_{\mathfrak{g}}$ {\bf is $\frak{g}$-invariant} or simply {\bf invariant} for short. 
 A \textbf{quadratic Lie algebra} is a triple $(\mathfrak{g},[\cdot,\cdot],B_{\mathfrak{g}})$ where $B_{\mathfrak{g}}$ is an invariant metric in $(\mathfrak{g},[\cdot,\cdot]_{\mathfrak{g}})$. 
 We shall occassionally say that $B_{\mathfrak{g}}$ is an invariant metric in $\frak{g}$.}
\end{definicion}

\noindent Two quadratic Lie algebras $(\g,[\cdot,\cdot]_\g,B_{\g})$ and $(\h,[\cdot,\cdot]_{\h},B_{\h})$ are {\bf isometric} if there exists an isomorphism of Lie algebras $\psi:(\g,[\cdot,\cdot]_{\g}) \to (\h,[\cdot,\cdot]_{\h})$, such that $B_{\g}(x,y)=B_{\h}(\psi(x),\psi(y))$, for all $x,y \in \g$. The linear map $\psi:\g \to \h$ is called an {\bf isometry}.
\smallskip

\noindent
We recall that for any non-degenerate and bilinear form $B_{\frak{g}}$ defined on a vector space
$\frak{g}$, we may define a linear map  $B_{\mathfrak{g}}^{\flat}:\mathfrak{g} \to \frak{g}^\ast$, 
into the dual vector space $\frak{g}^\ast$, by means of $B_{\mathfrak{g}}^{\flat}(x)\,(z)=
B_{\mathfrak{g}}(x,z)$ for all $z \in \mathfrak{g}$, and each $x\in\frak{g}$.
When $\frak{g}$ is a Lie algebra and $B_{\mathfrak{g}}$ is an invariant metric on it,
$B_{\mathfrak{g}}^{\flat}:\mathfrak{g} \to \frak{g}^\ast$ is actually an
isomorphism that intertwines the adjoint action of $\frak{g}$ in $\frak{g}$ with the
coadjoint action of $\frak{g}$ in $\frak{g}^\ast$; that is,
$$
B_{\mathfrak{g}}^{\flat}(\ad(x)(y))\,(z) = (\ad(x)^\ast B_{\mathfrak{g}}^{\flat}(y))\,(z),
$$ 
where $\ad(x)^\ast(\theta)=-\theta\circ\ad(x)$, for any $\theta\in\frak{g}^\ast$.

\begin{definicion}\label{adequiv}
%%%   {\rm
For any Lie algebra $(\mathfrak{g},[\cdot,\cdot]_{\mathfrak{g}})$, we shall denote by
$\Gamma(\mathfrak{g})$ the vector subspace of $\End(\mathfrak{g})$ consisting of those
linear maps that commute with the adjoint action; that is,
$$
\Gamma(\mathfrak{g})=\{T \in \End(\mathfrak{g})\,\mid\,T([x,y]_{\mathfrak{g}})=[T(x),y]_{\mathfrak{g}},
\,\forall x,y \in \mathfrak{g} \}.
$$
Using the skew-symmetry of the Lie bracket, it is easy to see that
$T([x,y]_{\mathfrak{g}})=[T(x),y]_{\mathfrak{g}}$ also implies that
$T([x,y]_{\mathfrak{g}})=[x,T(y)]_{\mathfrak{g}}$.
\end{definicion} 

\noindent
Following \cite{Zhu} we shall refer to any linear map $T\in \Gamma(\mathfrak{g})$
as a {\bf centroid}.
In particular, for quadratic Lie algebras, the self-adjoint centroids with respect to a given invariant metric will play an important role in this work. Thus, we introduce a special notation for them; namely, if $B_{\mathfrak{g}}:\mathfrak{g}\times\mathfrak{g}\to\mathbb{F}$ is an invariant metric on $\mathfrak{g}$, we define the following subspace of $\Gamma(\mathfrak{g})$:
$$
\Gamma_{B_{\mathfrak{g}}}(\mathfrak{g})=\{T \in
\Gamma(\mathfrak{g})\,\mid\,B_{\mathfrak{g}}(Tx,y)=B_{\mathfrak{g}}(x,Ty),\,\forall x,y \in
\mathfrak{g} \}.
$$
Moreover, we shall denote by $\Gamma_{B_{\mathfrak{g}}}^0(\mathfrak{g})$ the set of
{\it invertible\/} self-adjoint centroids in $\mathfrak{g}$; that is,
$$
\Gamma_{B_{\mathfrak{g}}}^0(\mathfrak{g})=\Gamma_{B_{\mathfrak{g}}}(\mathfrak{g}) \cap \operatorname{GL}(\mathfrak{g}).
$$
{\bf Note.} Given an invariant metric $B_{\mathfrak{g}}$ in $\mathfrak{g}$, we shall refer to the self-adjoint linear operators in $\mathfrak{g}$
as $B_{\mathfrak{g}}$-\textbf{symmetric maps}. Furthermore, fixing a given $B_{\mathfrak{g}}$, a bilinear form $\bar{B}$ in 
$\mathfrak{g}$ yields an invariant metric in $(\mathfrak{g},[\cdot,\cdot]_{\mathfrak{g}})$ if and only if there is a 
$T \in \Gamma_{B_{\mathfrak{g}}}^0(\mathfrak{g})$ such that $\bar{B}(x,y)=B_{\mathfrak{g}}(T(x),y)$ for all $x,y \in \mathfrak{g}$.
%%%   }

\subsection{Central extensions of Lie Algebras}

\noindent Central extensions on Lie algebras are defined as follows:

\begin{definicion}\label{definicion_cociclos}{\rm
Let $(\mathfrak{g},[\cdot,\cdot]_{\mathfrak{g}})$ be a Lie algebra and let $V$ be a vector space. The function $\theta:\mathfrak{g} \times \mathfrak{g} \rightarrow
V$ is a \textbf{2-cocycle of $\mathfrak{g}$ with values on $V$} if
\begin{itemize}

\item $\theta$ is bilinear and skew-symmetric,

\item For all $x,y,z \in \mathfrak{g}$,  
$$
\theta(x,[y,z]_{\mathfrak{g}})+\theta(y,[z,x]_{\mathfrak{g}})+\theta(z,[x,y]_{\mathfrak{g}})=0.
$$

\end{itemize}
The vector space of $2$-cocycles of $(\mathfrak{g},[\cdot,\cdot]_{\mathfrak{g}})$ with values on $V$ is denoted by $Z^2(\mathfrak{g},V)$. A \textbf{$2$-coboundary}  is a $2$-cocycle $\theta$ for which there exists a linear map $\tau:\mathfrak{g} \rightarrow V$, satisfying
$\theta(x,y)=\tau([x,y]_{\mathfrak{g}})$, for all $x,y \in \mathfrak{g}$. The vector space of the $2$-coboundaries is usually denoted by $B^2(\mathfrak{g},V)$.}
\end{definicion}

\begin{definicion}
{\rm
Let $(\mathfrak{g},[\cdot,\cdot]_{\mathfrak{g}})$ be a Lie algebra, let $V$ be a vector space and let $\theta \in Z^2(\mathfrak{g},V)$ be a $2-$cocycle. In the vector space $\mathfrak{G}=\mathfrak{g} \oplus V$, we define a Lie bracket $[\cdot,\cdot]_{\mathfrak{G}}$ by, 
\begin{equation}\label{tc1}
[x+u,y+v]_{\mathfrak{G}}=[x,y]_{\mathfrak{g}}+\theta(x,y),
\end{equation}
for all $x,y \in \mathfrak{g}$, and $u,v \in V$. We observe that $V$ is an \emph{ideal} of $(\mathfrak{G},[\cdot,\cdot]_{\mathfrak{G}})$ contained in the center of $\mathfrak{G}$.}
\end{definicion}
\noindent If $\iota:V \rightarrow \mathfrak{G}$ is the inclusion map and $\pi_{\mathfrak{g}}:\mathfrak{G} \rightarrow \mathfrak{g}$ is the projection onto $\mathfrak{g}$, then there exists a short exact sequence of Lie algebras:
\begin{equation}\label{m}
\begin{array}{rcccccccl}
\,& \, & \, & \iota & \,&\pi_{\mathfrak{g}} &\,&\,& \\
0 & \rightarrow & V & \rightarrow & \mathfrak{G} & \rightarrow & \mathfrak{g} & \rightarrow & 0.
\end{array}
\end{equation}
The Lie algebra $(\mathfrak{G},[\cdot,\cdot]_{\mathfrak{G}})$ is
\textbf{the central extension of $(\mathfrak{g},[\cdot,\cdot]_{\mathfrak{g}})$ by $V$
associated to the $2$-cocylce $\theta$}.
The ideal $V$ is \textbf{the kernel of the extension},
and $\dim_{\mathbb{F}}V$ is \textbf{the dimension of the extension}.

\begin{definicion}{\rm 
The short exact sequence \eqref{m} \textbf{splits} if there exists a Lie algebra morphism, $\alpha:(\mathfrak{g},[\cdot,\cdot]_{\mathfrak{g}}) \rightarrow (\mathfrak{G},[\cdot,\cdot]_{\mathfrak{G}})$, such that, $\pi_{\mathfrak{g}} \circ \alpha=\operatorname{Id}_{\mathfrak{g}}$.}
\end{definicion}

\noindent In a central extension $(\mathfrak{G},[\cdot,\cdot]_{\mathfrak{G}})$, it is not necessarily true that
$\mathfrak{g}$ is an ideal. It is, however, under the conditions given by the following:

\begin{proposicion}[\cite{Che}]\label{escision}
The following statements are equivalent:
\begin{enumerate}

\item[(i)] The short exact sequence \eqref{m}, splits.

\item[(ii)] There exists a linear map $\tau:\mathfrak{g} \rightarrow V$, such that, $\theta(x,y)=\tau([x,y]_{\g})$, for all $x,y \in g$, 
where $\theta$ is the $2$-cocycle of the extension.

\item[(iii)] There exists a Lie algebra isomorphism $\Psi:(\mathfrak{G},[\cdot,\cdot]_{\mathfrak{G}})\to(\mathfrak{g},[\cdot,\cdot]_{\mathfrak{g}}) \oplus V$, such that $\Psi|_V=\operatorname{Id}_V$.

\item[(iv)] There exists an ideal $\mathfrak{h}$ of $(\mathfrak{G},[\cdot,\cdot]_{\mathfrak{G}})$, such that, $(\mathfrak{G},[\cdot,\cdot]_{\mathfrak{G}})=(\mathfrak{h},[\cdot,\cdot]_{\mathfrak{G}}|_{\mathfrak{h} \times \mathfrak{h}}) \oplus V$, where $(\mathfrak{h},[\cdot,\cdot]_{\mathfrak{G}}|_{\mathfrak{h} \times \mathfrak{h}})$ is isomorphic to $(\mathfrak{g},[\cdot,\cdot]_{\mathfrak{g}})$.

\end{enumerate}

\end{proposicion}
  
\subsection{Skew-symmetric derivations of Quadratic Lie algebras}

\begin{definicion}{\rm
Let $(\mathfrak{g},[\cdot,\cdot]_{\mathfrak{g}})$ be a Lie algebra. A linear map $D \in \End_{\mathbb{F}}(\mathfrak{g})$ is a \textbf{derivation} of $(\mathfrak{g},[\cdot,\cdot]_{\mathfrak{g}})$,
if it satisfies the \textbf{Leibniz rule}:
$$
D([x,y]_{\mathfrak{g}})=[D(x),y]_{\mathfrak{g}}+[x,D(y)]_{\mathfrak{g}},\,\,\,\forall x,y \in \mathfrak{g}.
$$
The vector space of all the derivations of a Lie algebra $(\mathfrak{g},[\cdot,\cdot]_{\mathfrak{g}})$, is denoted 
by $\Der \mathfrak{g}$.
}
\end{definicion}

\begin{definicion}\label{defincion1}{\rm
Let $(\mathfrak{g},[\cdot,\cdot]_{\mathfrak{g}},B_{\mathfrak{g}})$ be a quadratic Lie algebra.
We shall say that a linear map $D \in \End_{\mathbb{F}}(\mathfrak{g})$
is $B_{\g}$-skew-symmetric if it satisfies
$B_{\mathfrak{g}}(D(x),y)=-B_{\mathfrak{g}}(x,D(y))$, for all $x,y \in \mathfrak{g}$. 
The vector space of $B_{\g}$-skew-symmetric linear maps is denoted by $\frak{o}(B_{\g})$.
}
\end{definicion}

\noindent 
There is a relationship between the vector space of $2$-cocycles
and the vector space of the $B_{\g}$-skew-symmetric derivations of a quadratic Lie algebra which we now describe. Let $(\mathfrak{g},[\cdot,\cdot]_{\mathfrak{g}},B_{\mathfrak{g}})$ be a quadratic Lie algebra and let $V$ be an $r$-dimensional vector space. Let $D_1,\ldots,D_r \in \Der \mathfrak{g}$ be $B_{\mathfrak{g}}$-skew-symmetric derivations and let
$\{v_1,\ldots,v_r \}$ be a basis of $V$. Let
$\theta:\mathfrak{g} \times \mathfrak{g} \rightarrow V$ be the bilinear map defined by
$\theta(x,y)=\overset{r}{\underset{i=1}{\sum}}B_{\mathfrak{g}}(D_i(x),y)v_i$, for all $x,y \in \mathfrak{g}$.
Then, $\theta$ is a $2$-cocycle of $(\mathfrak{g},[\cdot,\cdot]_{\mathfrak{g}})$ with values on $V$. 

\smallskip
\noindent
Conversely, let $\theta:\mathfrak{g} \times \mathfrak{g} \rightarrow V$ be a
$2$-cocycle and let $\{v_1,\ldots,v_r \}$ be a basis of $V$. For each pair of elements $x,y \in \mathfrak{g}$, there are scalars $\alpha_1(x,y),\ldots,\alpha_r(x,y) \in \mathbb{F}$ such that $\theta(x,y)=\alpha_1(x,y)v_1+\cdots+\alpha_r(x,y)v_r$. For $i=1,\ldots,r$, we define the linear maps
$\alpha_i(x):\mathfrak{g} \rightarrow \mathbb{F}$, $y \mapsto \alpha_i(x,y)$. Then, there is an element $D_i(x) \in \mathfrak{g}$, depending only on $x$, such that $\alpha_i(x)=B_{\mathfrak{g}}^{\flat}(D_i(x))$, for $i=1,\ldots,r$. 
The $2-$cocycle properties imply that $D_1,\ldots,D_r$ are $B_{\mathfrak{g}}$-skew-symmetric derivations of $(\mathfrak{g},[\cdot,\cdot]_{\mathfrak{g}},B_{\mathfrak{g}})$.
Thus,
\begin{proposicion}\label{proposicion 1}
Let $(\mathfrak{g},[\cdot,\cdot]_{\mathfrak{g}},B_{\mathfrak{g}})$ be a quadratic Lie algebra and 
let $V$ be an $r$-dimensional vector space. Then, the space of $2$-cocycles $Z^{2}(\mathfrak{g},V)$ 
is isomorphic to the space of $r$-tuples of $B_{\mathfrak{g}}$-skew-symmetric derivations of $(\mathfrak{g},[\cdot,\cdot]_{\mathfrak{g}},B_{\mathfrak{g}})$.
Moreover, a $2$-cocycle $\theta\in Z^{2}(\mathfrak{g},V)$ is a coboundary if and only if 
there are $a_1,\ldots,a_r \in \mathfrak{g}$ such that $D_i=\ad_{\mathfrak{g}}(a_i)$ ($1\le i\le r$). 
\end{proposicion}

\section{Central extensions of quadratic Lie algebras}

\noindent
\begin{rmk}\label{1} 
From this point on, we shall assume that $(\mathfrak{G},[\cdot,\cdot]_{\mathfrak{G}})$ is a central extension of a given quadratic Lie algebra $(\mathfrak{g},[\cdot,\cdot]_{\mathfrak{g}})$ by an $r$-dimensional vector space $V$,
and that $(\mathfrak{G},[\cdot,\cdot]_{\mathfrak{G}})$ admits an invariant metric. These are the hypotheses of
the following two technical results ({\bf Lemmas \ref{lema 1}} and {\bf \ref{lema 2}}):
\end{rmk}

\begin{lema}\label{lema 1}
There exist linear maps $h:\mathfrak{G} \rightarrow
\mathfrak{g}$ and $k:\mathfrak{g} \rightarrow \mathfrak{G}$ such that
\begin{itemize}

\item[(i)] $h \circ k=\operatorname{Id}_{\mathfrak{g}}$.

\item[(ii)] $B_{\mathfrak{G}}(x+v,y)=B_{\mathfrak{g}}(h(x+v),y)$ and $B_{\mathfrak{g}}(x,y)=B_{\mathfrak{G}}(k(x),y+v)$ for all $x,y \in \mathfrak{g}$ and $v \in V$.

\item[(iii)] $\Ker h=\mathfrak{g}^{\perp}$ and\ \ $\Im k=V^{\perp}$.

\item[(iv)] $\mathfrak{G}=\Ker h \oplus \Im k$.

\item[(v)] $k([x,y]_{\mathfrak{g}})=[k(x),y]_{\mathfrak{G}}$, for all $x,y \in \mathfrak{g}$.

\end{itemize}
\end{lema}
\begin{proof}
  
\noindent Let $\iota: \mathfrak{g} \rightarrow
\mathfrak{G}$ be the inclusion map and let
$\iota^*:\mathfrak{G}^{*} \rightarrow \mathfrak{g}^{*}$ be its dual map.
Since $\iota^{*}(B_{\mathfrak{G}}^{\flat}(\mathfrak{G})) \subseteq B_{\mathfrak{g}}^{\flat}(\mathfrak{g})=\mathfrak{g}^{*}$,
it follows that for each $x \in \mathfrak{G}$, 
there exists a unique $h(x) \in \mathfrak{g}$, such that, $\iota^*(B_{\mathfrak{G}}^{\flat}(x))=B_{\mathfrak{g}}^{\flat}(h(x))$. It is 
immediate to verify that the assignment $h:\mathfrak{G} \rightarrow\mathfrak{g}$ is linear. Since $B_{\mathfrak{g}}$ is non-degenerate,
it also follows that $\Ker h=\mathfrak{g}^{\perp}$. 

\smallskip
\noindent Let $\pi_{\mathfrak{g}}:\mathfrak{G} \rightarrow \mathfrak{g}$ be the linear projection onto 
$\mathfrak{g}$ and let $\pi_{\mathfrak{g}}^*:\mathfrak{g}^{*} \rightarrow \mathfrak{G}^{*}$ be its dual map. 
Using a similar argument for $\pi_{\mathfrak{g}}$ as we used before for $\iota$, we conclude that
there exists a linear map $k:\mathfrak{g} \rightarrow \mathfrak{G}$, such that 
for each $x \in \mathfrak{g}$,
$\pi_{\mathfrak{g}}^*(B_{\mathfrak{g}}^{\flat}(x))=B_{\mathfrak{G}}^{\flat}(k(x))$. 
Since $\pi_{\mathfrak{g}}^*(B_{\mathfrak{g}}^{\flat}(x))\in\operatorname{Ann}(V)=\{\alpha \in \G^{*}\,|\,\alpha(v)=0,\,\forall v \in V\}$,
it follows that $\Im k \subseteq V^{\perp}$. 
\smallskip

\noindent Now let $x,y \in \mathfrak{g}$. Then $B_{\mathfrak{g}}(x,y)=B_{\mathfrak{G}}(k(x),y)=B_{\mathfrak{g}}(h(k(x)),y)$, and
therefore, $h \circ k=\operatorname{Id}_{\mathfrak{g}}$. In particular, 
$h:\mathfrak{G} \rightarrow \mathfrak{g}$ is surjective and $k:\mathfrak{g} \rightarrow \mathfrak{G}$ is injective. 
Now $h \circ k=\operatorname{Id}_{\mathfrak{g}}$ implies that
$\mathfrak{G}=\Ker h \oplus \Im k$.
We may now prove that  $V^{\perp} \subseteq \Im k$. 
Take $y\in V^{\perp}$ and write $y=x+k(z)$ for some $x\in\operatorname{Ker} h$ and $z\in\mathfrak{g}$.
Clearly $B_{\mathfrak{G}}(x+k(z),v)=0$, for any $v\in V$. But $\Im k \subseteq V^{\perp}$
implies $B_{\mathfrak{G}}(x,v)=0$ for the given $x\in\operatorname{Ker} h=\mathfrak{g}^\perp$, and any $v\in V$.
Since $B_{\mathfrak{G}}$ is non-degenerate, this makes $x=0$, and hence $y=k(z)$.

\smallskip
\noindent
Since, $V^{\perp}=\Im k$, it follows that
$\Im k$ is an ideal of $(\mathfrak{G},[\cdot,\cdot]_{\mathfrak{G}})$. So, for each pair $x,y
\in \mathfrak{g}$, there exists a unique element $z \in \mathfrak{g}$ such that $[k(x),y]_{\mathfrak{G}}=k(z)$. 
Using the invariance and non-degeneracy of both $B_{\mathfrak{g}}$ and $B_{\mathfrak{G}}$,
together with the fact that $h\circ k=\operatorname{Id}_{\mathfrak{g}}$,
one concludes that $z=[x,y]_{\mathfrak{g}}$.
Indeed, for any $x^\prime\in\mathfrak{g}$ we have, on the one hand,
$$
\aligned
B_{\mathfrak{G}}([k(x),y]_{\mathfrak{G}},x^\prime)&=B_{\mathfrak{G}}(k(x),[y,x^\prime]_{\mathfrak{G}})=B_{\mathfrak{G}}(k(x),[y,x^\prime]_{\mathfrak{g}})
\\
&=B_{\mathfrak{g}}(h(k(x)),[y,x^\prime]_{\mathfrak{g}})=B_{\mathfrak{g}}([x,y]_{\mathfrak{g}},x^\prime).
\endaligned
$$
On the other hand, $B_{\mathfrak{G}}([k(x),y]_{\mathfrak{G}},x^\prime)=
B_{\mathfrak{G}}(k(z),x^\prime)=B_{\mathfrak{g}}(h(k(z)),x^\prime)=B_{\mathfrak{g}}(z,x^\prime)$.
Therefore $z=[x,y]_{\mathfrak{g}}$ as claimed, and
$k([x,y]_{\mathfrak{g}})=[k(x),y]_{\mathfrak{G}}$ for all $x,y \in \mathfrak{g}$.
\end{proof}

\begin{rmk}\label{2}The linear maps $h:\frak{G}\to\frak{g}$ and $k:\frak{g}\to\frak{G}$ 
can be easily described in terms of the maps $B_{\frak{G}}^\sharp:\frak{G}^\ast\to\frak{G}$
and $B_{\frak{g}}^\sharp:\frak{g}^\ast\to\frak{g}$ which are inverses to
$B_{\frak{G}}^\flat:\frak{G}^\ast\to\frak{G}$
and $B_{\frak{g}}^\flat:\frak{g}^\ast\to\frak{g}$, respectively.
In general, for any non-degenerate, bilinear, symmetric form $B_{\frak{g}}:\frak{g}\times \frak{g}\to\Bbb F$,
the map $B_{\frak{g}}^\sharp:\frak{g}^\ast\to\frak{g}$ is defined in such a way that for any $\theta\in\frak{g}^\ast$,
$B_{\frak{g}}^\sharp(\theta)$ is the unique vector in $\frak{g}$ that satisfies the identity,
$$
 B_{\frak{g}}(B_{\frak{g}}^\sharp(\theta),x)=\theta(x),\quad\forall\,x\in\frak{g}.
$$
Then, 
$$
h = B_{\frak{g}}^\sharp\circ\,\iota^\ast\circ\,B_{\frak{G}}^\flat,\qquad\text{and}\qquad
k = B_{\frak{G}}^\sharp\circ\,\pi_{\frak{g}}^\ast\circ\,B_{\frak{g}}^\flat,
$$
which are equivalent forms of the properties, 
$\pi_{\mathfrak{g}}^*(B_{\mathfrak{g}}^{\flat}(x))=B_{\mathfrak{G}}^{\flat}(k(x))$
and 
$\iota^*(B_{\mathfrak{G}}^{\flat}(x))=B_{\mathfrak{g}}^{\flat}(h(x))$,
respectively, from the statement of the Lemma above. On the other hand, observe that the existence of the linear maps $h:\G \to \g$ and $k:\g \to \G$ is independent of the Lie structure of $(\g,[\cdot,\cdot]_{\g})$ and the central extension $(\G,[\cdot,\cdot]_{\G})$ and their existence depends only of the non-degenerancy and symmetric properties of the bilinear forms $B_{\g}$ and $B_{\G}$.
\end{rmk}

\noindent  
Now, let $\theta:\mathfrak{g} \times \mathfrak{g} \rightarrow V$ be the $2$-cocycle associated to 
a central extension $\frak{G}$ of $\frak{g}$ by $V$. 
Each basis $\{v_1,\ldots,v_r \}$ of $V$, gives rise to $r$, $B_{\mathfrak{g}}$-skew-symmetric derivations 
$D_1,\ldots D_r\in\Der\mathfrak{g}$, through
$\theta(x,y)=\overset{r}{\underset{i=1}{\sum}}B_{\mathfrak{g}}(D_i(x),y)v_i$, for all $x,y \in \mathfrak{g}$
(see {\bf Propostition \ref{proposicion 1}}).
Thus, the Lie bracket $[\cdot ,\cdot]_{\mathfrak{G}}$, can be written as:
\begin{equation}\label{corchete-ext}
[x+u,y+v]_{\mathfrak{G}}=[x,y]_{\mathfrak{g}}+\sum_{i=1}^rB_{\mathfrak{g}}(D_i(x),y)v_i,\,\,\forall x,y \in \mathfrak{g},\,\,\forall u,v \in V.
\end{equation}

\begin{lema}\label{lema 2} 
Let $\{v_1,\ldots, v_r\}$ be a basis of $V$
for which the $2$-cocycle $\theta:\mathfrak{g}\times\mathfrak{g} \to V$
yields the Lie bracket $[\,\cdot\,,\,\cdot\,]_{\frak{G}}$ given in \eqref{corchete-ext} 
for the corresponding central extension $\frak{G}$ of $\frak{g}$ by $V$.
There is a basis $\{a_i+w_i \mid\,a_i \in \mathfrak{g},\,w_i \in
V,\,\,1 \leq i \leq r \}$ of $\mathfrak{g}^{\perp}$ in $\mathfrak{G}$, a 
$B_{\mathfrak{g}}$-symmetric map $T \in \Gamma(\mathfrak{g})$ 
and a linear map
$\rho:\mathfrak{G} \rightarrow \Der \mathfrak{g} \cap \frak{o}(B_{\g})$, such that:
\begin{itemize}

\item[(i)] $B_{\mathfrak{G}}(a_i+w_i,v_j)=\delta_{ij}$.

\item[(ii)] $\Ker T \subseteq C(\mathfrak{g})$.

\item[(iii)] $ h([x,[y,z]_{\mathfrak{G}}]_{\mathfrak{G}})=[x,h([y,z]_{\mathfrak{G}})]_{\mathfrak{g}}$ for all $x \in \mathfrak{g}$ and for all $y,z \in \mathfrak{G}$.

\item[(iv)] $\rho([x,y]_{\mathfrak{g}})=\rho(x) \circ T \circ \rho(y)-\rho(y) \circ T \circ \rho(x)$ for all $x,y \in \mathfrak{g}$ and $\rho(a_i)=D_i$ for all $1 \leq i \leq r$.

\item[(v)] $\Ker D_i=\{x \in \mathfrak{g}\,|\,[a_i,x]_{\mathfrak{G}}=0\}$,
for all $1 \leq i \leq r$.

\item[(vi)] $x=T\circ h(x+v)+\overset{r}{\underset{i=1}{\sum}}B_{\mathfrak{G}}(x+v,v_i)a_i$, for all $x \in \mathfrak{g}$, and for all $v \in V$.

\item[(vii)] $T \circ \rho(x)=\rho(x) \circ T=\ad_{\mathfrak{g}}(x)$, for all $x \in \mathfrak{g}$. In particular, $T \circ D_i=D_i \circ T=\ad_{\mathfrak{g}}(a_i)$, for all $1 \leq i \leq r$.

\item[(viii)] $\mathfrak{g}^\ell \subseteq \Im T^\ell$ and $\Ker T^\ell \subseteq C_{\ell}(\mathfrak{g})$ for all $\ell \in \mathbb{N}$.

\item[(ix)] If $T \neq 0$ and $\Ker(D_1)=\cdots=\Ker(D_r)$, then each $B_{\g}$-skew-symmetric derivation $D_{i}$ is inner.

\end{itemize}

\end{lema}

\begin{proof}

\noindent (i) Since $\mathfrak{g}^\perp=\Ker h \subset\mathfrak{G}$, we may define a 
linear map $\varphi:\Ker h \rightarrow V^{*}$ through,
$\varphi(x)(v)=B_{\mathfrak{G}}(x,v)$ for all
$x \in \Ker h$ and $v \in V$. 
{\bf Lemma \ref{lema 1}.(iv)} says that $\operatorname{dim}_{\mathbb{F}} \Ker h=\operatorname{dim}_{\mathbb{F}} V^{*}$ and
since $B_{\mathfrak{G}}$ is non-degenerate, it follows that $\varphi$ is actually an isomorphism.
Therefore there exists a basis
$\{a_1+w_1,\ldots,a_r+w_r \}$ of $\mathfrak{g}^{\perp}$, such that
$B_{\mathfrak{G}}(a_i+w_i,v_j)=\varphi(a_i+w_i)(v_j)=\delta_{ij}$, as in the statement.
\medskip

\noindent 
(ii) Let us consider the linear map $k:\mathfrak{g} \rightarrow \mathfrak{G}$ of {\bf Lemma \ref{lema 1}}. There are linear maps,
$T:\mathfrak{g} \rightarrow \mathfrak{g}$ and $R:\mathfrak{g} \rightarrow V$, 
such that, $k(x)=T(x)+R(x)$, for all $x \in \mathfrak{g}$. It is obvious that $T=\pi_{\mathfrak{g}} \circ k$. 
By Lemma \ref{lema 1}.(v), it follows that 
$T([x,y]_{\mathfrak{g}})=[Tx,y]_{\mathfrak{g}}$,
for all $x,y \in \mathfrak{g}$.
On the other hand, by Lemma \ref{lema 1}.(ii), we get 
for any $x,y\in\mathfrak{g}$ that, $B_{\mathfrak{g}}(T(x),y)=B_{\mathfrak{G}}(T(x),k(y))=B_{\mathfrak{G}}(k(x),k(y))$. In the same way, $B_{\mathfrak{g}}(x,T(y))=B_{\mathfrak{G}}(k(x),T(y))=B_{\mathfrak{G}}(k(x),k(y))$. Therefore, $T$ is an $\ad_{\mathfrak{g}}$-equivariant map which is $B_{\mathfrak{g}}$-symmetric; {\it ie\/,} $T \in \Gamma_{B_{\mathfrak{g}}}(\mathfrak{g})$.

\smallskip
\noindent 
Now, let $x \in \Ker T$. Then, $k([x,\mathfrak{g}]_{\mathfrak{g}})=[k(x),\mathfrak{g}]_{\mathfrak{G}}=\{0\}$,
since $k(x)\in C(\mathfrak{G})$.
But $k$ is injective ({\bf Lemma \ref{lema 1}.(i)}), thus $x \in C(\mathfrak{g})$.
So, $\Ker T \subseteq C(\mathfrak{g})$. 

\medskip
\noindent 
(iii) The linear map $k:\mathfrak{g} \rightarrow \mathfrak{G}$ of {\bf Lemma \ref{lema 1}} has been seen to satisfy
$k([x,y]_{\mathfrak{g}})=[k(x),y]_{\mathfrak{G}}$ for all $x,y \in \mathfrak{g}$. It is natural to expect
a similar property to hold true for the linear map $h:\mathfrak{G} \rightarrow \mathfrak{g}$. 
This is precisely the property given in the statement. In order to prove it, 
we are going to use the Lie bracket expression \eqref{corchete-ext} together with
the invariance and non-degeneracy of both $B_{\mathfrak{g}}$ and $B_{\mathfrak{G}}$. 
Let $x \in \mathfrak{G}$ and let $y,z \in \mathfrak{g}$. Then, by {\bf Lemma \ref{lema 1}.(ii)} and \eqref{corchete-ext}, it follows:
$$
\aligned
B_{\mathfrak{g}}(h([x,y]_{\mathfrak{G}}),z) &
=B_{\mathfrak{G}}([x,y]_{\mathfrak{G}},z)=B_{\mathfrak{G}}(x,[y,z]_{\mathfrak{G}})
\\
&=B_{\mathfrak{G}}(x,[y,z]_{\mathfrak{g}})+\sum_{i=1}^rB_{\mathfrak{g}}(D_i(y),z)B_{\mathfrak{G}}(x,v_i)
\\
&=B_{\mathfrak{g}}(h(x),[y,z]_{\mathfrak{g}})+B_{\mathfrak{g}}\biggl(\sum_{i=1}^r B_{\mathfrak{G}}(x,v_i)D_i(y),z \biggr)
\\
&=B_{\mathfrak{g}}\biggl([h(x),y]_{\mathfrak{g}}+\sum_{i=1}^rB_{\mathfrak{G}}(x,v_i)D_i(y),z\biggr).
\endaligned
$$
Whence,
\begin{equation}\label{nuevo1}
h([x,y]_{\mathfrak{G}})=[h(x),y]_{\mathfrak{g}}+\overset{r}{\underset{i=1}{\sum}}B_{\mathfrak{G}}(x,v_i)D_i(y).
\end{equation}
\noindent 
We exchange the roles of $x$ and $y$ in this expression, and use the skew-symmetry
of the Lie brackets in $\mathfrak{g}$ and $\mathfrak{G}$ to conclude that,
\begin{equation}\label{nuevo2}
h([x,y]_{\mathfrak{G}})=[x,h(y)]_{\mathfrak{g}}-\sum_{i=1}^rB_{\mathfrak{G}}(y,v_i)D_i(x),\quad \forall x \in \mathfrak{g},\quad \forall y \in \mathfrak{G}.
\end{equation}

\noindent 
We may now rewrite this expression in terms of $y^\prime=
[y,z]_{\mathfrak{G}}\in\mathfrak{G}$ rather than $y\in\mathfrak{G}$, and then
use the invariance of $B_{\mathfrak{G}}$ together with 
the fact that the $v_i$'s lie in the center of $\mathfrak{G}$ to conclude that,
\begin{equation}\label{nuevo3}
h([x,[y,z]_{\mathfrak{G}}]_{\mathfrak{G}})=[x,h([y,z]_{\mathfrak{G}})]_{\mathfrak{g}},\,\,\forall x \in \mathfrak{g},\,\forall y,z \in \mathfrak{G}.
\end{equation}

\noindent 
(iv) 
The invariance of $B_{\mathfrak{G}}$ implies that 
$[\mathfrak{G},\mathfrak{G}]_{\mathfrak{G}} \subseteq V^{\perp}=\operatorname{Im}k$. 
Thus, for each pair $x,y \in\mathfrak{G}$, there exists an unique element
$[x,y]_{\Delta} \in \mathfrak{g}$ such that $[x,y]_{\mathfrak{G}}=k([x,y]_{\Delta})$. 
Lemma \ref{lema 1}.(i) says that $h\circ k=\operatorname{Id}_{\mathfrak{g}}$, so $[x,y]_{\Delta}=h([x,y]_{\mathfrak{G}})$, for all $x,y \in \mathfrak{G}$. 
Thus $[\cdot,\cdot]_{\Delta}$ is a bilinear, skew-symmetric map on $\mathfrak{G}$ with values in $\mathfrak{g}$ and $[V,\mathfrak{G}]_{\Delta}=0$.
By \eqref{nuevo1} and the statement (i) above, it follows that $[a_i+w_i,x]_{\Delta}=[a_i,x]_{\Delta}=D_i(x)$, for all $x \in \mathfrak{g}$ and for all $1 \leq i \leq r$. 
\smallskip

\noindent
Now, the skew-symmetric map $[\cdot,\cdot]_{\Delta}$ induces a linear map $\rho:\mathfrak{G} \rightarrow \Der \mathfrak{g} \cap \frak{o}(B_{\g})$. It is given by $\rho(x)(y)=[x,y]_{\Delta}=h([x,y]_{\mathfrak{G}})$, for all $x \in \mathfrak{G}$ and $y \in \mathfrak{g}$. Indeed, 
using the adjoint representation $\ad_{\mathfrak{G}}:\mathfrak{G}\rightarrow \Der \mathfrak{G}$,
and the expressions \eqref{nuevo1} and \eqref{nuevo3} above, we get,
$$
\aligned
\rho(x)([y,z]_{\mathfrak{g}})
&=h([x,[y,z]_{\mathfrak{g}}]_{\mathfrak{G}})=h([x,[y,z]_{\mathfrak{G}}]_{\mathfrak{G}})=h([[x,y]_{\mathfrak{G}},z]_{\mathfrak{G}})+h([y,[x,z]_{\mathfrak{G}}]_{\mathfrak{G}})
\\
& =[h([x,y]_{\mathfrak{G}}),z]_{\mathfrak{g}}+[y,h([x,z]_{\mathfrak{G}})]_{\mathfrak{g}}=[\rho(x)(y),z]_{\mathfrak{g}}+[y,\rho(x)(z)]_{\mathfrak{g}}.
\endaligned
$$
On the other hand, notice that, $B_{\mathfrak{g}}(\rho(x)(y),z)=B_{\mathfrak{G}}([x,y]_{\mathfrak{G}},z)=-B_{\mathfrak{G}}(y,[x,z]_{\mathfrak{G}})=-B_{\mathfrak{g}}(y,\rho(x)(z))$. Therefore, $\operatorname{Im}\rho$ is contained in $\Der(\g) \cap \frak{o}(B_{\mathfrak{g}})$.
In particular, we may take $D_i=\rho(a_i)$, for all $1 \leq i \leq r$. We shall now prove that, 
\begin{equation}\label{T-representacion}	
\rho([x,y]_{\mathfrak{g}})=\rho(x) \circ T \circ \rho(y)-\rho(y) \circ T \circ \rho(x),\,\,\forall x,y \in \mathfrak{g}.
\end{equation}
For that, let us observe that by {\bf Lemma \ref{lema 1}.(i)-(v)}, we have,
\begin{equation}\label{f1}
[x,T(y)]_{\Delta}=h([x,T(y)]_{\mathfrak{G}})=h([x,k(y)]_{\mathfrak{G}})=[x,y]_{\mathfrak{g}},\,\,\,\forall x,y \in \mathfrak{g}.
\end{equation}
In order to prove \eqref{T-representacion}, we also need the Leibniz rule for $\ad_{\mathfrak{G}}$ and expression \eqref{nuevo2}. Indeed, let $x,y,z \in \mathfrak{g}$. Then,
$$
\aligned
\rho([x,y]_{\mathfrak{g}})(z)
&=h([[x,y]_{\mathfrak{g}},z]_{\mathfrak{G}})=h([[x,y]_{\mathfrak{G}},z]_{\mathfrak{G}})=h([x,[y,z]_{\mathfrak{G}}]_{\mathfrak{G}})-h([y,[x,z]_{\mathfrak{G}}]_{\mathfrak{G}})
\\
&=[x,[y,z]_{\Delta}]_{\mathfrak{g}}-[y,[x,z]_{\Delta}]_{\mathfrak{g}}=[x,T([y,z]_{\Delta})]_{\Delta}-[y,T([x,z]_{\Delta})]_{\Delta}
\\
&=\rho(x) \circ T \circ \rho(y)(z)-\rho(x) \circ T \circ \rho(y)(z),
\endaligned
$$
and our claim follows.

\smallskip
\noindent 
It is not difficult to verify that the space of derivations $\Der \mathfrak{g}$, with the bracket $[D,D']_T=D \circ T \circ D'-D' \circ T \circ D$, is a Lie algebra. Then, $\rho|_{\mathfrak{g}}:\mathfrak{g} \rightarrow \Der \mathfrak{g} \cap \frak{o}(B_{\g})$ is a Lie algebra morphism between $(\mathfrak{g},[\cdot,\cdot]_{\mathfrak{g}})$ and $(\Der \mathfrak{g}
\cap \frak{o}(B_{\g}),[\cdot,\cdot]_T)$.

\smallskip
\noindent (v) From (iv), we get $[a_i,x]_{\mathfrak{G}}=k(\rho(a_i)(x))=k(D_i(x))$. 
Since $k$ is injective, it follows that $\Ker D_i=\{x \in \mathfrak{g}\,\mid\,[a_i,x]_{\mathfrak{G}}=0\}$, $1 \leq i \leq r$.

\medskip
\noindent 
(vi) Since $B_{\mathfrak{g}}$ is non-degenerate, there are elements $a'_1,\ldots,a'_r \in \mathfrak{g}$, such that, $R(x)=\overset{r}{\underset{i=1}{\sum}}B_{\mathfrak{g}}(a'_i,x)v_i$, for all $x \in \mathfrak{g}$. So, the linear map $k:\mathfrak{g} \rightarrow \mathfrak{G}$, can be written as:
\begin{equation}\label{expresion para k}
k(x)=T(x)+\sum_{i=1}^rB_{\mathfrak{g}}(a'_i,x)v_i,\,\,\forall x \in \mathfrak{g}.
\end{equation}
We now want to show that $a'_i=a_i$ for all $1\le i\le r$.
Note first that, for any pair $x,y\in\mathfrak{g}$,
$$
\aligned
B_{\mathfrak{g}}(x,y) & =B_{\mathfrak{G}}(k(x),y)=B_{\mathfrak{G}}(T(x),y)+B_{\mathfrak{G}}(R(x),y)
\\
&
=B_{\mathfrak{g}}(T(x),h(y))+\sum_{i=1}^rB_{\mathfrak{g}}(a'_i,x)B_{\mathfrak{G}}(v_i,y)
\\
&
=B_{\mathfrak{g}}\biggl(x,T\circ h(y)+\overset{r}{\underset{i=1}{\sum}}B_{\mathfrak{G}}(v_i,y)a'_i \biggr)
\endaligned
$$
Since $B_{\mathfrak{g}}$ is non-degenerate, it follows that $y=T\circ h(y)+\overset{r}{\underset{i=1}{\sum}}B_{\mathfrak{G}}(y,v_i)a'_i$, for all $y \in \mathfrak{g}$. Let $x \in \mathfrak{g}$ and $v \in V$. Using the {\bf Lemma \ref{lema 1}(i)-(ii)} and the fact that $B_{\mathfrak{G}}(k(x),v)=0$, we deduce that $T\circ h(v)=-\overset{r}{\underset{i=1}{\sum}}B_{\mathfrak{G}}(v_i,v)a'_i$, for any $v \in V$. So, we get,
\begin{equation}
\label{ex4}
x=T\circ h(x+v)+\sum_{i=1}^r B_{\mathfrak{G}}(x+v,v_i)a'_i,\quad \forall x \in \mathfrak{g},\,\forall v \in V.
\end{equation}
By letting $x=a_j$ and $v=w_j$, it follows that $a_j=a'_j$, for all $1 \leq j \leq r$.

\medskip
\noindent (vii) Let $x,y \in \mathfrak{g}$. Then, by \eqref{f1}, we get $\rho(x)(T(y))=[x,y]_{\mathfrak{g}}$. Thus $\rho(x) \circ T=\ad_{\mathfrak{g}}(x)$ for each $x \in \mathfrak{g}$; in particular 
$D_i \circ T=\ad_{\mathfrak{g}}(a_i)$, for $1 \leq i \leq r$. 
Using the expression \eqref{ex4} and the invariance of $B_{\mathfrak{G}}$, we get, $[x,y]_{\mathfrak{g}}=T(h[x,y]_{\mathfrak{G}})=T \circ \rho(x)(y)$, for all $x,y \in \mathfrak{g}$.
%$$
%[x,y]_{\mathfrak{g}}=T(h[x,y]_{\mathfrak{G}})=T \circ \rho(x)(y),\,\forall x,y \in \mathfrak{g}.
%$$
Therefore $\ad_{\mathfrak{g}}(x)=T \circ \rho(x)$ for all $x \in \mathfrak{g}$; in particular $T \circ D_i=D_i \circ T=\ad_{\mathfrak{g}}(a_i)$, for all $1 \leq i \leq r$.

\medskip
\noindent 
(viii) By (vii), we know that $[x,y]_{\mathfrak{g}}=T(h([x,y]_{\mathfrak{G}}))$, for all $x,y \in \mathfrak{g}$. Then
$$
[x,[y,z]_{\mathfrak{g}}]_{\mathfrak{g}}=[x,T\circ h([y,z]_{\mathfrak{G}})]_{\mathfrak{g}}=T([x,h([y,z]_{\mathfrak{g}})]_{\mathfrak{g}})
$$
$$
=T^2\circ h([x,h([y,z]_{\mathfrak{G}})]_{\mathfrak{G}}),\forall x,y,z \in \mathfrak{g}.
$$
Therefore $\mathfrak{g}^2 \subseteq \Im T^2$. Continuing in this way, we obtain,
$$
[x_1,[\,\,,\,\ldots[x_\ell,x_{\ell+1}]_{\mathfrak{g}}\ldots]_{\mathfrak{g}}=T^{\ell}\circ h([x_1,[\,\,,\ldots,h([x_\ell,x_{\ell+1}]_{\mathfrak{G}})\,\ldots]_{\mathfrak{G}}).
$$
for any $x_1,\ldots,x_{\ell+1} \in \mathfrak{g}$.
Thus, $\mathfrak{g}^\ell \subseteq \Im T^\ell$. 
Since $(\mathfrak{g}^\ell)^{\perp}=C_{\ell}(\mathfrak{g})$, and 
it follows that $\Ker T^\ell \subseteq C_{\ell}(\mathfrak{g})$, $\forall\,\ell \in \mathbb{N}$.
\medskip

\noindent (ix) Suppose $\Ker(D_1)=\cdots=\Ker(D_r)$. 
We know from
 {\bf (v)} that $D_{i}(a_i)=0$, for all $1 \le i \le r$; that is, $a_i \in \Ker(D_i)$ for all $1 \le i \le r$. The hypothesis
allow us to conculde that $D_j(a_i)=0$, for all $1 \le i,j \le r$. On the other hand, from {\bf (iv)}, for each $B_{\g}$-skew-symmetric derivation $D_i$, we have $\rho(a_i)(x)=D_i(x)$, and therefore,
$$
D_i(x)=h([a_i,x]_{\G})=h([a_i,x]_{\g})+\sum_{j=1}^rB_{\g}(D_j(a_i),x)h(v_j)=h([a_i,x]_{\g}),\quad \forall x \in \g.
$$
It follows from {\bf (v)} that, $\Ker(D_i)=C_{\g}(a_i)=\{x \in \g\,|\,[a_i,x]_{\g}=0\}$, for all $1 \le i \le r$. Let $x \in \g$ be an arbitrary element and let $\chi_T=\alpha_s x^s+\alpha_{s-1}x^{s-1}+...+\alpha_1 x+\alpha_0 \in \F[X]$ be the 
characteristic polynomial of $T$. Then,
$$
\aligned
0 & =\chi_T(T)(D_i(x)) 
\\
& =D_i(\chi_T(T)(x))=D_i((\alpha_s
T^s+\alpha_{s-1}T^{s-1}+\cdots+\alpha_1T+\alpha_0)(x))
\\
&=\alpha_sD_i\circ T^{s}(x)+\cdots+\alpha_1D_i\circ T(x)+\alpha_0D_i(x)
\\
&=[\alpha_sT^{s-1}(a)+\cdots+\alpha_1a_i,x]_{\mathfrak{g}}+\alpha_0D_i(x)
\\
&=[\chi_1(T)(a_i),x]_{\mathfrak{g}}+\alpha_0D_i(x),
\endaligned
$$
where $\chi_1(T)=\alpha_s x^{s-1}+.\cdots+\alpha_2 x+\alpha_1 \in
\mathbb{F}[X]$. If $\alpha_0 \neq 0$, then
$D_i=\ad_{\mathfrak{g}}(-\alpha_0^{-1}\chi_1(T)(a_i))$. If $\alpha_0=0$,
it follows that $\chi_1(T)(x) \in C_{\mathfrak{g}}(a_i)=\Ker D_i$. Thus,
$$
\aligned
0&=D_i(\chi_1(T)(x))=D_i((\alpha_sT^{s-1}+\cdots+\alpha_2T+\alpha_1)(x))
\\
&=\alpha_s[T^{s-2}(a_i),x]_{\mathfrak{g}}+\cdots+\alpha_2[a_i,x]_{\mathfrak{g}}+\alpha_1D_i(x)
\\
&=[\chi_2(T)(a_i),x]_{\mathfrak{g}}+\alpha_1D_i(x)
\endaligned
$$
where $\chi_2=\alpha_sx^{s-2}+\ldots+\alpha_3x+\alpha_2$. If $\alpha_1
\neq 0$ then $D_i=\ad_{\mathfrak{g}}(-\alpha_1^{-1}\chi_2(T)(a_i))$. If
$\alpha_1=0$ then $\chi_2(T)(x) \in
C_{\mathfrak{g}}(a)=\Ker D_i$. Let $j$ be the minimal positive integer such that $\alpha_{j} \neq 0$. Then
$D_i=\ad_{\mathfrak{g}}(-\alpha_{j}^{-1}\chi_{j+1}(T)(a_i))$, where
$\chi_{j+1}=\alpha_sx^{s-j-1}+\cdots+\alpha_{j+2}x+\alpha_{j+1} \in
\mathbb{F}[X]$. Therefore, $D_i$ is an inner derivation, for each $i$. 
\end{proof}

\noindent 
It is an immediate consequence of {\bf (ix)} in {\bf Lemma} \ref{lema 2} the following result.

\begin{corolario}\label{corolario internas}
Let $(\g,[\cdot,\cdot]_{\g})$ be a Lie algebra over a field $\F$ of zero characteristic. Let $T \in \Gamma(\g)$ be a non-zero centroid and let $D \in \Der(\g)$ such that $T \circ D=D \circ T=\ad_{\g}(a)$, for some $a \in \g$. If $\Ker(D)=C_{\g}(a)$, then $D$ is an inner derivation.
\end{corolario}

\noindent
In the next Proposition, we use {\bf Fitting's Lemma} (see \cite{Jac}) to obtain orthogonal direct sum decompositions for both $(\mathfrak{G},[\cdot,\cdot]_{\mathfrak{G}},B_{\mathfrak{G}})$ and $(\mathfrak{g},[\cdot,\cdot]_{\mathfrak{g}},B_{\mathfrak{g}})$. Besides, these decompositions
would be immediate to obtain if $\mathfrak{g}$ were a $k$-invariant subspace, but in general it is not. Nevertheless, there are conclusions that can be thrown from the initial data $\mathfrak{g}$ to produce a natural invariant decomposition of $\mathfrak{G}$ under an appropriate extension $L$ of $k$. Namely,

\begin{proposicion}\label{mrd1}
Let $k:\mathfrak{g} \rightarrow \mathfrak{G}$ be the injective linear map 
defined in {\bf Lemma \ref{lema 1}}, with $T=\pi_{\mathfrak{g}}\circ\,k:\mathfrak{g} \rightarrow \mathfrak{g}$ as in {\bf Lemma \ref{lema 2}}. Let $D_1,\ldots,D_r \in \Der \mathfrak{g}$ be the $B_{\mathfrak{g}}$-skew-symmetric derivations associated to the $2$-cocycle of the central extension. 
There exists a linear map $L:\mathfrak{G} \rightarrow \mathfrak{G}$, 
satisfying $L(x+v)=k(x)$, for all $x\in\mathfrak{g}$ and all $v\in V$, and the following additional properties: 

\begin{itemize}

\item[(i)] $\Ker L=V$.

\item[(ii)] $L \in \Gamma_{B_{\mathfrak{G}}}(\mathfrak{G})$.

\item[(iii)] There exists a positive integer $m$ such that $\mathfrak{q}=\Im T^m$ and $\mathfrak{n}=\Ker T^m$ are both non-degenerate ideals of $(\mathfrak{g},[\cdot,\cdot]_{\mathfrak{g}},B_{\mathfrak{g}})$ and $\mathfrak{g}=\mathfrak{q} \overset{\perp}{\oplus} \mathfrak{n}$.

\item[(iv)] Both, $\Im L^{m+1}$ and $\Ker L^{m+1}$ are non-degenerate ideals of $(\mathfrak{G},[\cdot,\cdot]_{\mathfrak{G}},B_{\mathfrak{G}})$ and $\mathfrak{G}=\Im L^{m+1} \overset{\perp}{\oplus} \Ker L^{m+1}$.

\item[(v)] $\Ker L^{m+1}=\mathfrak{n} \oplus V$ 
and there exists an invariant metric $\bar{B}_{\mathfrak{q}}$ on $\mathfrak{q}$ making it isometric to $\Im L^{m+1}$.

\item[(vi)] Both ideals, $\mathfrak{q}$ and $\mathfrak{n}$, are invariant under the $r$ skew-symmetric derivations $D_i$ which are inner in $\mathfrak{q}$ for all $1 \leq i \leq r$. 

\item[(vii)] $\mathfrak{g}^m \subseteq \mathfrak{q} $  and $\mathfrak{n} \subseteq C_{m}(\mathfrak{g})$.

\end{itemize}
\end{proposicion}
\begin{proof}
\noindent Let $\mathfrak{a}=\span_{\mathbb{F}}\{a_1,\ldots,a_r\}$. By \eqref{ex4}, we get $\mathfrak{g}=\Im T+\mathfrak{a}$,
and the non-degeneracy of $B_{\g}$ implies $\Ker T \cap \mathfrak{a}^{\perp}=\{0\}$. 

\smallskip
\noindent (i) From the definition of $L$ and by the fact that $k:\mathfrak{g} \to \mathfrak{G}$ is injective ({\bf Lemma \ref{lema 1}.(i)}), the statement
$\Ker L=V$ is obvious.

\medskip
\noindent (ii) Let $x,y \in \mathfrak{g}$ and $u,v \in V$. Then $L([x+u,y+v]_{\mathfrak{G}})=k([x,y]_{\mathfrak{g}})=[k(x),y]_{\mathfrak{G}}=[L(x+u),y+v]_{\mathfrak{G}}$.

\medskip
\noindent On the other hand, 
$$
\aligned
B_{\mathfrak{G}}(L(x+u),y+v)&=B_{\G}(k(x),y+v)=B_{\mathfrak{G}}(k(x),y)=B_{\mathfrak{g}}(x,y)=B_{\G}(x,k(y))\\
\,&=B_{\mathfrak{G}}(x+u,k(y))=B_{\G}(x+u,L(y+v)),\,\,\forall x,y \in \g,\,\,\forall u,v \in V.
\endaligned
$$ 
\noindent Therefore $L \in \Gamma_{B_{\mathfrak{G}}}(\mathfrak{G})$.

\medskip
\noindent (iii) By {\bf Fitting's Lemma} (see \cite{Jac}), there exists $m \in \mathbb{N}\cup\{0\}$ such that $\mathfrak{g}=\Im T^m \oplus \Ker T^m$, as vector spaces. Let $\mathfrak{q}=\Im T^m$ and $\mathfrak{n}=\Ker T^m$. We observe that $\mathfrak{q}$ and $\mathfrak{n}$ are both non-degenerate ideals of $(\mathfrak{g},[\cdot,\cdot]_{\mathfrak{g}},B_{\mathfrak{g}})$ and $\mathfrak{g}=\mathfrak{q} \overset{\perp}{\oplus} \mathfrak{n}$.

\medskip
\noindent (iv) Let $x \in \mathfrak{g}$, $v \in V$ and $j \in \mathbb{N}$. We affirm that for all $1 \leq j < k$, $L^j(x+v)=T^{j}(x)+\displaystyle{\sum_{i=1}^rB_{\g}(a_i,T^{j-1}(x))v_i}$, for all $x \in \g$ and for all $v \in v$. We claim that the last expression is valid for $k \in \N$, $k>1$. Let $x \in \g$ and $v \in V$, then
$$
\aligned
L^k(x+v)&=L^k(x)=L^{k-1}(L(x))=L^{k-1}(T(x))\\
\,&=T^{k-1}(T(x))+\sum_{i=1}^rB_{\g}(a_i,T^{k-2}(T(x)))v_i=T^k(x)+\overset{r}{\underset{i=1}{\sum}}B_{\mathfrak{g}}(a_i,T^{k-1}(x))v_i.
\endaligned
$$
Let $m_1$ be the minimal positive integer appearing in the Fitting decomposition of $\mathfrak{G}$ associated to $L$. 
Since $L$ belongs to $\Gamma_{B_{\mathfrak{G}}}(\mathfrak{G})$,
it follows that $\Im L^{m_1}$ and $\Ker L^{m_1}$ are 
non-degenerate ideals of $(\mathfrak{G},[\cdot,\cdot]_{\mathfrak{G}},B_{\mathfrak{G}})$
and $\mathfrak{G}=\Im L^{m_1} \overset{\perp}\oplus \Ker L^{m_1}$.

\smallskip
\noindent 
We now want to show that $m_1=m+1$, $\Im L^{m_1}$ is isometric to $\mathfrak{q}$ and $\Ker(L^{m_1})=\mathfrak{n} \oplus V$. We claim first that $m \leq m_1-1$. Let $x \in \Ker T^{m_1}$, then $L^{m_1}(x)\in V$. 
Whence, $L^{m_1+1}(x)=0$, which implies $x \in \Ker L^{m_1+1}=\Ker L^{m_1}$, thus $T^{m_1-1}(x) \in \Ker T \cap \mathfrak{a}^{\perp}=\{0\}$, so $x \in \Ker T^{m_1-1}$. This proves that $\Ker T^{m_1}=\Ker T^{m_1-1}$. Therefore, $\Im T^{m_1}=\Im T^{m_1-1}$ and by the minimality of $m$, we conclude that $m \leq m_1-1$.

\smallskip
\noindent 
We shall now prove that $m_1 \leq m+1$. If $y \in \Im L^{m+1}$, there exists $x \in \mathfrak{g}$ such that $y=L^{m+1}(x)=T^{m+1}(x)+\overset{r}{\underset{i=1}{\sum}}B_{\mathfrak{g}}(a_i,T^{m}(x))v_i$. Then $T^{m+1}(x) \in \Im T^{m+1}=\Im T^{m+2}$.
Therefore, there exists $x' \in \mathfrak{g}$, such that $T^{m+1}(x)=T^{m+2}(x')$ and $x-T(x') \in \Ker T^{m+1}=\Ker T^m$.
Thus, $T^m(x)=T^{m+1}(x')$. Finally, $y$ can be written as $y=T^{m+2}(x')+\overset{r}{\underset{i=1}{\sum}}B_{\mathfrak{g}}(a_i,T^{m+1}(x'))v_i=L^{m+2}(x')$. This implies that $y \in \Im L^{m+2}$. Whence, $\Im L^{m+1} \subset \Im L^{m+2}$ and therefore,
$\Im L^{m+1}=\Im L^{m+2}$. Since $L \in \Gamma_{B_{\mathfrak{G}}}(\mathfrak{G})$, it follows that 
$\Ker L^{m+1}=\Ker L^{m+2}$, so that $m_1 \leq m+1$. Therefore, $m_1=m+1$. 

\medskip
\noindent 
(v) Now, $x+v \in \Ker L^{m+1}$ implies $T^{m+1}(x)=0$, so $x \in \Ker T^{m+1}=\Ker T^m$ and $\Ker L^{m+1} \subseteq \mathfrak{n} \oplus V$. On the other hand, if $x \in \Ker T^m$, then $L^{m}(x) \in V$, so $L^{m+1}(x)=0$. This proves that $\Ker T^m \oplus V \subseteq \Ker L^{m+1}$. Therefore $\Ker L^{m+1}=\Ker T^m \oplus V$.

\smallskip
\noindent 
Let us now consider the restriction $\sigma=T|_{\mathfrak{q}}:\mathfrak{q} \rightarrow \mathfrak{q}$.
By {\bf Fitting\'{}s Lemma}, $\sigma$ is invertible. 
Consider the assignment $\Lambda:
\mathfrak{q} \to V^\perp$, defined by $\Lambda(x):=k\circ\sigma^{-1}(x)$, for all $x \in \mathfrak{q}$. 
For any $x,y \in \mathfrak{q}$, we have $\sigma^{-1}([x,y]_{\mathfrak{g}})=[\sigma^{-1}(x),y]_{\mathfrak{g}}$, 
so that $\Lambda([x,y]_{\mathfrak{g}})=[\Lambda(x),y]_{\mathfrak{G}}=[x,y]_{\mathfrak{G}}$.
On the other hand, since $V \subseteq C(\mathfrak{G})$, it is clear that $[\Lambda(x),\Lambda(y)]_{\mathfrak{G}}=[x,y]_{\mathfrak{G}}$.
%$$
%[\Lambda(x)),\Lambda(y))]_{\mathfrak{G}}=[x,y]_{\mathfrak{G}}.
%$$ 
Thus, $\Lambda([x,y]_{\mathfrak{g}})=[\Lambda(x),\Lambda(y)]_{\mathfrak{G}}$ for all $x,y \in \mathfrak{q}$, 
which means that $\Lambda:\mathfrak{q}\to V^{\perp}$ is an injective Lie algebra morphism.
It only remains to prove that $\Im \Lambda=\Im L^{m+1}$.
Since $\sigma(\mathfrak{q})=\mathfrak{q}$ and $\mathfrak{q}=\mathfrak{n}^{\perp}$ (as subspaces of $\mathfrak{g}$),
 {\bf Lemma \ref{lema 1}.(ii)} implies that $B_{\mathfrak{G}}(\Lambda(x),y)=B_{\mathfrak{g}}(\sigma^{-1}(x),y)=0$,
for any $x \in \mathfrak{q}$ and $y \in \mathfrak{n}$.
Thus, $\Lambda(x) \in (\mathfrak{n} \oplus V)^{\perp}=\Im L^{m+1}$, for all $x \in \mathfrak{q}$.
Therefore, $\Im \Lambda \subseteq \Im L^{m+1}$.  
By the direct sum decompositions $\mathfrak{G}=\Im L^{m+1} \oplus (\mathfrak{n} \oplus V)$ and $\mathfrak{g}=\mathfrak{q} \oplus \mathfrak{n}$, it follows that $\dim_{\mathbb{F}}\Im L^{m+1}=\dim_{\mathbb{F}}\mathfrak{q}$.
Thus, $\Im \Lambda=\Im L^{m+1}$ and $\Lambda$ is an isomorphism of Lie algebras between $\mathfrak{q}$ and $\Im L^{m+1}$.

\smallskip
\noindent 
Since $\mathfrak{q}=\operatorname{Im}T^m$, and $\mathfrak{g}=\mathfrak{q}\overset{\perp}\oplus\mathfrak{n}$,
it follows that $B_{\mathfrak{g}}$ restricts to a non-degenerate invariant metric on $(\mathfrak{q},[\cdot,\cdot]_{\mathfrak{g}}|_{\mathfrak{q} \times \mathfrak{q}})$. Moreover, the invertible map $\sigma^{-1}\in\operatorname{End}(\mathfrak{q})$ satisfies
$B_{\mathfrak{g}}(\sigma^{-1}(x),y)=B_{\mathfrak{g}}(x,\sigma^{-1}(y))$, for any pair $x,y\in\mathfrak{q}$
and it is $\operatorname{ad}_{\mathfrak{q}}$-equivariant.
Therefore, we define the metric $\bar{B}_{\mathfrak{q}}$ on $\mathfrak{q}$ by means of
$\bar{B}_{\mathfrak{q}}(x,y)=B_{\mathfrak{g}}(\sigma^{-1}(x),y)$, for all $x,y \in \mathfrak{q}$, which is in general different from $B_{\mathfrak{q}}$.
Since $T(\sigma^{-1}(y))=y$ for all $y \in \mathfrak{q}$, it follows that,
$$
B_{\mathfrak{G}}(\Lambda(x),\Lambda(y))=B_{\mathfrak{G}}(\Lambda(x),y)
=B_{\mathfrak{g}}(\sigma^{-1}(x),y) =\bar{B}_{\mathfrak{q}}(x,y),
$$ 
for all $x,y \in \mathfrak{q}$. Therefore, $\Lambda:\mathfrak{q} \rightarrow \Im L^{m+1}$
is an isometry between 
$(\mathfrak{q},[\cdot,\cdot]_{\mathfrak{g}}|_{\mathfrak{q} \times \mathfrak{q}},\bar{B}_{\mathfrak{q}})$ and 
$(\Im L^{m+1},[\cdot,\cdot]_{\mathfrak{G}}|_{\Im L^{m+1} \times \Im L^{m+1}},B_{\mathfrak{G}}|_{\Im L^{m+1} \times \Im L^{m+1}})$.

\medskip
\noindent (vi) By {\bf Lemma \ref{lema 2}.(vii)}, it follows that both $\mathfrak{q}$ and $\mathfrak{n}$ are invariant under the $B_{\mathfrak{g}}$-skew-symmetric derivations $D_1,\ldots,D_r \in \Der \mathfrak{g}$. Let us consider the component $p_i \in \mathfrak{q}$ of $a_i \in \mathfrak{g}$; \emph{ie}, $a_i=p_i+n_i$, where $n_i \in \mathfrak{n}$, for $1 \leq i \leq r$. Since $\sigma=T|_{\mathfrak{q}}$ is invertible, there exists $q_i \in \mathfrak{q}$, such that, $p_i=T(q_i)$, for each $1 \leq i \leq r$. Let $x \in \mathfrak{q}$. Then, there exists a unique $y \in \mathfrak{q}$, such that $x=T(y)$, and by {\bf Lemma \ref{lema 2}.(vii)}, we conclude that, $D_i(x)=D_i(T(y))=[a_i,y]_{\mathfrak{g}}=[T(q_i)+n_i,y]_{\mathfrak{g}}=[q_i,T(y)]_{\mathfrak{q}}=[q_i,x]_{\mathfrak{q}}$.
%$$
%D_i(x)=D_i(T(y))=[a_i,y]_{\mathfrak{g}}=[T(q_i)+n_i,y]_{\mathfrak{g}}=[q_i,T(y)]_{\mathfrak{q}}%=[q_i,x]_{\mathfrak{q}}.
%$$
That is, $D_i|_{\mathfrak{q}}=\ad_{\mathfrak{q}}(q_i)$, for some $q_i\in\mathfrak{q}$.

\medskip
\noindent (vii) By {\bf Lemma \ref{lema 2}.(viii)}, we have $\mathfrak{g}^m \subseteq \mathfrak{q}$ and $\mathfrak{n} \subseteq C_m(\mathfrak{g})$. Therefore, $\mathfrak{n}$ is a non-degenerate nilpotent ideal of $(\mathfrak{g},[\cdot,\cdot]_{\mathfrak{g}},B_{\mathfrak{g}})$. 
\end{proof}

\noindent As an application of the results exhibited in the above Lemmas, we state the following criterion to determine the conditions under which $V$ either degenerates or not.

\begin{proposicion}\label{teorema de estructura 2}
Let $(\mathfrak{g},[\cdot,\cdot]_{\mathfrak{g}},B_{\mathfrak{g}})$ be a quadratic Lie algebra and let $V$ be an $r$-dimensional vector space. Let $(\mathfrak{G},[\cdot,\cdot]_{\mathfrak{G}})$ be a central extension of $(\mathfrak{g},[\cdot,\cdot]_{\mathfrak{g}})$ by $V$ with $D_1,\ldots,D_r \in \Der\mathfrak{g}$ being the $B_{\mathfrak{g}}$-skew-symmetric derivations associated to the $2$-cocycle of the extension. Let $T \in \End_{\mathbb{F}}\mathfrak{g}$ be the linear map defined in {\bf Lemma \ref{lema 2}}. Suppose there exists an invariant metric $B_{\mathfrak{G}}$ on $(\mathfrak{G},[\cdot,\cdot]_{\mathfrak{G}})$.
\begin{itemize}

\item[(i)] 
$T$ is invertible if and only if $V \cap V^{\perp}=\{0\}$ (\emph{i.e.} $V$ is a non-degenerate ideal of $(\mathfrak{G},[\cdot,\cdot]_{\mathfrak{G}},B_{\mathfrak{G}})$). 
If \,$T$ is invertible, it can be used to define an invariant metric $\bar{B}_{\mathfrak{g}}$ on $(\mathfrak{g},[\cdot,\cdot]_{\mathfrak{g}})$, 
through $\bar{B}_{\mathfrak{g}}(x,y)=B_{\mathfrak{g}}(T^{-1}(x),y)$,
making 
$(\mathfrak{G},[\cdot,\cdot]_{\mathfrak{G}},B_{\mathfrak{G}})$ isometric to
$(\mathfrak{g},[\cdot,\cdot]_{\mathfrak{g}},\bar{B}_{\mathfrak{g}}) \overset{\perp}{\oplus} (V,B_{\mathfrak{G}}|_{V \times V})$.

\item[(ii)] 
Suppose that at least one of the $B_{\g}$-skew -symmetric derivations $D_1,...,D_r$, associated to the 2-cocycle of the central extension is not inner. Then $V \cap V^{\perp} \neq \{0\}$ if and only if there exists some $m \in \mathbb{N}$, such that, $C_m(\mathfrak{g}) \neq \{0\}$, and,
$$
\mathfrak{g} = \mathfrak{g}^m
 \overset{\perp}{\oplus} C_m(\mathfrak{g}),
$$
is a direct orthogonal sum of quadratic Lie algebras, 
with $(\mathfrak{g}^m,[\cdot,\cdot]_{\mathfrak{g}}\vert_{
\mathfrak{g}^m\times \mathfrak{g}^m})$
being a perfect Lie algebra and $C_m(\mathfrak{g})$ a nilpotent one.

\end{itemize}

\end{proposicion}
\begin{proof}
\noindent (i) If $T$ is invertible then the result follows from {\bf Proposition \ref{mrd1}.(iii)}, 
{\bf (iv)} and {\bf (v)}. Suppose that $V$ is non-degenerate.  Then there exists a basis $\{v'_1,\ldots,v'_r \}$ of $V$ such that $B_{\mathfrak{G}}(v_i,v'_j)=\delta_{ij}$, $i,j=1,\ldots,r$. It follows
from {\bf Lemma \ref{lema 2}.(vi)} that $T\circ h(v'_i)=-a_i$, for all $1 \leq i \leq r$. By the decomposition
$\mathfrak{g}=\Im T+\mathfrak{a}$ of {\bf Lemma \ref{lema 2}.(vi)}, it is trivial that
$\mathfrak{g}=\Im T$, which implies that $T$ is surjective. By {\bf Proposition \ref{mrd1}.(iii)}, {\bf (iv)} and {\bf (v)}, if $T$ is invertible, then $(\mathfrak{G},[\cdot,\cdot]_{\mathfrak{G}},B_{\mathfrak{G}})$ is isometric to the orthogonal direct sum of $(\mathfrak{g},[\cdot,\cdot]_{\mathfrak{g}},\bar{B}_{\mathfrak{g}})$ and $(V,B_{\mathfrak{G}}|_{V \times V})$.
\medskip

\noindent (ii) According of the hypotheses of the statement, we assume that at least one of the $B_{\g}$-skew-symmetric derivations $D_i$ is not inner. We claim that there are $x_1,\ldots , x_r \in \g$ such that $\displaystyle{\cap_{i=1}^r\Ker(D_i-\ad_{\g}(x_i))=\{0\}}$. Indeed, for each $D_i$, $1 \le i \le r$, we define the following family of subalgebras of $(\g,[\cdot,\cdot]_{\g})$:
$$
\mathcal{F}_i=\{S \subset \g\,|\,S \neq \{0\}\mbox{ is a subalgebra and }D_{i}|_S=\ad_{\g}(x)|_S, \mbox{ for some }x \in \g \}.
$$
If $\mathcal{F}_{i}$ is empty, then $\Ker(D_i)=\{0\}$. In fact, assuming that $\Ker(D_i) \neq 0$, we may choose
a non-zero $x \in \Ker(D_i)$. Obviously, $D_i(x)=[x,x]_{\g}=0$, and $\Bbb F\,x \in \mathcal{F}_{i}$. 
On the other hand, if
$\Ker(D_i)=0$, for some $i$, it is obvious that $\displaystyle{\cap_{j=1}^r(\Ker(D_j))}=\{0\}$. Thus, we assume that $\mathcal{F}_i$ is non-empty for all $i$. Let $S_i \in \mathcal{F}_i$ be a minimal subalgebra. By the definition of $S_i$, there exists a non-zero $x_i \in \g$, such that $D_i|_{S_i}=\ad_{g}(x_i)|_{S_i}$, for all $1 \le i \le r$. 
Now, for each $i$, define 
$D^{\prime}_i=D_i-\ad_{\g}(x_i)$. 
Clearly, $\Ker(D_i^\prime)=S_i$, and $D_i^\prime$ is still a $B_{\frak{g}}$-skew-symmetric derivation of $\g$.
Let $S=\displaystyle{\cap_{i=1}^r\Ker(D^{\prime}_i)}$. Choose some $x \in S$. Then, $D_i(x)=[x_i,x]_{\g}$ for all $1 \le i \le r$. Since $S \subset S_i$, for all $i$, the minimality property of $S_i$ implies that
either $S=\{0\}$ or $S=S_i$, for all $1 \le i \le r$.
Suppose $S=S_i$ for all $i$. 
By {\bf Proposition \ref{escision}} and {\bf Proposition \ref{proposicion 1}}, 
the 2-cocycle $\theta$ associated to the $D_{i}$'s is  equivalent to the 2-cocycle 
$\theta^{\prime}$ associated to the $D^{\prime}_i$'s, in the sense of the central 
extensions arising from these 2-cocycles are naturally isomorphic. 
Let $[\cdot,\cdot]^{\prime}_{\G}$ be the Lie bracket on $\G=\g \oplus V$, 
associated to the the cocycle $\theta^\prime$ 
Therefore, there exists an invariant metric $B_{\G}^{\prime}$ in $(\G,[\cdot,\cdot]^{\prime}_{\G})$, induced by $B_{\G}$, for which the quadratic Lie algebras $(\G,[\cdot,\cdot]_{\G},B_{\G})$ and $(\G,[\cdot,\cdot]^{\prime}_{\G},B^{\prime}_{\G})$ are isometric. Now, by {\bf Lemma \ref{lema 1}.(ii)}, there exists a linear map $k^{\prime}:\g \to \G$, 
such that $B_{\g}(x,y)=B_{\G}^{\prime}(k^{\prime}(x),y)$, for all $x,y \in \g$. 
Let $T^{\prime}=\pi_{\g} \circ k^{\prime}:\g \to \g$ be the self-adjoint centroid appearing in {\bf Lemma \ref{lema 2}}. 
Then, {\bf Lemma \ref{lema 2}.(viii)} implies that there exist elements $a^{\prime \prime}_i$, $1 \le i \le r$, such that $T^{\prime} \circ D_{i}^{\prime}=D_i^{\prime} \circ T^{\prime}=\ad_{\g}(a_i^{\prime \prime})$. 
If $T^{\prime} \neq 0$, then 
{\bf Lemma \ref{lema 2}.(ix)} implies that all the $D^{\prime}_i$'s are inner derivations, because $\Ker(D^{\prime}_1)= \cdots =\Ker(D_r^{\prime})$. Consequently, 
each $D_i$ is itself inner, which contradicts our assumption.
On the other hand, if $T^{\prime}=0$, then $k^{\prime}([x,y]_{\g})=0$, for all $x,y \in \g$. Since $k^{\prime}$ is injective ({\bf Lemma \ref{lema 1}.(i)}), 
it follows that $(\g,[\cdot,\cdot]_{\g})$ is Abelian 
and $D^{\prime}_i=D_i$, for each $i$. Since $D_i(a_i)=0$, for all $i$, then $D_i(a_j)=0$, for all $i,j$, because $\Ker(D_1)=\ldots =\Ker(D_r)$. Thus,
$$
D_j(x)=h([a_j,x]_{\G})=\sum_{i=1}^rB_{\g}(D_i(a_j),x)h(v_i)=0,\quad \forall x \in \g,\quad \forall 1 \le j \le r.
$$ 
Whence $D_j=0$, for all $j$, contradicting the hypothesis that at least one $B_{\g}$-skew-symmetric derivation is not inner. Therefore, $S=\displaystyle{\cap_{i=1}^r\Ker(D_i-\ad_{g}(x_i))}=\{0\}$. 
As we said, the 2-cocycle associated to the $B_{\g}$-skew-symmetric derivations $D^{\prime}_i$'s
yields an isomorphic central extension as that obtained from  the 2-cocycle associated
to the $D_i$'s ({\bf Proposition \ref{escision}} and {\bf Proposition \ref{proposicion 1}})
we may conclude, with no loss of generality that  $\displaystyle{\cap_{i=1}^r\Ker(D_i)}=\{0\}$.
\smallskip

\noindent Let us consider the skew-symmetric product $[\cdot,\cdot]_{\Delta}:\mathfrak{g} \times \mathfrak{g} \rightarrow \mathfrak{g}$, defined in {\bf Lemma \ref{lema 2}.(iii)}, as $[x,y]_{\Delta}=\rho(x)(y)=h([x,y]_{\mathfrak{G}})$, for all $x,y \in \mathfrak{g}$. Using the fact $\displaystyle{\cap_{i=1}^r\Ker(D_i)}=\{0\}$, we shall now prove that if $x \in \mathfrak{g}$ satisfies $[x,y]_{\Delta}=0$ for all $y \in \mathfrak{g}$, then $x=0$. Indeed, let $x \in \mathfrak{g}$, satisfy $[x,y]_{\Delta}=h([x,y]_{\mathfrak{G}})=0$, for all $y \in \mathfrak{g}$. Then, $[x,y]_{\mathfrak{G}} \in \Ker(h)=\mathfrak{g}^{\perp}$, for all $y \in \mathfrak{g}$. Since $B_{\mathfrak{G}}$ is invariant under the Lie bracket $[\cdot,\cdot]_{\mathfrak{G}}$, and $V$ is contained in the center of $(\mathfrak{G},[\cdot,\cdot]_{\mathfrak{G}})$, then $[x,y]_{\mathfrak{G}} \in V^{\perp}$. Since $B_{\mathfrak{G}}$ is non-degenerate, and $[x,y]_{\mathfrak{G}}=[x,y]_{\mathfrak{g}}+\displaystyle{\small{\sum_{i=1}^r}B_{\mathfrak{g}}(D_i(x),y)v_i}$, for all $y \in \mathfrak{g}$, the condition $[x,y]_{\mathfrak{G}}=0$, for all $y \in \mathfrak{g}$, implies that, $[x,y]_{\mathfrak{g}}=0$ and $B_{\mathfrak{g}}(D_i(x),y)=0$, for all $y \in \mathfrak{g}$,
and all $1 \leq i \leq r$. Since $B_{\mathfrak{g}}$ is non-degenerate, then $D_i(x)=0$, for $1 \leq i \leq r$. Thus, $x \in C(\mathfrak{g}) \cap (\displaystyle{\cap_{i=1}^r\Ker(D_i)})=\{0\}$. This argument thus proves that the center of the skew-symmetric algebra $(\mathfrak{g},[\cdot,\cdot]_{\Delta})$ is zero.

\smallskip
\noindent 
On the other hand, {\bf Lemma \ref{lema 2}.(iii)} shows that $B_{\mathfrak{g}}([x,y]_{\Delta},z)=B_{\mathfrak{g}}(\rho(x)(y),z)=-B_{\mathfrak{g}}(y,\rho(x)(z))=-B_{\mathfrak{g}}(y,[x,z]_{\Delta})$.
%$$
%B_{\mathfrak{g}}([x,y]_{\Delta},z)=B_{\mathfrak{g}}(\rho(x)(y),z)=-B_{\mathfrak{g}}(y,\rho(x)(z))=-%B_{\mathfrak{g}}(y,[x,z]_{\Delta}).
%$$ 
So, $B_{\mathfrak{g}}$ is invariant under the skew-symmetric product $[\cdot,\cdot]_{\Delta}$. We have already known that the center of $(\mathfrak{g},[\cdot,\cdot]_{\Delta})$ is zero, so $\mathfrak{g}=[\mathfrak{g},\mathfrak{g}]_{\Delta}$. Then, for each $x\in\mathfrak{g}$, there is a positive integer $\ell$ and $x_i,y_i \in \mathfrak{g}$, ($1\le i\le \ell$), 
such that, $x=\overset{\ell}{\underset{j=1}{\sum}} [x_j,y_j]_{\Delta}=\overset{\ell}{\underset{j=1}{\sum}} h([x_j,y_j]_{\mathfrak{G}})$.
%$$
%x=\overset{\ell}{\underset{j=1}{\sum}} [x_j,y_j]_{\Delta}=\overset{\ell}{\underset{j=1}{\sum}} h([x_j,y_j]_{\mathfrak{G}}).
%$$
By {\bf Lemma \ref{lema 2}.(vi)}, we obtain,
$$
\aligned
\sum_{j=1}^{\ell}[x_j,y_j]_{\mathfrak{g}}
&=\sum_{j=1}^{\ell}T\circ h([x_i,y_j]_{\mathfrak{G}})+\sum_{j=1}^{\ell} \sum_{i=1}^rB_{\mathfrak{G}}([x_j,y_j]_{\mathfrak{G}},v_i)a_i
\\
&=\sum_{j=1}^{\ell}T\circ h([x_i,y_j]_{\mathfrak{G}})=T(x).
\endaligned
$$
This implies that $\Im T \subset [\mathfrak{g},\mathfrak{g}]_{\mathfrak{g}}$. Therefore, $\Im T=[\mathfrak{g},\mathfrak{g}]_{\mathfrak{g}}$
and $\Ker T=C(\mathfrak{g})$. From {\bf Lemma \ref{lema 2}.(viii)},
and an induction argument we may conclude that $\mathfrak{g}^{\ell}=\Im T^{\ell}$ and $\Ker T^{\ell}=C_{\ell}(\mathfrak{g})$ for each $\ell \in \mathbb{N}$. Indeed, suppose that $\Im T^{\ell-1}=\g^{\ell-1}$, then 
$$
\aligned 
T^{\ell}(\g)=&T(T^{\ell-1}(\g))=T([\g,\g^{\ell-2}]_{\g})=[T(\g),\g^{\ell-2}]_{\g}\\
\,&=[\g^1,\g^{\ell-2}]_{\g}=[[\g,\g]_{\g},\g^{\ell-2}]_{\g} \subset [\g,[\g,\g^{\ell-2}]_{\g}]_{\g}\\
\,&=[\g,\g^{\ell-1}]_{\g}=\g^{\ell}.
\endaligned
$$ 
Since $\Ker(T^{\ell})^{\perp}=\Im(T^{\ell})$ and ${(\g^{\ell})}^{\perp}=C_{\ell}(\g)$, for all $\ell \in \N$, one obtains $\Ker T^{\ell}=C_{\ell}(\mathfrak{g})$, for all $\ell \in \N$. Let $m$ be the positive integer that gives the Fitting decomposition for $T$; {\it ie\/,} $\mathfrak{g}=\Im T^m \oplus \Ker T^m$. By {\bf Propostion \ref{mrd1}.(iii)}, $\Im T^m$ and $\Ker T^m$ are non-degenerate ideals of $(\mathfrak{g},[\cdot,\cdot]_{\mathfrak{g}},B_{\mathfrak{g}})$. 
It follows that $\Im T^m=\mathfrak{g}^m$ and $\Ker T^m=C_m(\mathfrak{g})$ are the characteristic invariant ideals that satisfy the statement.

\smallskip
\noindent
We are now ready to prove the {\it if and only if\/} statement (ii):

\smallskip

\noindent $(\Leftarrow)$ From the hypothesis of the statement, we have that the quadratic Lie algebra $(\g,[\cdot,\cdot]_{\g},B_{\g})$, has the following orthogonal decomposition:
$$
\mathfrak{g} = \mathfrak{g}^m
 \overset{\perp}{\oplus} C_m(\mathfrak{g}),\quad C_{m}(\g)\neq \{0\}.
$$
As we prove above, this orthogonal decomposition corresponds to the Fitting's decomposition for the linear map $T:\g \to \g$, where by {\bf Lemma \ref{lema 2}.(viii)}, $\Ker T^m=C_m(\g)$ and $\Im T^m=\g^m$. If $V$ is non-degenerate. 
then $T$ is invertible, which implies $C_m(\g)=\{0\}$, contradicting the fact that 
the latter is non-zero. Therefore, $V \cap V^{\perp} \neq \{0\}$.
\smallskip

\noindent 
$(\Rightarrow)$ Assume $V \cap V^{\perp} \neq \{0\}$. We shall prove that $(\Im T^m,[\cdot,\cdot]_{\mathfrak{g}}|_{\Im T^m \times \Im T^m})$ has trivial center. By (i), we have that $T$ is non-invertible, then  $\{0\} \neq \Ker T^m=C_m(\mathfrak{g}) \neq \{0\}$.  We claim that the the quotient space $\mathfrak{b}_m=\mathfrak{g}/C_m(\mathfrak{g})$ has a Lie algebra structure induced by $[\cdot,\cdot]_{\Delta}$, given by $[x+C_m(\g),y+C_m(\g)]_{\mathfrak{b}_m}=[x,y]_{\Delta}+C_m(\g)$, for all $x,y \in \g$. Indeed, let $x,y,z \in \mathfrak{g}$. By {\bf Lemma \ref{lema 2}.(iii)} and {\bf (vii)}, we have 
$$
T([x,[y,z]_{\Delta}]_{\Delta})=[x,[y,z]_{\Delta}]_{\mathfrak{g}}=[x,h([y,z]_{\G})]_{\g}=h([x,[y,z]_{\G}]_{\G}).
$$
So, by the Jacobi identity of $[\cdot,\cdot]_{\mathfrak{G}}$, it follows that, $[x,[y,z]_{\Delta}]_{\Delta}+[y,[z,x]_{\Delta}]_{\Delta}+[z,[x,y]_{\Delta}]_{\Delta} \in \Ker(T) \subset C_m(\mathfrak{g})$, which implies that $\mathfrak{b}_{m}$ has a Lie algebra structure induced by the skew-symmetric product $[\cdot,\cdot]_{\Delta}$.
\smallskip

\noindent 
Let $x+C_m(\mathfrak{g})$ be in the center of the Lie algebra $(\mathfrak{b}_m,[\cdot,\cdot]_{\Delta})$ and let $y \in \mathfrak{g}$ be an arbitrary element. Then $[x,y]_{\Delta} \in C_m(\mathfrak{g})$, which implies $[x,y]_{\mathfrak{g}}=T([x,y]_{\Delta}) \in C_m(\mathfrak{g})$ for all $y \in \mathfrak{g}$; thus, $x \in C_{m+1}(\mathfrak{g})=C_{m}(\mathfrak{g})$. Therefore, $x+C_m(\mathfrak{g})=0$ in the quotient and $(\mathfrak{b}_m,[\cdot,\cdot]_{\Delta})$ has no center.
To conclude the argument, notice that the linear map $x \mapsto T(x)+C_m(\mathfrak{g})$ yields a Lie algebra isomorphism, $(\Im T^m,[\cdot,\cdot]_{\mathfrak{g}}|_{\Im T^m \times \Im T^m})\to (\mathfrak{b}_m,[\cdot,\cdot]_{\Delta})$.
Therefore, $C(\Im T^m)=\{0\}$, since the algebra is perfect.

\end{proof}

\noindent We end up this section with a result that reduces the study of quadratic Lie algebras arising from central extensions of quadratic Lie algebras, to the consideration
of two extreme cases: when the kernel of the central extension is \emph{isotropic} 
and when it is \emph{non-degenerate}.

\begin{proposicion}\label{proposicion reduccion}
Let $(\mathfrak{g},[\cdot,\cdot]_{\mathfrak{g}},B_{\mathfrak{g}})$ be a quadratic Lie algebra, 
let $V$ be a  finite-dimensional vector space and let
$(\mathfrak{G},[\cdot,\cdot]_{\mathfrak{G}})$ be a central extension of $(\mathfrak{g},[\cdot,\cdot]_{\mathfrak{g}})$ by $V$. 
If there exists an invariant metric $B_{\mathfrak{G}}$ on $(\mathfrak{G},[\cdot,\cdot]_{\mathfrak{G}})$,
then there are two non-degenerate ideals, $U$ and $U^\perp$, of $\mathfrak{G}$, such that,
\begin{itemize}

\item[(i)] $U \subseteq V$,

\item[(ii)] $\mathfrak{G}=U^\perp\oplus U$

\item[(iii)] $(U^\perp,[\cdot,\cdot]_{\mathfrak{G}}|_{U^{\perp} \times U^{\perp}},B_{\mathfrak{G}}|_{U^\perp \times U^\perp})$ 
is a central extension of  $(\mathfrak{g},[\cdot,\cdot]_{\mathfrak{g}},B_{\mathfrak{g}})$
and the kernel of this extension is equal to $V \cap V^{\perp}$; {\it ie.,\/}
$$
0\to V \cap V^{\perp}\to U^\perp\to \frak{g}\to 0.
$$

\end{itemize}

\end{proposicion}
\begin{proof} 
\noindent Suppose $V$ degenerates but it is not isotropic; {\it ie.,\/} $V \cap V^{\perp} \neq \{0\}$ and $V \nsubseteq V^{\perp}$. 
Let $U$ be a complementary subspace to $V \cap V^{\perp}$ in $V$; that is, $V=(V \cap V^{\perp}) \oplus U$.
Then $U$ is a non-degenerate ideal of $\mathfrak{G}$. Thus, $\mathfrak{G}=U^\perp\oplus U$.
Let $\{u_1,\ldots,u_\ell \} \subseteq U$ be a basis of $U$, such that $B_{\mathfrak{G}}(u_i,u_j)=\delta_{ij}$. 
Then, $U^\perp=\{ x+v-\overset{\ell}{\underset{i=1}{\sum}}B_{\mathfrak{G}}(x+v,u_i)u_i\,\mid\,x \in \mathfrak{g},\,v \in V \}$. 
We may define a Lie bracket 
$[\cdot,\cdot]_{\mathfrak{p}}:\mathfrak{p}
\times \mathfrak{p} \rightarrow \mathfrak{p}$,
on the subspace 
$\mathfrak{p}=\{ x-\overset{\ell}{\underset{i=1}{\sum}}B_{\mathfrak{G}}(x,u_i)u_i\,\mid\,x \in \mathfrak{g}\}$,
through,
\begin{equation*}
\biggl[x-\sum_{i=1}^{\ell}B_{\mathfrak{G}}(x,u_i)u_i\,,\,y-\sum_{i=1}^{\ell}B_{\mathfrak{G}}(y,u_i)u_i \biggr]_{\mathfrak{p}}=[x,y]_{\mathfrak{g}}-\sum_{i=1}^{\ell}B_{\mathfrak{G}}([x,y]_{\mathfrak{g}},u_i)u_i,
\end{equation*}
for all $x,y \in \mathfrak{g}$. Now define an invariant metric $B_{\mathfrak{p}}:\mathfrak{p} \times \mathfrak{p} \rightarrow \mathbb{F}$ by
\begin{equation*}
B_{\mathfrak{p}}
\biggl(x-\sum_{i=1}^{\ell}B_{\mathfrak{G}}(x,u_i)u_i\,,\,y-\sum_{i=1}^{\ell}B_{\mathfrak{G}}(y,u_i)u_i \biggr)=B_{\mathfrak{g}}(x,y),\,\,\forall
x,y \in \mathfrak{g}.
\end{equation*}
So, $(\mathfrak{p},[\cdot,\cdot]_{\mathfrak{p}},B_{\mathfrak{p}})$ is a quadratic Lie algebra naturally isometric to
$(\mathfrak{g},[\cdot,\cdot]_{\mathfrak{g}},B_{\mathfrak{g}})$ 
via, $x-\overset{\ell}{\underset{i=1}{\sum}}B_{\mathfrak{G}}(x,u_i)u_i \mapsto x$, for all $x \in \mathfrak{g}$.
\smallskip

\noindent 
Now, as a vector space $U^\perp$ is the direct sum of $\mathfrak{p}$ and $V \cap V^{\perp}$. 
Let $\pi_{\mathfrak{p}}:U^\perp \rightarrow \mathfrak{p}$ be the projection onto $\mathfrak{p}$. 
Since $[\mathfrak{G},\mathfrak{G}]_{\mathfrak{G}} \subseteq U^\perp$ 
and $V\subseteq C(\mathfrak{G})$, it is a straightforward matter to prove that 
$\pi_{\mathfrak{p}}$ is a Lie algebra morphism.
Therefore $(U^\perp,[\cdot,\cdot]_{\mathfrak{G}}|_{U^\perp \times U^{\perp}})$ is a central extension of $(\mathfrak{p},[\cdot,\cdot]_{\mathfrak{p}})$ with isotropic kernel $V \cap V^{\perp}$; {\it ie.,\/}
$$
0 \rightarrow V \cap V^{\perp} \rightarrow U^\perp \rightarrow \mathfrak{p} \rightarrow 0.
$$
Since $\mathfrak{p}$ is naturally isometric to $\mathfrak{g}$, the statement follows.
\end{proof}

\noindent According to {\bf Proposition \ref{proposicion reduccion}},
the study of central extensions of quadratic Lie algebras which are quadratic themselves,
can be reduced to the two extreme cases in which either the kernel $V$ of the extension
is isotropic or it is non-degenerate. The case 
when $V$ is non-degenerate has been handled by {\bf Propostion \ref{teorema de estructura 2}}.
We now look at the case when $V$ is isotropic.

\subsection{Isotropic kernel}

\begin{lema}\label{lema 31}
Let $V$ be an isotropic ideal of $(\mathfrak{G},[\cdot,\cdot]_{\mathfrak{G}},B_{\mathfrak{G}})$, and let
$\mathfrak{g}=\Im T+\mathfrak{a}$, with $\Ker T \cap \mathfrak{a}^{\perp}=\{0\}$ and $\mathfrak{a}=\span_{\mathbb{F}}\{a_1,\ldots,a_r\}$ as in 
{\bf Lemma \ref{mrd1}}. Let $\{v_1,\dots, v_r\}$ be the basis of $V$ used in {\bf Lemma \ref{lema 1}}
and let $h:\mathfrak{G}\to\mathfrak{g}$ be the map defined therein.
Then,

\begin{itemize}

\item[(i)] $\{a_1,\ldots,a_r\}$ y $\{h(v_1),\ldots,h(v_r)\}$ are both linearly independent sets.

\item[(ii)] $h(\mathfrak{a}) \subset h(V)=\Ker T$, $r=\dim_{\mathbb{F}}\Ker T$ and $h|_V:V \rightarrow \Ker T$ is bijective.

\item[(iii)] $B_{\mathfrak{G}}(\Im T,V)=\{0\}$ and $T=T \circ h \circ T$.

\item[(iv)] $E=T \circ (h|_{\mathfrak{g}})$ is a projection ($E^2=E$) and the decomposition of $\mathfrak{g}$ associated to $E$ is
$\mathfrak{g}=\Im T \oplus \mathfrak{a}$.

\item[(v)] $F=h \circ T$ is a projection $(F^2=F)$ and the decomposition of
$\mathfrak{g}$ associated to $F$ is $\mathfrak{g}=\Ker T \oplus
\mathfrak{a}^{\perp}$, where $\Ker T=\Ker F$. Moreover, for each $x \in \mathfrak{g}$,
\begin{equation}\label{l1}
x=F(x)+\sum_{i=1}^rB_{\mathfrak{g}}(a_i,x)h(v_i).
\end{equation}
In addition, there exists a quadratic Lie algebra structure in the subspace $\mathfrak{a}^{\perp}$, such that $(\mathfrak{g},[\cdot,\cdot]_{\mathfrak{g}},B_{\mathfrak{g}})$ is a central extension of $(\mathfrak{a}^{\perp},[\cdot,\cdot]_{\mathfrak{a}^{\perp}},B_{\mathfrak{a}^{\perp}})$ by $\Ker T$.
\end{itemize}
\end{lema}
\begin{proof}

(i) If $V$ is isotropic, {\bf Lemma \ref{lema 2}.(i)} says that the bases 
$\{a_1+w_1,\ldots, a_r+~w_r\}$ and $\{v_1,\ldots,v_r\}$ are dual to each other with respect to $B_{\mathfrak{G}}$;
thus, $B_{\mathfrak{g}}(a_i,h(v_j))=B_{\mathfrak{G}}(a_i,v_j)=B_{\mathfrak{G}}(a_i+w_i,v_j)=\delta_{ij}$, 
and therefore, $\{a_1,\ldots,a_r\}$ and $\{h(v_1),\ldots,h(v_r)\}$ are both linearly independent sets.

\medskip
\noindent 
(ii) Since $V$ is isotropic, \eqref{ex4} can be written as,
\begin{equation}\label{ex4-2}
x=T \circ h(x+v)+\sum_{i=1}^rB_{\mathfrak{G}}(x,v_i)a_i,\,\,\forall \,x \in \mathfrak{g}\,,v \in V.
\end{equation}
Making $x=0$ in this expression, we conclude that $T(h(v))=0$, for all $v \in V$; {\it ie\/,} $h(V) \subseteq
\Ker T$. 
Now consider $x \in \Ker T$. We use the fact that $h \circ k=\operatorname{Id}_{\mathfrak{g}}$ in 
\eqref{expresion para k}, to obtain,
$$
x=h(k(x))=h(T(x))+\sum_{i=1}^rB_{\mathfrak{g}}(a_i,x)h(v_i)=\sum_{i=1}^rB_{\mathfrak{g}}(a_i,x)h(v_i),
$$
but this shows that $\Ker T \subseteq h(V)$. Thus $\Ker T=h(V)$.

\smallskip
\noindent 
Observe that $h(a_i)=-h(w_i) \in h(V)$, for all $1 \leq i \leq r$; {\it ie\/,} $h(\mathfrak{a}) \subset h(V)$, a fact that will be needed in (iv) below.

\medskip
\noindent (iii) We shall now prove that $B_{\mathfrak{G}}(\Im T,V)=\{0\}$. 
We shall use the fact that $\Im k=V^{\perp}$, proved in {\bf Lemma \ref{lema 1}.(iii)}, to show our assertion. Indeed,
let $x \in \mathfrak{g}$ and let $v \in V$. Then, $B_{\mathfrak{G}}(T(x),v)=B_{\mathfrak{G}}(k(x),v)=0$, which proves our claim. 
Finally, by \eqref{ex4-2}, we get,
$$
T(x)=T(h(T(x)))+\overset{r}{\underset{i=1}{\sum}}B_{\mathfrak{G}}(T(x),v_i)a_i=T \circ h \circ T(x),
\quad\text{for all\ }x \in \mathfrak{g}.
$$ 
Whence, $T=T \circ h \circ T$, as claimed.

\medskip
\noindent
(iv) Let $E=T \circ (h|_{\mathfrak{g}})$. It follows from (iii) that $E$ is idempotent ($E^2=E$). 
We have observed in (ii) that $h(\mathfrak{a}) \subset h(V)=\Ker T$. Thus, it is clear that $E(a)=0$ for all $a \in \mathfrak{a}$;
{\it ie\/,} $\mathfrak{a} \subseteq \Ker E$. On the other hand, if $x \in \Ker E$, then $x \in \mathfrak{a}$,
by means of \eqref{ex4-2}. Therefore, $\Ker E=\mathfrak{a}$ and consequently, $\dim_{\mathbb{F}}\Ker E=r$. 
On the other hand, it is clear that $\Im E \subset \Im T$. If $x \in \mathfrak{g}$, it follows 
from (iii) that $T(x)=E(T(x))$, so that $\Im E=\Im T$. Thus, the decomposition of $\mathfrak{g}$ associated to the projection 
$E$ is $\mathfrak{g}=\Im T \oplus \mathfrak{a}$.

\medskip
\noindent
(v) As we proceeded in the proof of (ii), we use the fact that, $h \circ k=\operatorname{Id}_{\mathfrak{g}}$ together with expression \eqref{expresion para k}, 
to write,
\begin{equation}\label{expresion para F}
x=h(k(x))=h(T(x))+\sum_{i=1}^rB_{\mathfrak{g}}(a_i,x)h(v_i),\,\,\forall x \in \mathfrak{g}. 
\end{equation}
It follows from (iii) that the operator $F=h \circ T$ is idempotent. It is clear that $\Ker T \subset \Ker F$. 
On the other hand, if $x \in \Ker F$, using \eqref{expresion para F}, we obtain, $x \in h(V)=\Ker T$. 
Therefore $\Ker F=\Ker T$ and $\dim_{\mathbb{F}}\Ker F=r$. 
Since $F$ is a projection, $\dim_{\mathbb{F}}\Im F=n-r$. 

\smallskip
\noindent We claim that $\mathfrak{a}^{\perp}$ (as subspace of $\mathfrak{g}$) is equal to $\Im F$. 
Indeed, let $a \in \mathfrak{a}$ and let $x \in \mathfrak{g}$. 
Now (ii) says that $h(a) \in \Ker T$. Since both, $h$ and $T$ are $B_{\mathfrak{g}}$-symmetric, we get,
$$
B_{\mathfrak{g}}(F(x),a)=B_{\mathfrak{g}}(h(T(x)),a)=B_{\mathfrak{g}}(T(x),h(a))=B_{\mathfrak{g}}\left(x,T(h(a))\right)=0.
$$
Then, $\Im F \subset \mathfrak{a}^{\perp}$. On the other hand, take $x \in \mathfrak{a}^{\perp}$. 
It follows from \eqref{expresion para F} that, 
$x=h(T(x))=F(x)$, thus implying that $\Im F=\mathfrak{a}^{\perp}$. 
Therefore, the decomposition of
$\mathfrak{g}$ associated to the projection $F$ is
$\mathfrak{g}=\Ker T \oplus \mathfrak{a}^{\perp}$.

\smallskip
\noindent 
Let $[\cdot,\cdot]_{\mathfrak{a}^{\perp}}:\mathfrak{a}^{\perp} \times \mathfrak{a}^{\perp} \rightarrow \mathfrak{a}^{\perp}$ be the skew-symmetric bilinear map defined by $[x,y]_{\mathfrak{a}^{\perp}}=F([x,y]_{\mathfrak{g}})$. Then, $(\mathfrak{a}^{\perp},[\cdot,\cdot]_{\mathfrak{a}^{\perp}})$ is an $(n-r)$-dimensional Lie algebra. Also, let $B_{\mathfrak{a}^{\perp}}:\mathfrak{a}^{\perp} \times \mathfrak{a}^{\perp} \rightarrow \mathbb{F}$, be the symmetric bilinear form defined by $B_{\mathfrak{a}^{\perp}}(x,y)=B_{\mathfrak{g}}(T(x),y)$ for all $x,y \in \mathfrak{a}^{\perp}$. Then $(\mathfrak{a}^{\perp},[\cdot,\cdot]_{\mathfrak{a}^{\perp}},B_{\mathfrak{a}^{\perp}})$ is in fact an $(n-r)$-dimensional quadratic Lie algebra, such that the projection $\mathfrak{g} \rightarrow \mathfrak{a}^{\perp}$, is a surjective Lie algebra morphism whose kernel is $\Ker T$. Thus, the short exact sequence,
$$
0 \rightarrow \Ker T \rightarrow \mathfrak{g} \rightarrow \mathfrak{a}^{\perp} \rightarrow 0,
$$
states that $(\mathfrak{g},[\cdot,\cdot]_{\mathfrak{g}})$ is a central extension of $(\mathfrak{a}^{\perp},[\cdot,\cdot]_{\mathfrak{a}^{\perp}})$ by $\Ker T$.
\end{proof}

\noindent
\begin{rmk}\label{3rmk}
$\,$
\begin{enumerate}

\item
It follows from {\bf Lemma \ref{lema 31}.(i)-(ii)} that if the kernel $V$ of the central extension
is isotropic, then $\operatorname{dim}_{\mathbb{F}} V\le \operatorname{dim}_{\mathbb{F}}C(\mathfrak{g})$. 
This fact can be used to argue in the other direction: if $\operatorname{dim}_{\mathbb{F}} V> \operatorname{dim}_{\mathbb{F}}C(\mathfrak{g})$
then the kernel of the central extension, $V$, cannot be isotropic.

\item From {\bf Lemma \ref{lema 31}.(v)} 
we know that the subspace $\mathfrak{a}^{\perp} \subset\mathfrak{g}$
is a Lie algebra quotient of $\mathfrak{g}$, so that
the latter is a central extension of $\mathfrak{a}^{\perp}$ by $\Ker T$. We also know that
an invariant metric
$B_{\mathfrak{a}^{\perp}}:\mathfrak{a}^{\perp} \times \mathfrak{a}^{\perp} \rightarrow \mathbb{F}$,
can be defined through, 
$B_{\mathfrak{a}^{\perp}}(x,y)=B_{\mathfrak{g}}(T(x),y)$, for all $x,y \in \mathfrak{a}^{\perp}$. 
This can be done, in fact, {\it for any\/} complementary subspace $\mathfrak{p}$ to $\operatorname{Ker}T$ in $\mathfrak{g}$.
Indeed, for any pair of vectors $x,y\in\mathfrak{p}$, we write $[x,y]_{\mathfrak{g}}=[x,y]_{\mathfrak{p}}+\zeta(x,y)$,
where the skew-symmetric bilinear maps $[\cdot,\cdot]_{\mathfrak{p}}:\mathfrak{p}\times \mathfrak{p}\to \mathfrak{p}$
and $\zeta:\mathfrak{p}\times \mathfrak{p}\to \operatorname{Ker}T$ are defined by the direct sum
decomposition $\mathfrak{g}=\mathfrak{p}\oplus\operatorname{Ker}T$. The first bilinear map $[\cdot,\cdot]_{\mathfrak{p}}$, defines a Lie
algebra structure on $\mathfrak{p}$ such that $(\mathfrak{g},[\cdot,\cdot]_{\mathfrak{g}})$ becomes a central extension of $(\mathfrak{p},[\cdot,\cdot]_{\mathfrak{p}})$
by $\operatorname{Ker}T$, and the second bilinear map $\zeta$, is just the $2$-cocycle associated to this extension.
Furthermore, one may define a symmetric, bilinear form $B_{\mathfrak{p}}:\mathfrak{p}\times \mathfrak{p}\to\mathbb{F}$,
by means of $B_{\mathfrak{p}}(x,y)=B_{\mathfrak{g}}(T(x),y)$, for all $x,y \in \mathfrak{p}$, making $(\mathfrak{p},[\cdot,\cdot]_{\mathfrak{p}})$
into a quadratic Lie algebra.

\end{enumerate}
\end{rmk}
\noindent In the following Lemma we will prove the existence of an invariant metric in the central extension $(\mathfrak{G},[\cdot,\cdot]_{\mathfrak{G}})$, for which $\mathfrak{a}$ is an isotropic subspace of $\mathfrak{G}$.

\begin{lema}\label{a isotropico}
If $V$ is isotropic, there is an invariant metric $\bar{B}_{\mathfrak{G}}$ on $(\mathfrak{G},[\cdot,\cdot]_{\mathfrak{G}})$
for which $\mathfrak{a}\subset\mathfrak{G}$ is an isotropic subspace.
\end{lema}

\begin{proof}
By {\bf Lemma \ref{lema 31}.(iv)}, $\mathfrak{G}=\mathfrak{a} \oplus \Im T \oplus V$. This may be considered as a Witt decomposition for $\mathfrak{G}$. Consider the $B_{\mathfrak{G}}$-symmetric invertible centroid $Q:\mathfrak{G} \rightarrow \mathfrak{G}$  
defined by, $Q(T(x))=T(x)$, for all $x \in \mathfrak{g}$,
$Q(v)=v$, for all $v \in V$, and $Q(a)=a-\overset{r}{\underset{i=1}{\sum}}B_{\mathfrak{G}}(a,a_i)v_i$ for all $a \in \mathfrak{a}$. It is immediate to verify that $Q\in\Gamma_{B_{\mathfrak{G}}}^0(\mathfrak{G})$ and that $\mathfrak{a}$ becomes isotropic in the invariant metric $\bar{B}_{\mathfrak{G}}$ defined by $\bar{B}_{\mathfrak{G}}(x,y)=B_{\mathfrak{G}}({Q}(x),y)$, for all $x,y \in \mathfrak{G}$.
\end{proof}

\begin{rmk}\label{remark 4}
Consider the map $Q:\mathfrak{G} \rightarrow \mathfrak{G}$ just defined.
The linear maps $h$ and $k$ of {\bf Lemma \ref{lema 1}} depend on the invariant metrics $B_{\mathfrak{g}}$ and $B_{\mathfrak{G}}$. 
If $\bar{B}_{\mathfrak{G}}$ is used instead of $B_{\mathfrak{G}}$, $h$ and $k$ need to be changed for
$\bar{h}:\mathfrak{G} \rightarrow \mathfrak{g}$ and $\bar{k}:\mathfrak{g} \rightarrow \mathfrak{G}$, respectively. 
Take $x,y \in \mathfrak{g}$ and $v \in V$. By the definition of $Q$, it is clear that $Q(k(x))=k(x)$, and,
$$
B_{\mathfrak{G}}(k(x),y+v)=B_{\mathfrak{G}}(k(x),y)=B_{\mathfrak{g}}(x,y)=\bar{B}_{\mathfrak{G}}(\bar{k}(x),y+v)=B_{\mathfrak{G}}(\bar{k}(x),y+v).
$$
Therefore, $k=\bar{k}$. By {\bf Lemma \ref{lema 1}.(i)}, $\bar{h} \circ k=\operatorname{Id}_{\mathfrak{g}}=h \circ k$. The invariance of $B_{\mathfrak{G}}$ implies that the first derived ideal $[\mathfrak{G},\mathfrak{G}]_{\mathfrak{G}}$ is contained in $\Im k$.
Moreover, $\bar{h}([x,y]_{\mathfrak{G}})=h([x,y]_{\mathfrak{G}})$ for all $x,y \in \mathfrak{G}$. 
In particular, the linear map $\rho:\mathfrak{G} \rightarrow \Der\mathfrak{g} \cap \frak{o}(B_{\g})$ of {\bf Lemma \ref{lema 2}}, now satisfies $\rho(x)(y)=\bar{h}([x,y]_{\mathfrak{G}})$, for all $x \in \mathfrak{G}$ and $y \in \mathfrak{g}$. 
In other words, the replacement of $B_{\mathfrak{G}}$ by $\bar{B}_{\mathfrak{G}}$ simply amounts to change
$h$ by $\bar{h}$. Moreover, $h-\bar{h}:\mathfrak{a} \oplus \Im T \oplus V\to\mathfrak{g}$, 
might be different from zero only in the values
$(h-\bar{h})(a_i)$, $\{a_i+w_i\mid 1\le i\le r\}$ being the basis for $\Ker h=\mathfrak{g}^\perp$ (see {\bf Lemma \ref{lema 2}}).
\end{rmk}

\begin{lema}\label{q1}
Consider the linear map $\bar{h}:\mathfrak{G} \rightarrow \mathfrak{g}$, appearing in {\bf Lemma \ref{lema 1}}. If $V$ is isotropic, then $\Ker \bar{h}=\mathfrak{a}$.
\end{lema}
\begin{proof}
\noindent The fact that $\mathfrak{a}$ is an isotropic subspace 
of $\mathfrak{G}$ for the metric 
$\bar{B}_{\mathfrak{G}}$, is equivalent to the fact that $\Ker\bar{h}=\mathfrak{a}$;
{\it ie\/,} $\mathfrak{g}^{\perp}=\mathfrak{a}$ (see {\bf Lemma \ref{lema 1}.(iii)}). Indeed, if $\Ker\bar{h}=\mathfrak{a}$ then, by {\bf Lema \ref{lema 1}.(iii)}, we get $\mathfrak{g}^{\perp}=\Ker\bar{h}$, and $B_{\mathfrak{G}}(\mathfrak{a},\mathfrak{a})=\{0\}$. 

\smallskip
\noindent
On the other hand, let us assume that $\bar{B}_{\mathfrak{G}}(\mathfrak{a},\mathfrak{a})=\{0\}$. 
By {\bf Lemma \ref{lema 31}.(ii)}-{\bf (iii)}, it follows that $\bar{B}_{\mathfrak{G}}(a,T(x))=B_{\mathfrak{g}}(\bar{h}(a),T(x))=B_{\mathfrak{g}}(T\circ\bar{h}(a),x)=0$, for all $a \in \mathfrak{a}$ and for all $x \in \mathfrak{g}$, which implies $\mathfrak{a} \subseteq \mathfrak{g}^{\perp}$. 
Since $\dim_{\mathbb{F}}(\mathfrak{g}^{\perp})=r=\dim_{\mathbb{F}}(\mathfrak{a})$, it follows that $\mathfrak{g}^{\perp}=\mathfrak{a}$ ({\bf Lemma \ref{lema 1}.(iii)}-{\bf (iv)} and {\bf Lemma \ref{lema 31}}.(i)).
\end{proof}

\begin{rmk}\label{remark 5}
From now on we shall work with the invariant metric $\bar{B}_{\mathfrak{G}}$
that makes $\mathfrak{a}$ isotropic and $\mathfrak{a}=\Ker\bar{h}$. 
Therefore, from this point on, we shall write $B_{\mathfrak{G}}$ instead of
$\bar{B}_{\mathfrak{G}}$ and $h$ instead of $\bar{h}$, but always under the
assumption that $\mathfrak{a}$ is isotropic for $B_{\mathfrak{G}}$
and $\mathfrak{a}=\Ker h$.
\end{rmk}

%%%%%%%%%%%%%%%%%%%%%%%%%%%%%%%%%%%%%%%%%%%%%%%%%%%%%%%%%%%%%%%%%%%%%%%%%%%%%%%%%%%%%%%%%

\begin{lema}\label{lema 4}
Let $\Nil(\mathfrak{g})$ be the maximal nilpotent ideal of $\mathfrak{g}$.
If $V$ is an isotropic ideal, then,

\begin{itemize}

\item[(i)] $\mathfrak{g}=\Im T+\Nil(\mathfrak{g})$ and $\dim_{\mathbb{F}}\Nil(\mathfrak{g}) \geq r$.

\item[(ii)] If $\dim_{\mathbb{F}}\Nil(\mathfrak{g})=\dim_{\mathbb{F}}V=r$, then,
$$
(\mathfrak{g},[\cdot,\cdot]_{\mathfrak{g}},B_{\mathfrak{g}})=(\Im T,[\cdot,\cdot]_{\mathfrak{g}}\mid_{\Im T \times \Im T},B_{\mathfrak{g}}|_{\Im T \times \Im T}) \overset{\perp}{\oplus} (\mathfrak{a},B_{\mathfrak{g}}|_{\mathfrak{a} \times \mathfrak{a}}),
$$
where, 
$\mathfrak{a}=C(\mathfrak{g})$ and $\Im T=[\mathfrak{g},\mathfrak{g}]_{\mathfrak{g}}$. 
Furthermore, $(\Im T,[\cdot,\cdot]_{\mathfrak{g}}|_{\Im T \times \Im T})$ is a semisimple Lie algebra and $(\mathfrak{g},[\cdot,\cdot]_{\mathfrak{g}})$ 
is reductive. In addition, there exists a skew-symmetric bilinear map
$\omega:\mathfrak{a} \times \mathfrak{a} \rightarrow \mathfrak{a}^{*}$, such that,
$$
(\mathfrak{G},[\cdot,\cdot]_{\mathfrak{G}},B_{\mathfrak{G}})=(\Im T,[\cdot,\cdot]_{\mathfrak{g}}|_{\Im T \times \Im T},B_{\mathfrak{G}}|_{\Im T \times \Im T}) \overset{\perp}{\oplus} T^{*}\mathfrak{a},
$$
where $T^\ast\frak{a}\simeq \frak{a}\oplus\frak{a}^\ast$, is the cotangent bundle of $\frak{a}$,
and $\omega(x,y)(z)=\omega(y,z)(x)$.

\item[(iii)] If $(\mathfrak{g},[\cdot,\cdot]_{\mathfrak{g}},B_{\mathfrak{g}})$ is indecomposable and not simple, then $(\mathfrak{G},[\cdot,\cdot]_{\mathfrak{G}})$ is nilpotent.

\end{itemize}

\end{lema}

\begin{proof}
$\,$
\noindent (i) Let us consider the Fitting decomposition of $\mathfrak{g}=\mathfrak{q} \oplus \mathfrak{n}$ associated to $T$ (see {\bf Proposition \ref{mrd1}}), where $\mathfrak{q}=\Im T^m$ and $\mathfrak{n}=\Ker T^m$. The linear map $T$ induces an invertible transformation in the quotient $\mathfrak{g}/\mathfrak{n}$, which implies $\mathfrak{g}=\Im T+\mathfrak{n}$, but by {\bf Lemma \ref{lema 2}.(viii)}, $\mathfrak{n}$ is contained in $\Nil(\mathfrak{g})$, so the result follows. 
On the other hand, $\Ker T \subseteq C(\mathfrak{g}) \subseteq \Nil(\mathfrak{g})$ (see {\bf Lemma \ref{lema 2}.(ii)}). 
Therefore, $\dim_{\mathbb{F}}\Nil(\mathfrak{g}) \geq r$.

\medskip
\noindent
(ii) If $r=\dim_{\mathbb{F}}\Nil(\mathfrak{g})$, then (i) 
and the fact that $\dim_{\mathbb{F}}\Im T=n-r$, imply that $\mathfrak{g}=\Im T \oplus \Nil(\mathfrak{g})$. 
Since $T$ commutes with the adjoint representation, it is clear that $\Nil(\mathfrak{g})$ is invariant under $T$, so $T(\Nil(\mathfrak{g})) \subseteq \Nil(\mathfrak{g}) \cap \Im T=\{0\}$; therefore $\Nil(\mathfrak{g})=\Ker T=C(\mathfrak{g})$. Thus, $\Im T$ and $\Nil(\mathfrak{g})$ are both non-degenerates ideals of $(\mathfrak{g},[\cdot,\cdot]_{\mathfrak{g}},B_{\mathfrak{g}})$.

\smallskip
\noindent 
Now, {\bf Lemma \ref{lema 31}.(v)} says that, $\mathfrak{g}=\Ker T \oplus \mathfrak{a}^{\perp}$,
and that there is a Lie algebra structure on $\mathfrak{a}^{\perp}$, 
such that $(\mathfrak{g},[\cdot,\cdot]_{\mathfrak{g}})$ is the central extension of 
$(\mathfrak{a}^{\perp},[\cdot,\cdot]_{\mathfrak{a}^{\perp}})$ by $\Ker T$:
\begin{equation}\label{sucesionx}
0 \rightarrow \Ker T \rightarrow \mathfrak{g} \rightarrow \mathfrak{a}^{\perp} \rightarrow 0.
\end{equation}
However, we have just seen that $\Ker T=\Nil(\mathfrak{g})$ is a direct summand of $(\mathfrak{g},[\cdot,\cdot]_{\mathfrak{g}})$, 
so that the exact sequence \eqref{sucesionx} splits ({\bf Propostion \ref{escision}.(iv)}).
Therefore, $\mathfrak{a}$ and $\mathfrak{a}^{\perp}$ are both ideals of $(\mathfrak{g},[\cdot,\cdot]_{\mathfrak{g}})$. 
Since $\Ker T=C(\mathfrak{g})$, we have,
$$
\Im T={\Ker T}^\perp=C(\mathfrak{g})^\perp=[\mathfrak{g},\mathfrak{g}]_{\mathfrak{g}}.
$$
Hence, we also have 
$\mathfrak{g}=\Im T\oplus\mathfrak{a}$, with
$\mathfrak{a}=C(\mathfrak{g})=\Ker T$, and therefore, $\Im T\simeq\mathfrak{g}/\mathfrak{a}$ is a reductive Lie algebra. 
The direct sum decomposition of Lie algebra ideals, 
$\mathfrak{g}=\Im T\oplus\mathfrak{a}$, also implies that $C(\Im T)=\{0\}$. 
Therefore, $(\Im T,[\cdot,\cdot]_{\mathfrak{g}}|_{\Im T \times \Im T})$ is a semisimple Lie algebra and $(\mathfrak{g},[\cdot,\cdot]_{\mathfrak{g}})$ is reductive. 

\smallskip
\noindent 
We have proved that $\mathfrak{g}=\Ker T \oplus \Im T$. So, it is not difficult to see that this decomposition corresponds to 
the {\bf Fitting's Lemma} decomposition 
associated to $T$ (see {\bf Propostion \ref{mrd1}}). Therefore, by {\bf Propostion \ref{mrd1}.(iii),(iv),(v)}, there exists an invariant
metric $\bar{B}_{\Im T}$ on $(\Im T,[\cdot,\cdot]_{\mathfrak{g}}|_{\Im T \times \Im T})$, such that,
$(\mathfrak{G},[\cdot,\cdot]_{\mathfrak{G}},B_{\mathfrak{G}})$ is isometric to,
$$
(\Im T,[\cdot,\cdot]_{\mathfrak{g}}|_{\Im T \times \Im T},\bar{B}_{\Im T}) \overset{\perp}{\oplus} (\mathfrak{a} \oplus V,[\cdot,\cdot]_{\mathfrak{G}}|_{\mathfrak{a} \oplus V \times \mathfrak{a} \oplus V},B_{\mathfrak{G}}|_{\mathfrak{a} \oplus V \times \mathfrak{a} \oplus V}),
$$
Since $\mathfrak{a}$ and $V$ are isotropic under $B_{\mathfrak{G}}$ ({\bf Lemma \ref{isotropico}}), 
and since $B_{\mathfrak{G}}$ is non-degenerate, it follows that
$V\simeq \mathfrak{a}^\ast$. Therefore $(\mathfrak{a} \oplus V,[\cdot,\cdot]_{\mathfrak{a} \oplus V},B_{\mathfrak{G}}|_{\mathfrak{a} \oplus V \times \mathfrak{a} \oplus V})$ is the cotangent bundle extension of $\mathfrak{a}$
by $\frak{a}^\ast$, associated to the cyclic map $\omega:\mathfrak{a} \times \mathfrak{a} \rightarrow \mathfrak{a}^{*}$,
defined by $\omega(a,b)=\overset{r}{\underset{i=1}{\sum}}B_{\mathfrak{g}}(D_i(a),b)a_i^{*}$, for all $a,b \in \mathfrak{a}$.

\medskip
\noindent 
(iii) Let us consider again the Fitting decomposition $\mathfrak{g}=\mathfrak{q} \oplus \mathfrak{n}$ associated to $T$.
If $(\mathfrak{g},[\cdot,\cdot]_{\mathfrak{g}},B_{\mathfrak{g}})$ is not simple and indecomposable, we
deduce that $\mathfrak{q}=\{0\}$, since $\mathfrak{n}\neq \{0\}$. It then follows from (ii) 
that $(\mathfrak{g},[\cdot,\cdot]_{\mathfrak{g}})$ is nilpotent, and therefore, so is
$(\mathfrak{G},[\cdot,\cdot]_{\mathfrak{G}})$.
\end{proof}
\noindent 
This Lemma motivates the following characterization for reductive Lie algebras.
\noindent 
\begin{proposicion}\label{proposicion mrd2}
Let $(\mathfrak{g},[\cdot,\cdot]_{\mathfrak{g}},B_{\mathfrak{g}})$ be a quadratic Lie algebra with $\dim_{\mathbb{F}}\Nil(\mathfrak{g})=r$.
Let $V$ be an $r$-dimensional vector space and let $\mathfrak{G}=\mathfrak{g} \oplus V$ be a central extension of $(\mathfrak{g},[\cdot,\cdot]_{\mathfrak{g}})$ by $V$. Then $(\mathfrak{g},[\cdot,\cdot]_{\mathfrak{g}})$ is a reductive Lie algebra if and only if $(\mathfrak{G},[\cdot,\cdot]_{\mathfrak{G}})$ admits an invariant metric with $V$ isotropic.
\end{proposicion}
\begin{proof}

\noindent 
{\bf Lemma \ref{lema 4}.(ii)} proves the {\lq\lq}{\it if\/}{\rq\rq} part of the statement.
In order to prove the {\lq\lq}{\it only if\/}{\rq\rq} part, we shall 
assume that $(\mathfrak{g},[\cdot,\cdot]_{\mathfrak{g}})$ is a reductive Lie algebra with semisimple Levi summand $\mathfrak{s}$.
Then, $\Nil(\mathfrak{g})=\Rad(\mathfrak{g})=C(\mathfrak{g})$, and $\mathfrak{g}=\mathfrak{s} \oplus C(\mathfrak{g})$. 
Let $\{a_1,\ldots,a_r\}$ be a basis for $C(\mathfrak{g})$ and define a non-degenerate symmetric bilinear form $B_{C(\mathfrak{g})}:C(\mathfrak{g}) \times C(\mathfrak{g}) \rightarrow \mathbb{F}$, by $B_{C(\mathfrak{g})}(a_i,a_j)=\delta_{ij}$. 
Now let the invariant metric $B_{\mathfrak{g}}$ in $(\mathfrak{g},[\cdot,\cdot]_{\mathfrak{g}})$ 
be defined by $K_{\mathfrak{s}} \overset{\perp}{\oplus} B_{C(\mathfrak{g})}$, where $K_{\mathfrak{s}}$ is the 
Cartan-Killing form in the semisimple Lie algebra $(\mathfrak{s},[\cdot,\cdot]_{\mathfrak{s}})$. 

\smallskip
\noindent 
Let $\alpha_{ijk} \in \mathbb{F}$ be such that $\alpha_{ijk}=\alpha_{jki}=\alpha_{kij}$ 
and $\alpha_{ijk}=-\alpha_{ikj}$, for all $1\le i,j,k\le r$. 
For each $1\le j\le r$, we define a linear map $D_j\in\End_{\mathbb{F}} C(\mathfrak{g})$, by letting,
$D_j(a_k)=\overset{r}{\underset{i=1}{\sum}}\alpha_{ijk}\,a_i$, for each $1\le k\le r$. 
The requirements on the $\alpha$'s show that the $D_j$'s are $B_{C(\mathfrak{g})}$-skew-symmetric.

\smallskip
\noindent 
Let $V$ be an $r$-dimensional vector space, with a given basis $\{v_i\mid 1\le i\le r\}$.
Now let $(\mathfrak{G},[\cdot,\cdot]_{\mathfrak{G}})$ be the central extension of $(\mathfrak{g},[\cdot,\cdot]_{\mathfrak{g}})$ by $V$
associated to the $2$-cocycle $\theta \in Z^2(\mathfrak{g},V)$, defined by $\theta(s+a,s^\prime+a^\prime)=\overset{r}{\underset{i=1}\sum}B_{C(\mathfrak{g})}(D_i(a),a^\prime)v_i$, for all $s,s^{\prime} \in \mathfrak{s}$ and $a,a^{\prime} \in \mathfrak{a}$.
Thus,
\begin{equation}\label{reductiva1}
[s+a+u,s'+c+v]_{\mathfrak{G}}=[s,s']_{\mathfrak{s}}+\overset{r}{\underset{i=1}{\sum}}B_{C(\mathfrak{g})}(D_i(a),c)v_i.
\end{equation}
Clearly, the projection $\pi_{\mathfrak{g}}:\mathfrak{G} \rightarrow \mathfrak{g}$ is a Lie algebra morphism. 
Now, the Lie bracket $[\cdot,\cdot]_{\mathfrak{G}}$ just defined also shows that $C(\mathfrak{g})\oplus V$
is an ideal in $\mathfrak{G}$, and its restriction to $C(\mathfrak{g})\oplus V$
yields the Lie bracket,
$$
[a+u,c+v]_{C(\mathfrak{g}) \oplus V}=\sum_{i=1}^rB_{C(\mathfrak{g})}(D_i(a),c)v_i,\,\,\forall a,c \in \mathfrak{a},\,\,u,v \in V.
$$
Now, define an invariant metric $B_{C(\mathfrak{g}) \oplus V}$ on $C(\mathfrak{g}) \oplus V$
by making $C(\mathfrak{g})$ and $V$ isotropic, and letting $V$ and $C(\mathfrak{g})$ be an hyperbolic pair: $B_{C(\mathfrak{g}) \oplus V}(a_i,v_j)=\delta_{ij}$, for all $1\le i,j\le r$.
%$$
%B_{C(\mathfrak{g}) \oplus V}(a_i,v_j)=\delta_{ij},\qquad 1\le i,j\le r.
%$$
It is a straightforward matter to verify that this metric is invariant indeed. 
On the other hand, define $B_{\mathfrak{G}}$ on $\mathfrak{G}$ by means of $B_{\mathfrak{G}}=K_{\mathfrak{s}} \overset{\perp}{\oplus} B_{C(\mathfrak{g}) \oplus V}$.
%$$
%B_{\mathfrak{G}}=K_{\mathfrak{s}} \overset{\perp}{\oplus} B_{C(\mathfrak{g}) \oplus V}.
%$$
In particular, $V$ is an isotropic ideal of $(\mathfrak{G},[\cdot,\cdot]_{\mathfrak{G}},B_{\mathfrak{G}})$ and, 
$$
(\mathfrak{G},[\cdot,\cdot]_{\mathfrak{G}},B_{\mathfrak{G}})=(\mathfrak{s},[\cdot,\cdot]_{\mathfrak{s}},K_{\mathfrak{s}}) \overset{\perp}{\oplus} (C(\mathfrak{g}) \oplus V,[\cdot,\cdot]_{C(\mathfrak{g}) \oplus V},B_{C(\mathfrak{g}) \oplus V}).
$$
In fact, the structure of the quadratic Lie algebra in $C(\mathfrak{g}) \oplus V$ is that of the cotangent bundle extension of $C(\mathfrak{g})$ by $V$, where $V$ is identified with $C(\mathfrak{g})^{*}$ under the correspondence $v_k \mapsto a_k^{*}$
and $\omega:C(\mathfrak{g}) \times  C(\mathfrak{g}) \rightarrow C(\mathfrak{g})^{*}$ is defined by,
$$
\omega(a_j,a_k)=
\sum_{i=1}^rB_{C(\mathfrak{g})}(D_i(a_j),a_k)\,a_i^*
=
\sum_{i=1}^r\alpha_{kij}\,a_i^*
=
\sum_{i=1}^r\alpha_{ijk}\,a_i^*
,\,\,\, 1\le i,j,k\le r.
$$
\end{proof}

\noindent For indecomposable $2$-step nilpotent quadratic Lie algebras, we have:

\begin{proposicion}\label{proposicion 2-nilpotente}
Let $(\mathfrak{g},[\cdot,\cdot]_{\mathfrak{g}},B_{\mathfrak{g}})$ be a  $2r$-dimensional,
indecomposable, $2$-step nilpotent, quadratic Lie algebra over $\Bbb F$,
with $\mathfrak{g}=\mathfrak{a} \oplus \mathfrak{b}$, where
$\mathfrak{b}=C(\mathfrak{g})=[\mathfrak{g},\mathfrak{g}]_{\mathfrak{g}}$, 
$\dim_{\mathbb{F}}(\mathfrak{a})=\dim_{\mathbb{F}}(\mathfrak{b})=r$, and 
$\mathfrak{a}$ and $\mathfrak{b}$ are isotropic with respect to $B_{\mathfrak{g}}$. 
Let $\{a_i\,|\,i=1,..,r\}$, and $\{b_i\,|\,i=1,\ldots,r\}$ be bases for 
$\mathfrak{a}$ and $\mathfrak{b}$, respectively,
such that $B_{\mathfrak{g}}(a_i,b_j)=\delta_{ij}$. 
Let $\alpha_{ijk} \in \mathbb{F}$ be the structure constants of  $\mathfrak{g}$,
so that $[a_j,a_k]_{\mathfrak{g}}=\overset{r}{\underset{i=1}\sum}\alpha_{ijk}b_i$.
For each index $1\le j\le r$, let $A_j$ be the $r\times r$ matrix whose $(i,k)$ entry is $\alpha_{ijk}$.
Then $\mathcal{A}=\span_{\mathbb{F}}\{A_1,\ldots,A_r\}$ is a linearly independent set.
Moreover, 
there exists a central extension $(\mathfrak{G},[\cdot,\cdot]_{\mathfrak{G}})$
of $(\mathfrak{g},[\cdot,\cdot]_{\mathfrak{g}})$ by
an $r$-dimensional vector space $V$ and an invariant metric
$B_{\mathfrak{G}}:\mathfrak{G}\times\mathfrak{G}\to\mathbb{F}$
making $V$ isotropic,
if and only if there exists an invertible symmetric matrix $\mu \in \Mat_{r \times r}(\mathbb{F})$ 
for which the skew-symmetric, bilinear pairing defined through,
$$
 [A_j,A_k]_{\mathcal{A}}=A_j \mu A_k-A_k \mu A_j=\overset{r}{\underset{i=1}\sum}(\mu A_j)_{ik}A_i,
 \qquad 1\le j,k\le r,
$$
makes $(\mathcal{A},[\cdot,\cdot]_{\mathcal{A}})$ into an $r$-dimensional Lie algebra,
in which case the map $B_{\mathcal{A}}:\mathcal{A}\times \mathcal{A}\to\mathbb{F}$
given by $B_{\mathcal{A}}(A_j,A_k)=(\mu^{-1})_{jk}$, defines an invariant metric
turning $(\mathcal{A},[\cdot,\cdot]_{\mathcal{A}},B_\mathcal{A})$ into a 
perfect, indecomposable, quadratic Lie algebra.
\end{proposicion}
\begin{proof}

\noindent 
We shall first prove that
the matrices $A_1,\ldots, A_r$ are linearly independent. Assume $\overset{r}{\underset{i=1}\sum}\nu_i A_i=0$
for some scalars $\nu_i\in\mathbb{F}$, and consider 
$a=\overset{r}{\underset{i=1}\sum}\nu_i a_i \in \mathfrak{a}$.
Since $[a,a_k]_{\mathfrak{g}}=\overset{r}{\underset{i,j=1}\sum}(\nu_iA_i)_{jk}\,b_j=
0$, for each $1\le k\le r$,
it follows that $a\in C(\mathfrak{g})\cap \mathfrak{a}=\{0\}$, and hence $\nu_i=0$ for all $1\le i\le r$.

\smallskip
\noindent 
$(\Rightarrow)$ 
\noindent Suppose the central extension $\mathfrak{G}=\mathfrak{g} \oplus V$ 
admits an invariant metric $B_{\mathfrak{G}}$ that makes $V$ isotropic. 
Then, there exist linear maps $h:\mathfrak{G}\to\mathfrak{g}$ and $k:\mathfrak{g}\to\mathfrak{G}$
as in {\bf Lemma \ref{lema 1}}. Now, {\bf Lemma \ref{lema 2}.(vi)} and {\bf Lemma \ref{lema 31}.(iv)}-{\bf (v)} show that there exists a $B_{\mathfrak{g}}$-symmetric linear map $T:\mathfrak{g}\to\mathfrak{g}$ such that
$\Ker T=\mathfrak{b}=C(\mathfrak{g})$ and $\Im T=[\mathfrak{g},\mathfrak{g}]_{\mathfrak{g}}=\mathfrak{b}$, so $\Ker T=\Im T$, $T^{2}=0$ and $T|_{\mathfrak{a}}:\mathfrak{a} \rightarrow \mathfrak{b}$ is bijective.

\smallskip
\noindent 
Notice that {\bf Remark \ref{3rmk}.2} (after {\bf Lemma \ref{lema 31}}), says that $\mathfrak{a}$ has a Lie bracket $[\cdot,\cdot]_{\mathfrak{a}}:\mathfrak{a}\times\mathfrak{a}\to\mathfrak{a}$, and an invariant metric $B_{\mathfrak{a}}:\mathfrak{a}\times\mathfrak{a}\to\mathbb{F}$ with respect to it, defined by
$B_{\mathfrak{a}}(a,a^\prime)=B_{\mathfrak{g}}(T(a),a^\prime)$.
Moreover, in this case $[\cdot,\cdot]_{\mathfrak{a}}$ is identically zero and
$(\mathfrak{g},[\cdot,\cdot]_{\mathfrak{g}})$ can be viewed as a central extension of 
the Abelian Lie algebra $\mathfrak{a}$ by $\mathfrak{b}$:
$$
0 \rightarrow \mathfrak{b} \rightarrow \mathfrak{g} \rightarrow \mathfrak{a} \rightarrow 0.
$$
\noindent 
Let $\lambda \in \Mat_{r \times r}(\mathbb{F})$ be the matrix with entries $\lambda_{ij}=B_{\mathfrak{a}}(a_i,a_j)$
and let $\mu=\lambda^{-1}$.
Then, $T(a_j)=\overset{r}{\underset{i=1}\sum}\lambda_{ij}b_i$, for all $j$.

\smallskip
\noindent 
By {\bf Lemma \ref{lema 2}}, there exists a linear map $\rho:\mathfrak{G} \rightarrow \Der\mathfrak{g} \cap\frak{o}(B_{\g})$, such that $\rho(x)(y)=h([x,y]_{\mathfrak{G}})$ for all $x \in \mathfrak{G}$ and $y \in \mathfrak{g}$. Now consider the $B_{\mathfrak{g}}$-skew-symmetric derivations 
$D_j \in \Der \mathfrak{g}$, 
defined by $D_j=\rho(a_j)$, for $1 \leq j \leq r$. Let $d_{ijk},e_{ijk},f_{ijk} \in \mathbb{F}$ be the scalars such that:
$$
D_j(a_k)=\sum_{i=1}^rd_{ijk}a_i+\sum_{i=1}^r e_{ijk}b_i,
\qquad\text{and}\qquad
D_j(b_k)=\sum_{i=1}^rf_{ijk}b_i.
$$
Let $d_j$, $e_j$ and $f_j$ be the $r \times r$-matrices whose $(i,k)$-entries are $d_{ijk}$, $e_{ijk}$ and $f_{ijk}$, respectively. It is straightforward to verify that the $B_{\mathfrak{g}}$-skew-symmetry property for $D_j$ implies $f_j=-d_j^t$ and $e_j^t=-e_j$, for $1 \leq j \leq r$.
\smallskip

\noindent 
It follows from Lemma \ref{lema 2}.(v) that, $D_j \circ T=T \circ D_j=\ad_{\mathfrak{g}}(a_j)$. 
Therefore, $d_j=\mu A_j$, for $1 \leq j \leq r$. So, the matrix of $D_j$ in the basis $\{a_j,b_k\,\mid 1\le j,k\le r\}$ has the block form, $D_j=\left(\begin{smallmatrix} \mu A_j & 0\\
e_j & A_j \mu \end{smallmatrix}\right)$.
%$$
%D_j
%=\left(
%\begin{array}{cc}
%\mu A_j & 0\\
%e_j & A_j \mu
%\end{array}
%\right).
%$$
Notice that none of the $D_j$'s is inner. If one of them is, then the corresponding upper
left block $\mu A_j$ would have to vanish. That is, $A_j=0$, thus implying that $a_j\in C({\mathfrak g})
=\operatorname{Span}_{\mathbb{F}}\{b_i\mid 1\le i \le r\}$, which cannot be the case.

\smallskip
\noindent
Now, from Leibniz rule $D_j([a_k,a_l]_{\mathfrak{g}})=[D_j(a_k),a_l]_{\mathfrak{g}}+[a_k,D_j(a_l)]_{\mathfrak{g}}$, 
we get,
$$
A_j\mu A_k-A_k \mu A_j=\sum_{i,\ell=1}^r \mu_{i\ell}\alpha_{\ell jk}A_i=\sum_{i=1}^r(\mu A_j)_{ik}A_i.
$$
This expression suggests to consider the skew-symmetric bilinear map
 $[\cdot,\cdot]_{\mathcal{A}}:\mathcal{A} \times \mathcal{A} \rightarrow \mathcal{A}$, defined 
on $\mathcal{A}=\operatorname{Span}_{\mathbb{F}}\{A_i\mid 1\le i\le r\}$ by,
$$
[A_j,A_k]_{\mathcal{A}}=A_j\mu A_k-A_k \mu A_j,\,\,\,1\le j,k\le r,
$$ 
and prove that $(\mathcal{A},[\cdot,\cdot]_{\mathcal{A}})$ is in fact an $r$-dimensional Lie algebra.
Furthermore, define the bilinear form $B_{\mathcal{A}}:\mathcal{A} \times \mathcal{A} \rightarrow \mathbb{F}$, 
by $B_{\mathcal{A}}(A_j,A_k)=\lambda_{jk}$ for all $j,k$. We can prove that, $B_{\mathcal{A}}(A_j,[A_k,A_l]_{\mathcal{A}})=B_{\mathcal{A}}([A_j,A_k]_{\mathcal{A}},A_l)$, for all $1\le j,k,l\le r$.
%$$
%B_{\mathcal{A}}(A_j,[A_k,A_l]_{\mathcal{A}})=B_{\mathcal{A}}([A_j,A_k]_{\mathcal{A}},A_l),\,\,\, 1\le j,k,l\le r.
%$$
Indeed, this can be done by using the fact that the structure constants $\alpha_{ijk}$
satisfy the cyclic conditions $\alpha_{ijk}=\alpha_{kij}=-\alpha_{kji}$. This property in turn,
follows from the hypotheses on the decomposition $\mathfrak{g}=\mathfrak{a}\oplus \mathfrak{b}$,
where $\mathfrak{a}$ and $\mathfrak{b}$ are isotropic for $B_{\mathfrak{g}}$, and
$B_{\mathcal{A}}(A_j,A_k)=
\lambda_{ij}=B_{\mathfrak{a}}(a_i,a_j)=B_{\mathfrak{g}}(T(a_i),a_j)$.
Therefore $(\mathcal{A},[\cdot,\cdot]_{\mathcal{A}},B_{\mathcal{A}})$ is an 
$r$-dimensional quadratic Lie algebra. 

\smallskip
\noindent 
We shall now prove that $(\mathcal{A},[\cdot,\cdot]_{\mathcal{A}})$ has trivial center.
Since $[A,A_k]_{\mathcal{A}}=\overset{r}{\underset{i=1}\sum}(\mu A)_{ik}A_i$ for all $k$, 
it is clear that $A \in C(\mathcal{A})$ implies $\mu A=0$, and hence, $A=0$. 
This shows that the quadratic Lie algebra $(\mathcal{A},[\cdot,\cdot]_{\mathcal{A}},B_{\mathcal{A}})$ is perfect.

\smallskip
\noindent 
We shall now prove that $(\mathcal{A},[\cdot,\cdot]_{\mathcal{A}},B_{\mathcal{A}})$ is indecomposable. 
Indeed, let $\mathcal{A}_1,\mathcal{A}_2$ be non-degenerate ideals of $\mathcal{A}$, such that 
$\mathcal{A}=\mathcal{A}_1 \overset{\perp}{\oplus} \mathcal{A}_2$. 
Let $\{E_1,\ldots,E_{r_1}\}$ and $\{E_{r_1+1},\ldots,E_{r_1+r_2}\}$ be bases for 
$\mathcal{A}_1$ and $\mathcal{A}_2$, respectively, such that, $B_{\mathcal{A}}(E_j,E_k)=\delta_{jk}$. 
Let $\gamma_{ijk} \in \mathbb{F}$ be the structure constants with respect to the basis 
$\{E_1,\ldots,E_{r}\}$; {\it ie\/,}
$[E_j,E_k]_{\mathcal{A}}=\overset{r}{\underset{i=1}\sum}\gamma_{ijk}E_i$. Then, $\gamma_{ijk}=\gamma_{jki}=-\gamma_{jik}$, for all $1\le i,j,k\le r$. Since the direct sum decomposition $\mathcal{A}
=\mathcal{A}_1\oplus\mathcal{A}_2$ is orthogonal, and since
$\mathcal{A}_1$ and $\mathcal{A}_2$ are ideals, it follows that 
$\gamma_{ijk}=0$ if $j \in \{1,\ldots,r_1\}$ and $k\in \{r_1+1,\ldots,r_1+r_2\}$, for each $i$.
As we did before, we shall denote by $\gamma_j\in \Mat_{r \times r}(\mathbb{F})$
the matrix whose $(i,k)$-entry is $\gamma_{ijk}$  ($1 \leq j \leq r$).
Let $\tau=(\tau_{ij}) \in \Mat_{r \times r}(\mathbb{F})$ be the invertible matrix for which 
$E_j=\overset{r}{\underset{i=1}\sum}\tau_{ij}A_i$, for all $j$. Then, $\tau^t \lambda\, \tau=
\operatorname{Id}_{r \times r}$. On the other hand,
$$
[E_j,E_k]_{\mathcal{A}}=\sum_{l=1}^r\gamma_{ljk}E_l=\sum_{i,l=1}^r\gamma_{ljk}\tau_{il}A_i
=\sum_{i=1}^r(\tau \gamma_j)_{ik}A_i.
$$
Similarly,
$$
[E_j,E_k]_{\mathcal{A}}=\!\!\sum_{s,t=1}^r\tau_{sj}\tau_{tk}[A_s,A_t]_{\mathcal{A}}\,=\!\!\!\sum_{i,l,s,t=1}^r\tau_{sj}\tau_{tk}\mu_{il}\alpha_{lst}A_i
=\sum_{i=1}^r(\mu E_j \tau)_{ik}A_i.
$$
Then $(\lambda \tau)\gamma_j=E_j \tau$, and $\gamma_j=\tau^t E_j \tau$, for $1 \leq j \leq r$.

\smallskip
\noindent 
Let $e_j=\overset{r}{\underset{i=1}\sum}\tau_{ij}a_i \in \mathfrak{a}$ and $e'_k=\overset{r}{\underset{l=1}\sum}(\lambda\tau)_{lk}b_l \in \mathfrak{b}$. Since $\lambda$ and $\tau$ are invertible matrices, the set $\{e_j,e'_k\,|\,j,k\}$ is a basis for $(\mathfrak{g},[\cdot,\cdot]_{\mathfrak{g}},B_{\mathfrak{g}})$ satisfying:
$$
B_{\mathfrak{g}}(e_j,e'_k)=\delta_{jk},
\qquad\text{and}\qquad
[e_j,e_k]_{\mathfrak{g}}=\sum_{i=1}^r\gamma_{ijk}e'_i,\quad 1 \le j,k \le r.
$$
This implies, however, that the subspaces $I=\span_{\mathbb{F}}\{e_1,\ldots,e_{r_1},e'_1,\ldots,e'_{r_1}\}$ and $J=\span_{\mathbb{F}}\{e_{r_1+1},\ldots,e_{r_1+r_2},e'_{r_1},\ldots,e'_{r_1+r_2}\}$ are non-degenerate ideals of $(\mathfrak{g},[\cdot,\cdot]_{\mathfrak{g}})$,
such that, $\mathfrak{g}=I \overset{\perp}{\oplus} J$ with respect to the invariant metric $B_{\mathfrak{g}}$, but this contradicts
the fact that $\mathfrak g$ is indecomposable. 
Therefore, $(\mathcal{A},[\cdot,\cdot]_{\mathcal{A}},B_{\mathcal{A}})$ is an $r$-dimensional, quadratic,
indecomposable and perfect Lie algebra.
\smallskip

\noindent $(\Leftarrow)$ In order to consider central extensions of $(\mathfrak{g},[\cdot,\cdot]_{\mathfrak{g}},B_{\mathfrak{g}})$, 
{\bf Proposition \ref{proposicion 1}} says that some $B_{\mathfrak{g}}$-skew-symmetric derivations are needed. So, based on the hypotheses 
we shall produce $r$ derivations $D_1,\ldots,D_r \in \Der\mathfrak{g}$, which will be $B_{\mathfrak{g}}$-skew-symmetric. Let us consider scalars $e_{ijk} \in \mathbb{F}$ such that $e_{ijk}=e_{jki}=-e_{ikj}$, with $1\le i,j,k\le r$, and for each $j$, let $e_j \in \Mat_{r \times r}(\mathbb{F})$ be the matrix whose $(i,k)$-entry is, $e_{ijk}$. Let $D_j \in \End\mathfrak{g}$ be the linear map whose matrix
in the basis $\{a_j,b_k\,|\,j,k=1,\ldots,r\}$ is $D=\left(
\begin{smallmatrix}
\mu A_j & 0\\
e_j & A_j \mu
\end{smallmatrix}
\right)$.
%$$
%D_j=\left(
%\begin{array}{cc}
%\mu A_j & 0\\
%e_j & A_j \mu
%\end{array}
%\right).
%$$
The fact that $[\cdot,\cdot]_{\mathcal{A}}$ is a Lie bracket in $\mathcal{A}$, implies that each $D_j$ satisfies the Leibniz rule.
On the other hand, since $\mathfrak a=\operatorname{Span}_{\mathbb{F}}\{a_j\}$ and $\mathfrak b=\operatorname{Span}_{\mathbb{F}}\{b_j\}$ form a hyperbolic pair for $\mathfrak{g}$, it follows that each $D_j$ is skew-symmetric with respect to $B_{\mathfrak{g}}$.

\smallskip
\noindent
Let $V=\operatorname{Span}_{\mathbb{F}}\{v_1,\ldots,v_r\}$ be an $r$-dimensional vector space and let
$\mathfrak{G}=\mathfrak{g} \oplus V$ be the central extension of $(\mathfrak{g},[\cdot,\cdot]_{\mathfrak{g}})$ by $V$, 
associated to the $2-$cocycle $\theta \in Z^2(\mathfrak{g},V)$, given by,
$$
\theta(x,y)=\sum_{i=1}^rB_{\mathfrak{g}}(D_i(x),y)v_i,\,\,\,\forall x,y \in \mathfrak{g}.
$$
In order to construct an invariant metric $B_{\mathfrak{G}}$ in $(\mathfrak{G},[\cdot,\cdot]_{\mathfrak{G}})$, we need linear maps $h:\mathfrak{G} \rightarrow \mathfrak{g}$ and $k:\mathfrak{g} \rightarrow \mathfrak{G}$ (see Lemma \ref{lema 1}), satisfying,
\begin{itemize}
\item[(i)] $h \circ k=\operatorname{Id}_{\mathfrak{g}}$,
\item[(ii)] $B_{\mathfrak{G}}(x,y)=B_{\mathfrak{g}}(h(x),y)$ for all $x \in \mathfrak{G}$ and $y \in \mathfrak{g}$,
\item[(iii)] $B_{\mathfrak{g}}(x,y)=B_{\mathfrak{G}}(k(x),y+v)$, for all $x,y \in \mathfrak{g}$ and $v \in V$,
\item[(iv)] $\Ker h=\mathfrak{g}^{\perp}$,
\item[(v)] $\Im k=V^{\perp}$.
\end{itemize} 
Now let $T:\mathfrak{g}\to\mathfrak{b}$ be the projection operator onto $\mathfrak{b}$ defined by
$T(a_j)=\overset{r}{\underset{i=1}\sum}\lambda_{ij}b_i$ and $T(b_j)=0$ for all $1 \leq j \leq r$.
It follows that 
\begin{itemize}
\item $\Ker T=\Im T=\mathfrak{b}$, 
\item $T|_{\mathfrak{a}}:\mathfrak{a} \rightarrow \mathfrak{b}$ is bijective,
\item $T^2=0$. 
\end{itemize}
In fact, the matrix of $T$ in the given basis is $T=\left(\begin{smallmatrix}0 & 0\\
\lambda & 0 \end{smallmatrix}\right)$.
%\begin{equation}
%T=\left(
%\begin{array}{cc}
%0 & 0\\
%\lambda & 0
%\end{array}
%\right)
%\end{equation}
Since $\lambda$ is a symmetric matrix, it follows that $T$ is $B_{\mathfrak{g}}$-symmetric.
Now, notice that any linear map in $\End\mathfrak{g}$ with image in $\mathfrak{b}$, 
commutes with the adjoint representation.
Therefore, $T\in\Gamma_{B_{\mathfrak{g}}}(\mathfrak{g})$.
Now, define $k:\mathfrak{g} \rightarrow \mathfrak{G}$ by, 
$$
k(x)=T(x)+\overset{r}{\underset{i=1}\sum}B_{\mathfrak{g}}(a_i,x)v_i,\qquad\text{for all\ }\ x \in \mathfrak{g},
$$
and define $h:\mathfrak{G} \rightarrow \mathfrak{g}$ by, 
$$
h(a_j)=0,\qquad h(b_j)=\sum_{i=1}^r\mu_{ij}a_i,\qquad h(v_j)=b_j.
$$
With this data we may now define the following symmetric bilinear form $B_{\mathfrak{G}}:\mathfrak{G} \times \mathfrak{G} \rightarrow \mathbb{F}$: $B_{\mathfrak{G}}(x,y)=B_{\mathfrak{g}}(h(x),y)$, for all $x \in \mathfrak{G}$, $y \in \mathfrak{g}$, and $B_{\mathfrak{G}}(u,v)=0$, for all $u,v \in V$. It is now easy to see that $B_{\mathfrak{G}}$ is invariant and non-degenerate. Moreover, it makes $V$ isotropic, and
that it establishes a duality relationship between $\mathfrak{a}$ and $V$; {\it ie\/,} $\mathfrak{a}^*\simeq V$.
\end{proof}

\noindent 
Here is an example that illustrates the importance of the above result.

\begin{ejemplo}\label{ejemplo-correccion}
Let $(\mathfrak{g},[\cdot,\cdot]_{\mathfrak{g}},B_{\mathfrak{g}})$ be a 6-dimensional $2$-step nilpotent quadratic Lie algebra, where $\mathfrak{g}=\mathfrak{a} \oplus \mathfrak{b}$, $\mathfrak{a}=\operatorname{Span}_{\mathbb{F}}\{a_1,a_2,a_3\}$ and $\mathfrak{b}=\operatorname{Span}_{\mathbb{F}}\{b_1,b_2,b_3\}$. Let $\sigma$ be the permutation $\sigma=(1,2,3)$. The Lie bracket $[\,\cdot,\,\cdot]_{\mathfrak{g}}$ is given by $[a_i,a_{\sigma(i)}]_{\mathfrak{g}}=b_{\sigma^2(i)}$, for all $1 \le i \le 3$, and $[\mathfrak{g},\mathfrak{b}]_{\mathfrak{g}}=\{0\}$. The invariant metric $B_{\mathfrak{g}}$, is defined as $B_{\mathfrak{g}}(a_j,b_k)=\delta_{jk}$, $B_{\mathfrak{g}}(a_j,a_k)=B_{\mathfrak{g}}(b_j,b_k)=0$, for all $1 \le j,k \le 3$.
We define the $B_{\mathfrak{g}}$-skew-symmetric derivations $D_1,D_2,D_3\in\Der\mathfrak{g}$, by, $D_i(a_i)=D_i(b_i)=0$ for all $1 \le i \le 3$, $D_{i}(b_{\sigma(i)})=-D_{\sigma(i)}(b_i)=b_{\sigma^2(i)}$ and $D_i(a_{\sigma(i)})=-D_{\sigma(i)}(a_i)=a_{\sigma^2(i)}+b_{\sigma^2(i)}$ for all $1 \le i \le 3$.
The derivations $D_1,D_2,D_3$ are not inner. If one of them were inner, its image would lie in $\operatorname{Span}_{\mathbb{F}}\{b_1,b_2,b_3\}
=[\mathfrak{g},\mathfrak{g}]_{\mathfrak{g}}$, which is not the case.
\smallskip

\noindent 
Let $V=\operatorname{Span}_{\mathbb{F}}\{v_1,v_2,v_3 \}$ be a $3$-dimensional vector space. 
Let $\mathfrak{G}=\mathfrak{g} \oplus V$, and define 
$\theta:\mathfrak{g} \times \mathfrak{g} \rightarrow V$ by, $\theta(x,y)=\displaystyle{\sum_{i=1}^3B_{\mathfrak{g}}(D_i(x),y)v_i}$, for all $x,y \in \mathfrak{g}$.
This is a skew-symmetric bilinear map defined on $\mathfrak{g}$ and taking values on $V$ with which we define the following 
Lie bracket: $[\cdot,\cdot]_{\mathfrak{G}}:\mathfrak{G} \times \mathfrak{G} \rightarrow \mathfrak{G}$: $[x+u,y+v]_{\mathfrak{G}}=[x,y]_{\mathfrak{g}}+\theta(x,y)$, for all $x,y \in \mathfrak{g}$ and for all $u,v \in V$.
Thus, $(\mathfrak{G},[\cdot,\cdot]_{\mathfrak{G}})$ is a central extension of 
$(\mathfrak{g},[\cdot,\cdot]_{\mathfrak{g}})$:
$$
\begin{array}{ccccccccc}
  &             &   &                 &              &  \pi_{\mathfrak{g}} &    & \\
0 & \rightarrow & V & \hookrightarrow & \mathfrak{G} & \to & \mathfrak{g} & \to & 0
\end{array}
$$
\noindent 
The symmetric bilinear form $B_{\mathfrak{G}}:\mathfrak{G} \times \mathfrak{G} \rightarrow \mathbb{F}$, given by, $B_{\mathfrak{G}}(b_i,b_j)=B_{\mathfrak{G}}(a_i,v_j)=\delta_{ij}$, and
$B_{\mathfrak{G}}(b_i,a_j+v_k)=B_{\mathfrak{G}}(a_i,a_j)=B_{\mathfrak{G}}(v_i,v_j)=0$,
for all $1\le i,j,k\le 3$, is easily seen to be non-degenerate and invariant. 
Therefore, $(\mathfrak{G},[\cdot,\cdot]_{\mathfrak{G}},B_{\mathfrak{G}})$ is a 
$9$-dimensional, $3$-step nilpotent, quadratic Lie algebra. Clearly, $\mathfrak{G}$ has the vector space 
decomposition, $\mathfrak{G}=\mathfrak{a} \oplus \mathfrak{b} \oplus V$, where both, $\mathfrak{a}$ and $V$, are isotropic subspaces under $B_{\mathfrak{G}}$, and $\mathfrak{b}^{\perp}=\mathfrak{a}\oplus V$. %The linear maps $h:\mathfrak{G} \rightarrow \mathfrak{g}$ and $k:\mathfrak{g} \rightarrow \mathfrak{G}$ appearing in Lemma \ref{lema 1} are given by $h(v_i)=b_i$, $h(b_i)=a_i$, $h(a_i)=0$, for $i=1,2,3$, and $k(x)=T(x)+\overset{3}{\underset{i=1}{\sum}}B_{\mathfrak{g}}(a_i,x)v_i$, for all $x \in \mathfrak{g}$, respectively.
\noindent 
Write $\alpha_{ijk}$ for the structure constants of $(\mathfrak{g},[\cdot,\cdot]_{\mathfrak{g}})$; {\it ie\/,} $[a_j,a_k]_{\mathfrak{g}}=
\sum_{i=1}^3 \alpha_{ijk} b_i$, for $1 \leq j,k \leq 3$. Let us denote by $A_j$ the $3 \times 3$-matrix with entries in $\mathbb{F}$
whose $(i,k)$-entry is $\alpha_{ijk}$. That is, $A_1=\left(\begin{smallmatrix} 0 & 0 & 0\\
0 & 0 & -1\\
0 & 1 & 0 \end{smallmatrix} \right)$, $A_2=\left(\begin{smallmatrix} 0 & 0 & 1\\
0 & 0 & 0\\
-1 & 0 & 0 \end{smallmatrix} \right)$, $A_3=\left(\begin{smallmatrix} 0 & -1 & 0\\
1 & 0 & 0\\
0 & 0 & 0 \end{smallmatrix} \right)$.
It is easy to check that the $3 \times 3$ matrices $A_1,A_2,A_3$, verify $[A_i,A_{\sigma(i)}]_{\mathfrak{gl}(3)}=A_{\sigma^2(i)}$, for all $ 1 \le i \le 3$. Then, $\mathcal{A}=\operatorname{Span}_{\mathbb{F}}\{A_1,A_2,A_3\}$, is a $3$-dimensional Lie algebra. Furthermore, $B_{\mathcal{A}}:\mathcal{A} \times \mathcal{A} \to \mathbb{F}$, defined by, $B_{\mathcal{A}}(A_j,A_k)=\delta_{jk}$, yields an invariant metric on $(\mathcal{A},[\cdot,\cdot]_{\mathfrak{gl}(3)}|_{\mathcal{A} \times \mathcal{A}})$. It is thus a quadratic Lie algebra, which is furthermore isomorphic to $\mathfrak{sl}_2(\mathbb{F})$,
up to a constant factor in its Cartan-Killing metric.
\end{ejemplo}
\noindent
We would like to point out the relevance of this example.
We have exhibited a special case of one of the families constructed in the proof of {\bf Proposition \ref{proposicion 2-nilpotente}};
namely, central extensions $(\mathfrak{G},[\cdot,\cdot]_{\mathfrak{G}})$ of a $2$-step nilpotent quadratic Lie algebra $(\mathfrak{g},[\cdot,\cdot]_{\mathfrak{g}})$, by a 
$3$-dimensional vector space $V$, 
for which $0\to V\to\mathfrak{G}\to\mathfrak{g}\to 0$, does not split
---an actual consequence of the fact that the derivations $D_j$ of the $2$-cocycle
associated to the central extension, cannot be inner; that is, the $2$-cocycle is not a coboundary.
Moreover, no invariant metric $B_{\mathfrak{G}}$
can be defined on $(\mathfrak{G},[\cdot,\cdot]_{\mathfrak{G}})$ that can be identified to the orthogonal direct sum of invariant metrics on $(\mathfrak{g},[\cdot,\cdot]_{\mathfrak{g}})$ and $V$.

%%%%%%%%%%%%%%%%%%%%%%%%%%%%%%%%%%%%%%%%%%%%%%%%%%%%%%%%%%%%%%%%%%%%%%%%%%%%%%%%%%%%%
\section{Double central extensions}\label{seccion4}

\noindent 
In this section we give the details of the construction that we called in the introduction 
\textbf{double central extension} of a quadratic Lie algebra. This construction is inspired on the ideas
in \cite{Ben} that  led to the notion of {\it 
 double extensions for even quadratic Lie superalgebras\/.} 
We shall show that any central extension of a quadratic Lie algebra, 
which admits itself an invariant metric with isotropic kernel, is precisely such a double central extension.
\smallskip

\noindent Let $(\mathfrak{h},[\cdot,\cdot]_{\mathfrak{h}},B_{\mathfrak{h}})$ be a $p$-dimensional quadratic Lie algebra and let $\mathfrak{a}$ be an $r$-dimensional vector space. Our first goal is to set a quadratic Lie algebra structure on the vector space $\mathfrak{G}=\mathfrak{a} \oplus \mathfrak{h} \oplus \mathfrak{a}^{*}$ satisfying the following conditions:
\begin{itemize}

\item[(i)] $\mathfrak{h}$ is a  \emph{non-degenerate subspace} of $\mathfrak{G}$ and $\mathfrak{h}^{\perp}=\mathfrak{a} \oplus \mathfrak{a}^{*}$.

\item[(ii)]  The subspaces $\mathfrak{a}$ and $\mathfrak{a}^{*}$ form an hyperbolic pair for the quadratic form $B_{\mathfrak{G}}$.

\item[(iii)] $\mathfrak{a}^*$ is an \emph{isotropic} subspace contained in the center of $(\mathfrak{G},[\cdot,\cdot]_{\mathfrak{G}})$.

\end{itemize}

\noindent 
These conditions characterize the classical double extensions found in \cite{Kac} and \cite{Med},
except for the fact that we are requiring the special condition $\mathfrak{a}^{*} \subseteq C(\mathfrak{G})$
to be satisfied and $\mathfrak{a}^{*}$ is not necessarily a \emph{minimal} ideal of $(\mathfrak{G},[\cdot,\cdot]_{\mathfrak{G}})$.

\smallskip
\noindent 
We shall see that the Lie bracket on $\mathfrak{G}$ induces a Lie bracket $[\cdot,\cdot]_{\mathfrak{a} \oplus \mathfrak{h}}$ on the vector subspace $\mathfrak{a} \oplus \mathfrak{h}$, in such a way that $(\mathfrak{G},[\cdot,\cdot]_{\mathfrak{G}})$ is a central extension of $(\mathfrak{a} \oplus \mathfrak{h},[\cdot,\cdot]_{\mathfrak{a} \oplus \mathfrak{h}})$ by $\mathfrak{a}^{*}$. Then, our second goal is to determine necessary and sufficient conditions on $(\mathfrak{a} \oplus \mathfrak{h},[\cdot,\cdot]_{\mathfrak{a} \oplus \mathfrak{h}})$ to admit an invariant metric. This is a non-trivial problem, since we can give an example 
of a Lie algebra decomposed in the form $\mathfrak{a} \oplus \mathfrak{h} \oplus \mathfrak{a}^{*}$ satisfying the conditions (i), (ii) and (iii) above,
but $\mathfrak{a} \oplus \mathfrak{h}$ does not admit a quadratic Lie algebra structure (see {\bf Example \ref{Heisenberg}} and {\bf Theorem} \ref{teorema_c1} below).
This is the specific form in which we have addressed the question posed in the third parragraph of the introduction.

\smallskip
\noindent 
We now proceed to turn
$\mathfrak{G}=\mathfrak{a} \oplus \mathfrak{h} \oplus \mathfrak{a}^{*}$
into a quadratic Lie algebra satisfying the conditions above. 
Let $\phi:\mathfrak{a} \rightarrow \Der\mathfrak{h} \cap \frak{o}(B_{\h})$ be a linear map
and let $\Psi:\mathfrak{a} \times \mathfrak{a} \rightarrow \mathfrak{h}$ be a skew-symmetric bilinear map satisfying the following two conditions:
\begin{equation}\label{phi_ad}
[\phi(a),\phi(b)]_{\mathfrak{gl}(\mathfrak{h})}=\ad_{\mathfrak{h}}(\Psi(a,b)),\,\,\forall a,b \in \mathfrak{a},
\end{equation}
\begin{equation}\label{ciclico_phi}
\phi(a)(\Psi(b,c))+\phi(b)(\Psi(c,a))+\phi(c)(\Psi(a,b))=0,\,\,\forall a,b,c \in \mathfrak{a}.
\end{equation}
Let $\Phi:\mathfrak{h} \times \mathfrak{h} \rightarrow \mathfrak{a}^{*}$ be the skew-symmetric bilinear map given by,
\begin{equation}\label{phi}
\Phi(x,y)(a)=B_{\mathfrak{h}}(\phi(a)(x),y),\,\,\forall x,y \in \mathfrak{h},\,\,\forall a \in \mathfrak{a}.
\end{equation}
We shall also consider the bilinear map $\chi:\mathfrak{a} \times \mathfrak{h} \rightarrow \mathfrak{a}^{*}$ given by,
\begin{equation}\label{chi}
\chi(a,x)(b)=-B_{\mathfrak{h}}(\Psi(a,b),x),\,\,\forall a,b \in \mathfrak{a},\,\,\forall x \in \mathfrak{h}.
\end{equation}
Further require $\chi$ and $\Psi$ to satisfy the cyclic condition,
\begin{equation}\label{chi_Phi_sum1}
\chi(a,\Psi(b,c))+\chi(b,\Psi(c,a))+\chi(c,\Psi(a,b))=0,\,\,\forall a,b,c \in \mathfrak{a}.
\end{equation}
This obviously depends on $\Psi$, which in turn was defined in (16) through $\phi$. Thus, 
the pair $(\phi,\Psi)$ satisfying (16)-(20), together with a skew-symmetric bilinear map 
$\omega:\mathfrak{a} \times \mathfrak{a} \rightarrow \mathfrak{a}^{*}$ satisfying,
\begin{equation}\label{omega}
\omega(a,b)(c)=\omega(b,c)(a),\,\,\,\forall a,b,c \in \mathfrak{a},
\end{equation}
fits into what has been called in \cite{Ben} {\it a context for a generalized double extension\/.}
In fact, (16)-(21) are necessary and sufficient conditions to define a Lie bracket on the vector space 
$\mathfrak{G}=\mathfrak{a} \oplus \mathfrak{h} \oplus \mathfrak{a}^{*}$, through,
\begin{equation}\label{doble extension central}
\begin{array}{rcl}
\,[a,b]_{\mathfrak{G}}&=&\Psi(a,b)+\omega(a,b),\,\quad\forall a,b \in \mathfrak{a},\\
\,[a,x]_{\mathfrak{G}}&=&\phi(a)(x)+\chi(a,x),\,\,\quad\forall a,\forall x \in \mathfrak{h},\\
\,[x,y]_{\mathfrak{G}}&=&[x,y]_{\mathfrak{h}}+\Phi(x,y),\,\quad\forall x,y \in \mathfrak{h},\\
\,[\alpha,x]_{\mathfrak{G}}&=&0,\,\quad\forall \alpha \in \mathfrak{a}^{*},\,\forall x \in \mathfrak{G},
\end{array}
\end{equation}
together with an invariant metric $B_{\mathfrak{G}}$ on $(\mathfrak{G},[\cdot,\cdot]_{\mathfrak{G}})$
defined in terms of triples $(a,x,\alpha)$ and $(b,y,\beta)$ in $\mathfrak{a}\oplus\mathfrak{h}\oplus\mathfrak{a}^*$ by means of,
$$
B_{\mathfrak{G}}(a+x+\alpha,b+y+\beta)=B_{\mathfrak{h}}(x,y)+\beta(a)+\alpha(b).
$$
It is a straightforward matter to verify that the invariance condition
of the metric $B_{\mathfrak{G}}$ with respect to the Lie bracket
$[\,\cdot\,,\,\cdot\,]_{\mathfrak{G}}$ is satisfied, so that 
$(\mathfrak{G},[\cdot,\cdot]_{\mathfrak{G}},B_{\mathfrak{G}})$ is a quadratic Lie algebra.
Following \cite{Ben} we shall refer to an extension of this type as a 
\textbf{double central extension of $(\mathfrak{h},[\cdot,\cdot]_{\mathfrak{h}},B_{\mathfrak{h}})$ by $\mathfrak{a}$}.

\noindent Observe that if $\mathfrak{h}=\{0\}$, then $\mathfrak{G}=\mathfrak{a} \oplus \mathfrak{a}^{*}$ and, in this case, $(\mathfrak{G},[\cdot,\cdot]_{\mathfrak{G}},B_{\mathfrak{G}})$ is the cotangente bundle
extension defined by $\omega$ (see \cite{Bordemann}, {\bf \S3}).

\smallskip
\noindent 
The Jacobi identity for the Lie bracket $[\,\cdot\,,\,\cdot\,]_{\mathfrak{G}}$ defined in
\eqref{doble extension central}, guarantee that the following bracket defined on the vector subspace
$\mathfrak{g}=\mathfrak{a} \oplus \mathfrak{h}\subseteq \mathfrak{G}$ satisfies
the Jacobi identity itself:
\begin{eqnarray}
\,[a,b]_{\mathfrak{g}}&=&\Psi(a,b),\,\forall a,b \in \mathfrak{a}\label{new3},\\
\,[a,x]_{\mathfrak{g}}&=&\phi(a)(x),\,\forall a \in \mathfrak{a},\,x \in \mathfrak{h}\label{new2},\\
\,[x,y]_{\mathfrak{g}}&=&[x,y]_{\mathfrak{h}},\,\forall x,y \in \mathfrak{h}\label{new1}.
\end{eqnarray}
Thus, $(\mathfrak{g},[\cdot,\cdot]_{\mathfrak{g}})$ is
a Lie algebra having its first derived ideal inside $\mathfrak{h}$.
Now, if $\iota:\mathfrak{a}^{*} \rightarrow \mathfrak{G}$ is the inclusion map and 
$\pi_{\mathfrak{g}}:\mathfrak{G} \rightarrow \mathfrak{g}$ the projection map, 
we obtain the following short exact sequence of Lie algebras:
$$
\begin{array}{rcccccccl}
\,& \, & \, & \iota & \,& \pi_{\mathfrak{g}} &\,&\,& \\
0 & \rightarrow & \mathfrak{a}^{*} & \rightarrow & \mathfrak{G} & \rightarrow & \mathfrak{g} & \rightarrow & 0.
\end{array}
$$
Therefore, $(\mathfrak{G},[\cdot,\cdot]_{\mathfrak{G}})$ is a central extension of $(\mathfrak{g},[\cdot,\cdot]_{\mathfrak{g}})$ by $\mathfrak{a}^{*}$. We observe that $(\mathfrak{G},[\cdot,\cdot]_{\mathfrak{G}})$ admits an invariant metric but we do not know if $(\mathfrak{g},[\cdot,\cdot]_{\mathfrak{g}})$ does.
The next Proposition provides necessary and sufficient conditions for $(\mathfrak{g},[\cdot,\cdot]_{\mathfrak{g}})$ 
to admit an invariant metric ---a question not addressed in previous works.
\begin{proposicion}\label{prop_c1}
Let $(\mathfrak{G},[\cdot,\cdot]_{\mathfrak{G}},B_{\mathfrak{G}})$ be a double central extension of $(\mathfrak{h},[\cdot,\cdot]_{\mathfrak{h}},B_{\mathfrak{h}})$ by $\mathfrak{a}$ and $(\mathfrak{g}=\mathfrak{a} \oplus \mathfrak{h},[\cdot,\cdot]_{\mathfrak{g}})$ be the Lie algebra defined by \eqref{new3}, \eqref{new2} and \eqref{new1}. Then $(\mathfrak{g},[\cdot,\cdot]_{\mathfrak{g}})$ admits an invariant metric if and only if there exists a linear map $L:\mathfrak{G} \rightarrow \mathfrak{G}$ such that:
\begin{itemize}

\item[(i)] $\Ker L=\mathfrak{a}^{*}$.

\item[(ii)] $L|_{\mathfrak{g}}$ is injective.

\item[(iii)] $L \in \Gamma_{B_{\mathfrak{G}}}(\mathfrak{G})$.

\end{itemize}

\end{proposicion}
\begin{proof}

\noindent 
$(\Rightarrow)$ Suppose that $(\mathfrak{g},[\cdot,\cdot]_{\mathfrak{g}})$ admits 
an invariant metric $B_{\mathfrak{g}}$. Let $\iota: \mathfrak{g} \rightarrow
\mathfrak{G}$ be the inclusion map and let $\pi_{\mathfrak{g}}:\mathfrak{G} \rightarrow \mathfrak{g}$ be
the projection onto $\mathfrak{g}$. By {\bf Lemma \ref{lema 1}}, there are linear maps $h:\mathfrak{G} \rightarrow \mathfrak{g}$ and $k:\mathfrak{g} \rightarrow \mathfrak{G}$ ---associated to $\iota$ and $\pi_{\mathfrak{g}}$--- such that, $h \circ k=\operatorname{Id}_{\mathfrak{g}}$, so that $h$ is surjective and $k$ injective. Also, by {\bf Lemma \ref{lema 1}.(v)}, $[k(x),y]_{\mathfrak{G}}=k([x,y]_{\mathfrak{g}})$ for all $x,y \in \mathfrak{g}$. 

\smallskip
\noindent 
As in {\bf Proposition \ref{mrd1}}, we extend the linear map $k:\mathfrak{g} \rightarrow \mathfrak{G}$ to the whole vector space $\mathfrak{G}$ by means of $L(x+\alpha)=k(x)$ for all $x \in \mathfrak{g}$, and for all $\alpha \in \mathfrak{a}^{*}$. Then $\mathfrak{a}^{*} \subseteq \Ker L$.
Since $L|_{\mathfrak{g}}=k$, and $k$ is injective, it follows that $\Ker L=\mathfrak{a}^{*}$. It remains to show that $L \in \Gamma_{B_{\mathfrak{G}}}(\mathfrak{G})$ but this is proved exactly as in the proof of {\bf Proposition \ref{mrd1}}.

\smallskip
\noindent 
$(\Leftarrow)$ Suppose there is a linear map $L:\mathfrak{G}\to\mathfrak{G}$ satisfying the conditions of the statement. 
Let $k=L|_{\mathfrak{g}}$, and define $B_{\mathfrak{g}}:\mathfrak{g} \times \mathfrak{g} \rightarrow \mathbb{F}$ by $B_{\mathfrak{g}}(x,y)=B_{\mathfrak{G}}(k(x),y)$ for all $x,y \in \mathfrak{g}$. Using (i), (ii) and (iii), it is straightforward to see that $B_{\mathfrak{g}}$ is an invariant metric in $(\mathfrak{g},[\cdot,\cdot]_{\mathfrak{g}})$.
\end{proof}
\noindent 
Now, the double central extension of $(\mathfrak{h},[\cdot,\cdot]_{\mathfrak{h}},B_{\mathfrak{h}})$ by $\mathfrak{a}$ is, by definition, the central extension of $(\mathfrak{a} \oplus \mathfrak{h},[\cdot,\cdot]_{\mathfrak{a} \oplus \mathfrak{h}})$ by $\mathfrak{a}^{*}$, where the Lie bracket $[\cdot,\cdot]_{\mathfrak{a} \oplus \mathfrak{h}}$ is defined by the expressions \eqref{new3}, \eqref{new2} and \eqref{new1} in terms of the data $(\phi,\Psi)$. If at the same time, the hypotheses of {\bf Proposition \ref{prop_c1}} are satisfied, then $(\mathfrak{a} \oplus \mathfrak{h},[\cdot,\cdot]_{\mathfrak{a} \oplus \mathfrak{h}})$ admits an invariant metric. These hypotheses, however, are far from being trivially satisfied as the following example shows how $(\mathfrak{a} \oplus \mathfrak{h},[\cdot,\cdot]_{\mathfrak{a} \oplus \mathfrak{h}})$ might not have any invariant metric defined on it.

\smallskip
\noindent
\begin{ejemplo}\label{Heisenberg}
Let $(\mathfrak{h}_m=V_m \oplus \mathbb{F}\hslash,[\cdot,\cdot]_{\mathfrak{h}_m})$ be the $(2m+1)$-dimensional Heisenberg Lie algebra, with skew-symmetric, non-degenerate, bilinear form $\omega:V_m \times V_m \rightarrow \mathbb{F}$, defined by $[x,y]_{\mathfrak{h}_m}=\omega(x,y)\hslash$ for all $x,y \in V_m$, and $C(\mathfrak{h}_m)=\mathbb{F}\hslash$. It is known that every invariant bilinear form defined on $(\mathfrak{h}_m,[\cdot,\cdot]_{\mathfrak{h}_m})$ degenerates in $\hslash$, which means that $(\mathfrak{h}_m,[\cdot,\cdot]_{\mathfrak{h}_m})$ does not admit invariant metrics. However, the Heisenberg Lie algebra $(\mathfrak{h}_m,[\cdot,\cdot]_{\mathfrak{h}_m})$ can be extended to a Lie algebra that does admit an invariant metric. Such an extension is defined in terms of  a derivation $D \in \Der\mathfrak{h}_m$ satisfying the following properties (see \cite{Mac}):
$\Ker D=\mathbb{F}\hslash$, $D(V_m) \subset V_m$ and $\omega(D(x),y)=-\omega(x,D(y))$ for all $x,y \in V_m$. Now, let us consider the semi-direct product of $(\mathfrak{h}_m,[\cdot,\cdot]_{\mathfrak{h}_m})$ and $D$ and denote it by $\mathfrak{h}_m[D]$. Then $\mathfrak{h}_m[D]$ is a double central extension of the Abelian Lie algebra $V_m$ by $D$ (see \cite{Mac}).

\noindent Now consider the Lie bracket $[\cdot,\cdot]_{\mathfrak{g}}$ defined on $\mathfrak{g}=\mathbb{F}D \oplus V_m$ by,
$$
[x+\lambda D,y+\mu D]_{\mathfrak{g}}=\lambda D(y)-\mu D(x),\,\,\forall x,y \in V_m,\,\,\forall \lambda,\mu \in \mathbb{F}.
$$
Then $\mathfrak{h}_m[D]=\mathfrak{g} \oplus \mathbb{F}\hslash$, and make
it the central extension of $(\mathfrak{g},[\cdot,\cdot]_{\mathfrak{g}})$ by $\mathbb{F}\hslash$,
$$
0 \rightarrow \mathbb{F}\hslash \rightarrow\mathfrak{h}_m[D] \rightarrow \mathfrak{g} \rightarrow 0,
$$
associated to the $2-$cocycle $\theta:\mathfrak{g} \times \mathfrak{g} \rightarrow \mathbb{F}\hslash$ 
given by $\theta(x+\lambda D,y+\mu D)=\omega(x,y) \hslash$, for all $x,y \in V_m$ and for all $\lambda,\mu \in \mathbb{F}$. 
Note that if $(\mathfrak{g},[\cdot,\cdot]_{\mathfrak{g}})$ admits an invariant metric $B_{\mathfrak{g}}$, then $B_{\mathfrak{g}}(x,D(y))=B_{\g}(x,[D,y]_{\g})=-B_{\g}([x,y]_{\g},D)=0$, and $B_{\g}(x,D)=B_{\g}(D(D^{-1})(x),D)=B_{\g}([D,D^{-1}(x)]_{\g},D)=-B_{\g}(D^{-1}(x),[D,D]_{\g})=0$, for all $x,y \in V_m$, which implies that $B_{\mathfrak{g}}$  degenerates in $V_m$. This shows that there are double central extensions which cannot be central extensions of a quadratic Lie algebra.
\end{ejemplo}

\noindent
We can now state the following:
\begin{teorema}\label{teorema_c1}
Let $(\mathfrak{g},[\cdot,\cdot]_{\mathfrak{g}},B_{\mathfrak{g}})$ be a $n$-dimensional quadratic Lie algebra, $V$ be an $r$-dimensional vector space and $(\mathfrak{G},[\cdot,\cdot]_{\mathfrak{G}})$ be a central extension of $(\mathfrak{g},[\cdot,\cdot]_{\mathfrak{g}})$ by $V$. Suppose there exists an invariant metric $B_{\mathfrak{G}}$ on $(\mathfrak{G},[\cdot,\cdot]_{\mathfrak{G}})$.
\begin{itemize}

\item[(i)] If $V$ is non degenerate, there exists an invariant metric $\bar{B}_{\mathfrak{g}}$ on $(\mathfrak{g},[\cdot,\cdot]_{\mathfrak{g}})$ such that $(\mathfrak{G},[\cdot,\cdot]_{\mathfrak{G}},B_{\mathfrak{G}})$ is isometric to $(\mathfrak{g},[\cdot,\cdot]_{\mathfrak{g}},\bar{B}_{\mathfrak{g}}) \overset{\perp}{\oplus} (V,B_{\mathfrak{G}}|_{V \times V})$.

\item[(ii)] If $V$ is isotropic then 

\begin{itemize}

\item[$\bullet$] $r \leq \dim_{\mathbb{F}}C(\mathfrak{g})$,

\item[$\bullet$] there is an $(n-r)$-dimensional quadratic Lie algebra $(\mathfrak{h},[\cdot,\cdot]_{\mathfrak{h}},B_{\mathfrak{h}})$ 
and an $r$-dimensional vector space $\mathfrak{a}$, such that $\mathfrak{G}=\mathfrak{a} \oplus \mathfrak{h} \oplus V$ and $(\mathfrak{G},[\cdot,\cdot]_{\mathfrak{G}},B_{\mathfrak{G}})$ is a double central extension of $(\mathfrak{h},[\cdot,\cdot]_{\mathfrak{h}},B_{\mathfrak{h}})$ by $\mathfrak{a}$, where $\mathfrak{a}^{*}$ and $V$ are isomorphic vector spaces.  

\item[$\bullet$] There exists a quadratic Lie algebra structure on the vector space $\mathfrak{a} \oplus \mathfrak{h}$ such that $(\mathfrak{g},[\cdot,\cdot]_{\mathfrak{g}},B_{\mathfrak{g}})=(\mathfrak{a} \oplus \mathfrak{h},[\cdot,\cdot]_{\mathfrak{a} \oplus \mathfrak{h}},B_{\mathfrak{a} \oplus \mathfrak{h}})$ and $[\mathfrak{g},\mathfrak{g}]_{\mathfrak{g}} \subseteq \mathfrak{h}$.

\end{itemize}

\end{itemize}

\end{teorema}
\begin{proof}
 {\bf Proposition \ref{teorema de estructura 2}} proves the case when $V$ is non-degenerate. It only remains to prove the following statement that deals with the case when $V$ is an isotropic ideal.
\end{proof}

\medskip
\noindent
\begin{lema}\label{isotropico}
Let the hypotheses be as in {\bf Theorem 1} above. 
If $V$ is an isotropic ideal in $\mathfrak{G}$, then $(\mathfrak{G},[\cdot,\cdot]_{\mathfrak{G}},B_{\mathfrak{G}})$ is isometric to a double central extension of an $(n-r)$-dimensional quadratic Lie algebra.
\end{lema}
\begin{proof}
\noindent 
Let $T\in\End_{\mathbb{F}}{\mathfrak{G}}$ be as in the hypothesis of {\bf Lemma \ref{lema 2}}. By {\bf Lemma \ref{lema 31}.(iv)},
we have the following vector space decomposition: $\mathfrak{G}=\mathfrak{a} \oplus \Im T \oplus V$, where $\mathfrak{g}=\mathfrak{a} \oplus \Im T$ and $\mathfrak{a}$ is an isotropic subspace of $\mathfrak{G}$ (see {\bf Lemma \ref{a isotropico}}). 

\smallskip
\noindent 
We need to find maps $\phi$, $\Psi$, $\chi$ and $\omega$ satisfying \eqref{phi_ad}-\eqref{omega}. 
Now, write $\mathfrak{h}=\Im T$. 
From {\bf Lemma \ref{lema 31}.(iii)}, it is not difficult to conclude that
$\mathfrak{h}$ is a non-degenerate vector subspace of $\mathfrak{G}$, and that $\mathfrak{h}^{\perp}=\mathfrak{a} \oplus V$. 
Since $\Im T$ is an ideal in $\mathfrak{g}$, we may
define a Lie bracket $[\cdot,\cdot]_{\mathfrak{h}}$ on $\mathfrak{h}$ 
by letting $[x,y]_{\mathfrak{h}}=[x,y]_{\mathfrak{g}}$ for all $x,y \in \mathfrak{h}$. 
Let $B_{\mathfrak{h}}$ be the restriction $B_{\mathfrak{G}}|_{\mathfrak{h} \times \mathfrak{h}}$.
From {\bf Lemma \ref{lema 31}.(iii)} one can easily show that $B_{\mathfrak{h}}$ is invariant. Indeed,
for any $x^\prime=T(x)$, $y^\prime=T(y)$ and $z^\prime=T(z)\in \mathfrak{h}$, we have,
$$
\aligned
B_{\mathfrak{h}}([T(x),T(y)]_{\mathfrak{h}}, T(z)) & =
B_{\mathfrak{G}}([T(x),T(y)]_{\mathfrak{g}}, T(z)) = B_{\mathfrak{g}}({[T(x),T(y)]_{\mathfrak{g}}},h\circ T(z))\\
& = B_{\mathfrak{g}}(T{[x,T(y)]_{\mathfrak{g}}},h\circ T(z)) = B_{\mathfrak{g}}({[x,T(y)]_{\mathfrak{g}}},T\circ h\circ T(z))\\
& = B_{\mathfrak{g}}({[x,T(y)]_{\mathfrak{g}}},T(z)) = B_{\mathfrak{g}}(x,[T(y),T(z)]_{\mathfrak{g}})\\
& = B_{\mathfrak{G}}(k(x),[T(y),T(z)]_{\mathfrak{g}})\\
& = B_{\mathfrak{G}}(T(x)+\sum_{i=1}^rB_{\mathfrak{g}}(a_i,x)v_i,[T(y),T(z)]_{\mathfrak{g}})\\
& = B_{\mathfrak{h}}(T(x),[T(y),T(z)]_{\mathfrak{h}})
\endaligned
$$
where in the last step we used the fact that $V$ is isotropic, as follows:
$$
\aligned
B_{\mathfrak{G}}(v,[T(y),T(z)]_{\mathfrak{G}})&=
B_{\mathfrak{G}}(v,[T(y),T(z)]_{\mathfrak{g}}+\sum_{i=1}^rB_{\mathfrak{g}}(D_i(T(y)),T(z)])v_i)\\
&=
B_{\mathfrak{G}}(v,[T(y),T(z)]_{\mathfrak{G}})=0.
\endaligned
$$
Thus,
$(\mathfrak{h},[\cdot,\cdot]_{\mathfrak{h}},B_{\mathfrak{h}})$ is an $n-r$-dimensional quadratic Lie algebra.

\smallskip
\noindent 
We now define a linear map $\phi:\mathfrak{a} \rightarrow \Der\mathfrak{h} \cap \frak{o}(B_{\h})$ 
by $\phi(a)(x)=[a,x]_{\mathfrak{g}}$ for all $a \in \mathfrak{a}$ and for all $x \in \mathfrak{h}$. 
It follows from the invariance of $B_{\mathfrak{h}}$, that $\phi(a)$ is
$B_{\mathfrak{h}}$-skew-symmetric.
Since $\mathfrak{h}$ is an ideal of $(\mathfrak{g},[\cdot,\cdot]_{\mathfrak{g}})$, the linear map $\phi$ is well defined.
We now define a skew-symmetric bilinear map $\Psi:\mathfrak{a} \times \mathfrak{a} \rightarrow \mathfrak{h}$ by $\Psi(a,b)=[a,b]_{\mathfrak{g}}$ for all $a,b \in \mathfrak{a}$. Note that by {\bf Lemma \ref{lema 31}.(iv)}, the linear map $\Psi$ is well defined. In addition, the Jacobi identity in $\mathfrak{g}$, implies that $[\phi(a),\phi(b)]_{\mathfrak{gl}(\mathfrak{h})}=\ad_{\mathfrak{h}}(\Psi(a,b))$ and $\phi(a) \Psi(b,c)+\phi(b) \Psi(c,a)+\phi(c) \Psi(a,b)=0$, for all $a,b,c \in \mathfrak{a}$, thus satisfying \eqref{phi_ad} and \eqref{ciclico_phi}.

\smallskip
\noindent 
In order to identify $V$ with $\mathfrak{a}^*$ define the correspondence
$V\ni v\mapsto \xi(v)=B_{\mathfrak{G}}(\cdot,v)\vert_{\mathfrak{a}}\in \mathfrak{a}^*$.
Since $B_{\mathfrak{G}}$ is non-degenerate, $\xi$ defines a bijective linear map.
Besides, {\bf Lemma \ref{lema 2}.(i)} says that  $B_{\mathfrak{G}}(a_i,v_j)=\delta_{ij}$ for $1\le i,j\le r$,
which makes the identification $\xi(v_i)\leftrightarrow a^{*}_i$, for $1 \leq i \leq r$. 

\smallskip
\noindent 
Now use the maps $D_i\in\operatorname{Der}\mathfrak{g}$
defined by the $2$-cocycle of the extension of $\mathfrak{g}$ by $V$
to define the skew-symmetric bilinear map $\Phi:\mathfrak{h} \times \mathfrak{h} \rightarrow \mathfrak{a}^{*}$, 
by $\Phi(x,y)=\overset{r}{\underset{i=1}{\sum}}B_{\mathfrak{g}}(D_i(x),y)a^{*}_i$, for all $x,y \in \mathfrak{h}$.
We want to prove that equation \eqref{phi} is satisfied. Indeed,
$$
\aligned
\Phi(x,y)(a_j)&=\sum_{i=1}^rB_{\mathfrak{g}}(D_i(x),y)a^{*}_i(a_j)=B_{\mathfrak{g}}(D_j(x),y)=B_{\mathfrak{g}}( h([a_j,x]_{\mathfrak{G}}),y)\\
&=B_{\mathfrak{G}}( [a_j,x]_{\mathfrak{G}},y)=B_{\mathfrak{G}}( [a_j,x]_{\mathfrak{g}}+\sum_{i=i}^rB_{\mathfrak{g}}(D_i(a_j),x)v_i,y)\\
&=B_{\mathfrak{G}}( [a_j,x]_{\mathfrak{g}},y)=B_{\mathfrak{G}}( \phi(a_j)(x),y).
\endaligned
$$
Since both $\phi(a_j)(x)$ and $y$ belong to $\mathfrak{h}=\operatorname{Im}T$, it follows
from the definition of $B_{\mathfrak{h}}$ that $\Phi(x,y)(a_j)=B_{\mathfrak{h}}(\phi(a_j)(x),y)$, for all $x,y \in \mathfrak{h}$.
\smallskip

\noindent 
Now let $\chi:\mathfrak{a} \times \mathfrak{h} \rightarrow \mathfrak{a}^{*}$ be defined by 
$\chi(a,x)=\overset{r}{\underset{i=1}\sum}B_{\mathfrak{g}}(D_i(a),x)a_i^{*}$, for all $a \in \mathfrak{a}$, and all $x \in \mathfrak{h}$.
After a similar computation we may conclude that $\chi(a,x)(a_j)=-B_{\mathfrak{h}}(\Psi(a,a_j),x)$ for all $a \in \mathfrak{a}$, $x \in \mathfrak{h}$ and $1 \leq j \leq r$. Thus \eqref{chi} is also satisfied.

\smallskip
\noindent 
Finally, define the skew-symmetric bilinear map 
$\omega:\mathfrak{a} \times \mathfrak{a} \rightarrow \mathfrak{a}^{*}$ by letting,
$\omega(a,b)=\overset{r}{\underset{i=1}{\sum}}B_{\mathfrak{g}}(D_i(a),b)a_i^{*}$, for all $a,b \in \mathfrak{a}$.
It follows easily from {\bf Lemma \ref{lema 2}.(iv)} that, $\omega(a_j,a_k)(a_i)=\omega(a_k,a_i)(a_j)$, for all $1 \leq i,j,k \leq r$. 
Since the $D_i$'s are $B_{\mathfrak{g}}$-skew-symmetric derivations of $\mathfrak{g}$, 
Leibniz rule and Jacobi identity imply that 
$\chi(a,\Psi(b,c))+\chi(b,\Psi(c,a))+\chi(c,\Psi(a,b))=0$ for all $a,b,c \in \mathfrak{a}$,
and therefore \eqref{chi_Phi_sum1} is also satisfied.

\smallskip
\noindent 
Now let $\mathfrak{H}=\mathfrak{a} \oplus \mathfrak{h} \oplus \mathfrak{a}^{*}$
be the double central extension of $(\mathfrak{h},[\cdot,\cdot]_{\mathfrak{h}},B_{\mathfrak{h}})$ by $\mathfrak{a}$.
Thus, the Lie bracket $[\cdot,\cdot]_{\mathfrak{H}}$ in $\mathfrak{H}$ is given
by means of \eqref{doble extension central}, and the invariant metric on $\mathfrak{H}$ is given
by $B_{\mathfrak{H}}(a+x+\alpha,b+y+\beta)=B_{\mathfrak{h}}(x,y)+\beta(a)+\alpha(b)$, for all $a,b \in \mathfrak{a}$, $x,y \in \mathfrak{h}$ and $\alpha,\beta \in \mathfrak{a}^{*}$.

\smallskip
\noindent 
Finally, the map $\Omega:\mathfrak{G} \rightarrow \mathfrak{H}$, defined by $\Omega(a+x+v)=a+x+\xi(v)$,
gives the desired isometry between $(\mathfrak{G},[\cdot,\cdot]_{\mathfrak{G}},B_{\mathfrak{G}})$ and 
the double central extension of $(\mathfrak{h},[\cdot,\cdot]_{\mathfrak{h}},B_{\mathfrak{h}})$ by $\mathfrak{a}$,
as claimed. And to round it up, notice that the linear map $L:\mathfrak{G} \rightarrow \mathfrak{G}$, 
given by, $L(x+u)=k(x)$ for all $x \in \mathfrak{g}$ and $u \in V$, satisfies the conditions of {\bf Proposition \ref{prop_c1}}.
\end{proof}

\noindent 
The next {\bf Corollary} of {\bf Theorem \ref{teorema_c1}} addresses the question of when $V$ is a non-degenerate ideal
of $(\mathfrak{G},[\cdot,\cdot]_{\mathfrak{G}},B_{\mathfrak{G}})$ under the special
hypothesis that $(\mathfrak{g},[\cdot,\cdot]_{\mathfrak{g}},B_{\mathfrak{g}})$
be indecomposable and non-nilpotent.
\begin{corolario}
Let $(\mathfrak{g},[\cdot,\cdot]_{\mathfrak{g}},B_{\mathfrak{g}})$ be an indecomposable non-nilpotent quadratic Lie algebra. Let $V$ be a vector space and let $(\mathfrak{G},[\cdot,\cdot]_{\mathfrak{G}})$ be a central extension of $(\mathfrak{g},[\cdot,\cdot]_{\mathfrak{g}})$ by $V$. If there is an invariant metric on $(\mathfrak{G},[\cdot,\cdot]_{\mathfrak{G}})$, then $V$ is a non-degenerate ideal of $(\mathfrak{G},[\cdot,\cdot]_{\mathfrak{G}},B_{\mathfrak{G}})$.
\end{corolario}
\begin{proof}
Let $B_{\mathfrak{G}}$ be an invariant metric on $(\mathfrak{G},[\cdot,\cdot]_{\mathfrak{G}})$. If $V$ is isotropic, {\bf Lemma \ref{lema 4}.(iii)} implies that $(\mathfrak{g},[\cdot,\cdot]_{\mathfrak{g}})$ is nilpotent. On the other hand, if $V \cap V^{\perp} \neq \{0\}$, {\bf Proposition \ref{proposicion reduccion}} states that there exists a quadratic Lie algebra $(\mathfrak{H},[\cdot,\cdot]_{\mathfrak{H}},B_{\mathfrak{H}})$
with $B_{\mathfrak{H}}=B_{\mathfrak{G}}|_{\mathfrak{H} \times \mathfrak{H}}$, 
 such that $(\mathfrak{G},[\cdot,\cdot]_{\mathfrak{G}},B_{\mathfrak{G}})$ is the orthogonal direct sum of $(\mathfrak{H},[\cdot,\cdot]_{\mathfrak{G}},B_{\mathfrak{G}}|_{\mathfrak{H} \times \mathfrak{H}})$ and $(U,B_{\mathfrak{G}}|_{U \times U})$, where $V=(V \cap V^{\perp}) \oplus U$, and $U^\perp=\mathfrak{H}$. Also, $(\mathfrak{H},[\cdot,\cdot]_{\mathfrak{G}},B_{\mathfrak{G}}|_{\mathfrak{H} \times \mathfrak{H}})$ is a central extension of $(\mathfrak{g},[\cdot,\cdot]_{\mathfrak{g}},B_{\mathfrak{g}})$ with isotropic kernel $V \cap V^{\perp}$. Thus, {\bf Lemma \ref{lema 4}.(iii)} says again that $(\mathfrak{g},[\cdot,\cdot]_{\mathfrak{g}})$ is nilpotent. 
\end{proof}
\noindent
Notice that if the indecomposable quadratic Lie algebra $(\mathfrak{g},[\cdot,\cdot]_{\mathfrak{g}})$ is nilpotent, the example exhibited in \S {\bf 2}, shows that the invariant
metric defined on the central extension $(\mathfrak{G},[\cdot,\cdot]_{\mathfrak{G}})$, cannot be obtained as the orthogonal direct sum of invariant metrics on $(\mathfrak{g},[\cdot,\cdot]_{\mathfrak{g}})$ and $V$.

\subsection{Central extensions of dimension one or two}

\noindent 
In this section we shall prove that within the context of quadratic Lie algebras,
any central extension $\mathfrak{G}$ of $\mathfrak{g}$ by $V$ 
having $\dim_{\mathbb{F}}V\le 2$, yields a Lie algebra monorphism 
$\mathfrak{g}\hookrightarrow\mathfrak{G}$, such that the composition 
$$
\begin{array}{ccccccc}
& & &  & \pi_{\mathfrak{g}} & &\\
&\mathfrak{g}&\hookrightarrow &\mathfrak{G} & \rightarrow & \mathfrak{g}, &
\end{array}
$$
is the identity map $\operatorname{Id}_{\mathfrak{g}}$. That is, the corresponding exact sequence $0\to V\to \mathfrak{G}\to\mathfrak{g}\to 0$, splits.
This amounts to show, as we shall see, that the derivations of the $2$-cocycle of the central extension,
are inner. On the other hand, the example {\bf \ref{ejemplo-correccion}} exhibited in {\bf \S 3} shows
that this is no longer true as soon as $\dim_{\mathbb{F}}V\ge 3$.

\subsubsection{One-dimensional central extensions}

\noindent Let $V$ be a one-dimensional vector space; say, $V=\mathbb{F}{v}$.
Let $D \in \Der\mathfrak{g}$ be the $B_{\mathfrak{g}}$-skew-symmetric derivation associated to the $2$-cocycle of the central extension. 
{\bf Lemma \ref{lema 2}} implies that  there is a $T \in \Gamma_{B_\mathfrak{g}}(\mathfrak{g})$ and a vector $a:=a_1 \in \Ker D$, such that,
$T \circ D= D\circ T=\ad_{\mathfrak{g}}(a)$, in which case, $\Ker D=\{x \in
\mathfrak{g}\,\mid\,[a,x]_{\mathfrak{G}}=0\}$ (see {\bf Lemma \ref{lema 2}.(v)}). It is easy to see, however, that $\Ker D= C_{\mathfrak{g}}(a)$. 
Indeed, write  $[a,x]_{\mathfrak{G}}=[a,x]_{\mathfrak{g}}+B_{\mathfrak{g}}(D(a),x)\,v$.
For $a \in \Ker D$, it follows that $[a,x]_{\mathfrak{G}}=[a,x]_{\mathfrak{g}}$.

\smallskip
\noindent
Now, if $T=0$ then $\mathfrak{g}$ is one-dimensional, 
because $k(\mathfrak{g}) \subseteq V$ and $k$ is injective (see {\bf Lemma \ref{lema 1}.(ii)}). 
But then, both, $[\cdot,\cdot]_{\mathfrak{g}}$ and $[\cdot,\cdot]_{\mathfrak{G}}$ 
are trivial and $D=0$.
Thus, assume that $T \neq 0$. Applying the argument given in {\bf Proposition \ref{lema 2}.(ix)}, we prove that $D$ is an inner derivation

\subsubsection{Two-dimensional central extensions}

\noindent 
We shall now assume that $\dim_{\mathbb{F}}V=2$.
Recall from {\bf Remark \ref{3rmk}.2} (after {\bf Lemma \ref{lema 31}}) that we may assume $\dim_{\mathbb{F}}V\le \dim_{\mathbb{F}}C(\mathfrak{g})$.
In particular, $\dim_{\mathbb{F}}\mathfrak{g}\ge 2$. We shall now prove
by induction on $\dim_{\mathbb{F}}\mathfrak{g}$ that the
derivations $D_1$ and $D_2$ associated to the $2$-cocycle of the central extension are inner.
\smallskip

\noindent Notice that we only have to analyze the case when $V$ is isotropic.
Were it the case that $V\cap V^\perp\ne\{0\}$ and $V \nsubseteq V^{\perp}$, then $V \cap V^{\perp}$ is one dimensional. Let $u \in V$ such that $V=V \cap V^{\perp} \oplus \F u$. Then $B_{\G}(u,u) \neq 0$. Since $u \in C(\G)$, then $\F u$ is a non-degenerate ideal of $\G$, and $\G=(\F u)^{\perp} \oplus \F u$. Using the same argument as in {\bf Proposition \ref{proposicion reduccion}}, we have that $(\F u)^{\perp}$ is a one-dimensional central extension of $(\g,[\cdot,\cdot]_{\g},B_{\g})$, but just as we have seen, the short exact sequence associated to this central extension splits, which implies that the short exact sequence associated to the central extension $(\G,[\cdot,\cdot]_{\G},B_{\G})$, splits too.
\smallskip

\noindent The induction hypothesis consists of that the existence of an invariant metric on any central extension of an $m$-dimensional quadratic Lie algebra, with $m<n$, by a $2$-dimensional vector space, implies that the short exact sequence associated to the extension splits.
\smallskip

\noindent
Start with a 2-dimensional quadratic Lie algebra, $(\mathfrak{g},[\cdot,\cdot]_{\mathfrak{g}},B_{\mathfrak{g}})$. In this case, 
$(\mathfrak{g},[\cdot,\cdot]_{\mathfrak{g}})$ is Abelian. Let $(\mathfrak{G},[\cdot,\cdot]_{\mathfrak{G}})$ be a central extension of $\mathfrak{g}$ by $V$ admitting an invariant metric. If $V$ is isotropic then, {\bf Lemma \ref{lema 31}.(i)} implies that $\mathfrak{g}$ is 
spanned by $a_1$ and $a_2$. Using the hypotheses $D_1(a_1)=D_2(a_2)=0$ together with the $B_{\mathfrak{g}}$-skew-symmetry property of $D_1$ and $D_2$, we can show that $[a_1,a_2]_{\mathfrak{G}}=0$. Indeed,
$$
[a_1,a_2]_{\mathfrak{G}}=[a_1,a_2]_{\mathfrak{g}}+B_{\mathfrak{g}}(D_1(a_1),a_2)v_1+B_{\mathfrak{g}}(D_2(a_1),a_2)v_2
$$
$$
=B_{\mathfrak{g}}(D_2(a_1),a_2)v_2=-B_{\mathfrak{g}}(a_1,D_2(a_2))v_2=0.
$$
Since $V \subseteq C(\mathfrak{G})$, it follows that $[\cdot,\cdot]_{\mathfrak{G}}\equiv 0$, which implies that $D_1=D_2=0$.
\smallskip

\noindent Let $(\G,[\cdot,\cdot]_{\G},B_{\G})$ be a quadratic Lie algebra which in turn is a central extension of an $n$-dimensional quadratic Lie algebra $(\g,[\cdot,\cdot]_{\g},B)$, by a $2$-dimensional vector space $V$, with $n \ge 2$.

\begin{lema}
Let $\mathfrak{g}$ be $n$-dimensional. If $V$ be isotropic, then $D_1=D_2=0$.
\end{lema}
\begin{proof}
By {\bf Lemma \ref{lema 31}.(i)}, $\{a_1,a_2\}$ and $\{h(v_1),h(v_2)\}$ are both linearly independent sets.
So, $\mathfrak{a}=\operatorname{Span}_{\mathbb{F}}\{a_1,a_2\}$ is two-dimensional. 
Using the {\bf Remark \ref{3rmk}.2} that follows {\bf Lemma \ref{lema 31}}, there is a Lie algebra structure on $\mathfrak{a}^{\perp}$, such that $(\mathfrak{g},[\cdot,\cdot]_{\mathfrak{g}})$ is a central extension of $(\mathfrak{a}^{\perp},[\cdot,\cdot]_{\mathfrak{a}^{\perp}})$ 
by $\Ker T=\operatorname{Span}_{\mathbb{F}}\{h(v_1),h(v_2)\}$:
\begin{equation}\label{ex19}
\begin{array}{rcccccccl}
\,& \, & \, & \iota & \,&\pi_{\mathfrak{a}^{\perp}} &\,&\,& \\
0 & \rightarrow & \Ker T & \rightarrow & \mathfrak{g} & \rightarrow & \mathfrak{a}^{\perp} & \rightarrow & 0.
\end{array}
\end{equation}
Furthermore, the bilinear map $B_{\mathfrak{a}^{\perp}}:\mathfrak{a}^{\perp} \times
\mathfrak{a}^{\perp} \rightarrow \mathbb{F}$, defined by
$B_{\mathfrak{a}^{\perp}}(x,y)=B_{\mathfrak{g}}(Tx,y)$, for all $x,y \in \mathfrak{a}^{\perp}$, makes $(\mathfrak{a}^{\perp},[\cdot,\cdot]_{\mathfrak{a}^{\perp}},B_{\mathfrak{a}^{\perp}})$ into an $(n-2)$-dimen\-sion\-al quadratic Lie algebra.  
By the induction hypothesis, the short exact sequence \eqref{ex19} splits. Therefore,
$\mathfrak{a}^{\perp}$ is an ideal of $(\mathfrak{g},[\cdot,\cdot]_{\mathfrak{g}})$ and, consequently, so is $\mathfrak{a}$.
Since $[\mathfrak{g},\mathfrak{g}]_{\mathfrak{g}} \subseteq
\Im T$ and $\mathfrak{g}=\Im T \oplus \mathfrak{a}$ (see {\bf Lemma \ref{lema 31}.(iv)}), it follows that $\mathfrak{a} \subseteq C(\mathfrak{g})$. Hence, $D_i\circ T(x)=T \circ D_i(x)=[a_i,x]_{\mathfrak{g}}=0$ for all $x \in \mathfrak{g}$ and $i=1,2$;
that is, $D_i(\Im T)=\{0\}$ ($i=1,2$). Since $D_1(a_1)=D_2(a_2)=0$ ({\bf Lemma \ref{lema 2}.(iii)}), 
it follows that $[a_1,a_2]_{\mathfrak{G}}=0$, which implies $\mathfrak{a} \subseteq C(\mathfrak{G})$. 
Also from {\bf Lemma \ref{lema 2}.(iii)}, we finally conclude that $D_i=h \circ \ad_{\mathfrak{G}}(a_i)=0$ ($i=1,2$).
\end{proof}
\noindent We can now state the following result.

\begin{proposicion}\label{teorema c2}
Let $(\mathfrak{g},[\cdot,\cdot]_{\mathfrak{g}},B_{\mathfrak{g}})$ be a quadratic Lie algebra. 
Let $(\mathfrak{G},[\cdot,\cdot]_{\mathfrak{G}})$ be a central extension of $(\mathfrak{g},[\cdot,\cdot]_{\mathfrak{g}})$ by $V$,
and assume that $\dim_{\mathbb{F}}V\le 2$.
If there exists an invariant metric on $(\mathfrak{G},[\cdot,\cdot]_{\mathfrak{G}})$, 
then the short exact sequence associated to the extension, splits.
\end{proposicion}

\medskip
\noindent
\subsection*{Acknowledgements}

\noindent 
Part of this work was contained in the Ph.D. thesis  (2014) of the junior author.
While in graduate school he was supported by CONACYT Grant $33605$.
The final version of this manuscript was produced while he was a postdoctoral 
fellow at CIMAT-M\'erida with financial support provided by Proyecto FOMIX: YUC-$2013$-C$14$-$221183$.
He thanks his colleagues at CIMAT-M\'erida for their help, friendship and hospitality,
making his stay at CIMAT's facilities a highly rewarding professional experience.
Last, but not least, he would also like to thank Cesar Maldonado Ahumada for
reading the first draft of this paper, providing valuable suggestions that improved the
quality of the manuscript.

%\Addresses
\end{document}